\documentclass[14pt]{extarticle}
\usepackage{amssymb,amsfonts,amsthm,amsmath,bbold}
\usepackage{mathrsfs,pifont}
\usepackage{mathabx}
\usepackage[all]{xy}
\usepackage{array}
\usepackage{transparent}
\usepackage{stmaryrd}
\usepackage{tabularx}

\usepackage[top=1.2in, bottom=1.2in, left=0.5in,
right=0.5in]{geometry}

\usepackage[font={small,sl}]{caption}
\usepackage{tikz}
\usepackage{tikz-cd}
\usetikzlibrary{matrix,arrows}

\usepackage[T1]{fontenc}
\usepackage[utf8]{inputenc}
\usepackage{frcursive}

\newcommand{\cz}{\textup{\small \cursive z}}

\newcommand{\Rd}{{\mathsf{R}}^{\raisebox{0.5mm}{$\scriptscriptstyle \bullet$}}}

\usepackage{hyperref}
\usepackage[backrefs,msc-links]{amsrefs}

\newcommand{\C}{\mathbb{C}}
\newcommand{\Gm}{\mathbb{G}_\mathbf{m}}
\newcommand{\Ga}{\mathbb{G}_\mathbf{a}}
\newcommand{\Ct}{\mathbb{C}^\times}
\newcommand{\Q}{\mathbb{Q}}
\newcommand{\A}{\mathbb{A}}
\newcommand{\F}{\mathbb{F}}
\newcommand{\I}{\mathbb{I}}

\newcommand{\Z}{\mathbb{Z}}
\newcommand{\R}{\mathbb{R}}

\newcommand{\bT}{\mathsf{T}}
\newcommand{\bA}{\mathsf{A}}

\newcommand{\bG}{\mathsf{G}}
\newcommand{\bH}{\mathsf{H}}

\newcommand{\bGt}{{\widetilde{\mathsf{G}}}}

\newcommand{\sP}{\mathsf{P}}

\newcommand{\bC}{\mathsf{C}}
\newcommand{\bCh}{\widehat{\mathsf{C}}}
\newcommand{\bY}{\mathsf{Y}}

\newcommand{\bM}{\mathsf{M}}

\newcommand{\bU}{\mathsf{U}}

\newcommand{\bB}{\mathsf{B}}

\newcommand{\bX}{\mathsf{X}}
\newcommand{\bK}{\mathsf{K}}

\newcommand{\bP}{\mathbb{P}}

\newcommand{\cT}{\mathscr{T}}
\newcommand{\cL}{\mathscr{L}}

\newcommand{\cR}{\mathscr{R}}
\newcommand{\cP}{\mathscr{P}}
\newcommand{\cU}{\mathscr{U}}
\newcommand{\cQ}{\mathscr{Q}}
\newcommand{\cF}{\mathscr{F}}

\newcommand{\cV}{\mathscr{V}}

\newcommand{\cA}{\mathscr{A}}
\newcommand{\cB}{\mathscr{B}}

\newcommand{\La}{{}^{\raisebox{2pt}{\textup{\tiny\textsf{L}}}}}
\newcommand{\uA}{\underline{\bA}}
\newcommand{\sH}{\mathbf{sH}}

\newcommand{\bmu}{\boldsymbol{\mu}}

\newcommand{\bpsi}{\boldsymbol{\psi}}
\newcommand{\bPi}{\boldsymbol{\Pi}}

\newcommand{\bc}{\mathsf{c}}
\newcommand{\bbch}{\widehat{\bbc}}
\newcommand{\bbc}{\mathbf{c}}

\newcommand{\bPhi}{\boldsymbol{\Phi}}

\newcommand{\cH}{\mathscr{H}}

\newcommand{\ft}{\mathfrak{t}}

\newcommand{\fg}{\mathfrak{g}}

\newcommand{\fp}{\mathfrak{p}}

\newcommand{\fh}{\mathfrak{h}}
\newcommand{\fn}{\mathfrak{n}}
\newcommand{\fu}{\mathfrak{u}}
\newcommand{\fb}{\mathfrak{b}}
\newcommand{\fc}{\mathfrak{c}}

\newcommand{\rd}{/\!\!/\!\!/\!\!/}

\newcommand{\bS}{\mathsf{S}}
\newcommand{\bbS}{\mathbf{S}}
\newcommand{\bbSt}{\widetilde{\bbS}}

\newcommand{\eqdef}{\overset{\tiny\textup{def}}=}

\newcommand{\cO}{\mathscr{O}}
\newcommand{\Hd}{{H}^{\raisebox{0.5mm}{$\scriptscriptstyle \bullet$}}}
\newcommand{\Sd}{{\bS}^{\raisebox{0.5mm}{$\scriptscriptstyle \bullet$}}}

\newcommand{\vir}{\textup{vir}}

\newcommand{\Langl}{\textup{Langl}}

\newcommand{\fa}{\mathfrak{a}}

\newcommand{\Attr}{\mathsf{Attr}}
\newcommand{\sThom}{\mathsf{Thom}}

\newcommand{\comp}{\textup{cmp}}
\newcommand{\uni}{\textup{uni}}
\newcommand{\red}{\textup{red}}

\newcommand{\an}{\textup{an}}

\newcommand{\Tau}{\mathscr{T}}

\newcommand{\bAut}{\mathsf{Aut}}

\newcommand{\tama}{\textup{\raisebox{-1mm}{\includegraphics[height=10pt]{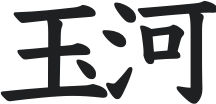}}}}

\DeclareMathOperator{\Coh}{Coh}

\DeclareMathOperator{\Hom}{Hom}
\DeclareMathOperator{\Ext}{Ext}
\DeclareMathOperator{\Ker}{Ker}

\DeclareMathOperator{\Aut}{Aut}

\DeclareMathOperator{\Lie}{Lie}
\DeclareMathOperator{\sgn}{sgn}
\DeclareMathOperator{\Res}{Res}
\DeclareMathOperator{\Ind}{Ind}
\DeclareMathOperator{\const}{const}

\DeclareMathOperator{\chr}{char}
\DeclareMathOperator{\pt}{pt}

\DeclareMathOperator{\rk}{rk}
\DeclareMathOperator{\Pic}{Pic}
\DeclareMathOperator{\tr}{tr}

\DeclareMathOperator{\Spec}{Spec}
\DeclareMathOperator{\Span}{span}

\DeclareMathOperator{\Fix}{Fix}

\DeclareMathOperator{\spec}{Spec}

\DeclareMathOperator{\supp}{supp}

\DeclareMathOperator{\Thom}{Thom}

\DeclareMathOperator{\codim}{codim}

\DeclareMathOperator{\ad}{ad}

\DeclareMathOperator{\hht}{ht}
\DeclareMathOperator{\Slice}{Slice}
\DeclareMathOperator{\DD}{\mathbf{D}}
\DeclareMathOperator{\DDa}{\mathbf{D}_a}

\DeclareMathOperator{\vol}{vol}

\newcommand{\Eis}{\textup{Eis}}

\newtheorem{Theorem}{Theorem}

\newtheorem{Lemma}{Lemma}[section]
\newtheorem{Proposition}[Lemma]{Proposition}
\newtheorem{Corollary}[Lemma]{Corollary}

\theoremstyle{definition}

\newtheorem{Remark}{Remark}[section]

\begin{document}

\title{On the unramified Eisenstein spectrum} 
\author{David Kazhdan and Andrei Okounkov}
\maketitle

\abstract{For a split reductive group $\La \bG$ over a global field,
  we determine the spectrum of the spherical Hecke algebra coming from
  the unramified Eisenstein series for the minimal parabolic $\La
  \bB$. This is done using a certain decomposition of the Springer
  stack $T^*(\bB\backslash \bG/\bB)$ for the Langlands dual group
  $\bG$ in the additive cobordism group of cohomologically proper
  derived quotient stacks.}

\setcounter{tocdepth}{2}
\tableofcontents

\section{Introduction}

\subsection{Introduction to introduction}

\subsubsection{}

Let $\F$ be a global field, that is, a finite extension of the field
$\Q$ of rational
numbers or a finite extension of the field $\Bbbk(t)$  of rational functions in one variable with coefficients in a
finite field $\Bbbk$. We denote by $\A$ the ring of adeles for
$\F$. For a reductive $\F$-group $\bG$, the quotient
$\bG(\A)/\bG(\F)$ has a natural measure which yields a unitary representation $\rho$ of $\bG(\A)$  in $L^2( \bG(\A)/
\bG(\F))$. The study of  automorphic forms, that is the decomposition
of the representation $\rho$, has been at the center of the progress
of Number Theory in the last 50 years.

As explained by R.~Langlands in his fundamental book \cite{L}, the study of
automorphic forms for $\bG$ naturally breaks up into two parts.
The first part is the study of cuspidal automorphic forms $\Phi$ for
Levy subgroups $\bM \subset \bG$. From
every such form $\bPhi$ on $\bM\ne \bG$, one can construct
a family of automorphic forms for $\bG$, known as the Eisenstein
series, see \eqref{eq:97bis}. The spectral analysis of the span of
these 
Eisenstein series constitutes the second part of the strategy
from \cite{L}. 

Our paper deals with the simplest, but already highly non-trivial case
of this second part in the case when 
\begin{itemize}
\item the group $\bG$ is split, 
\item $\bM=\bT$ is the  torus,  
\item $\bPhi = 1$, and 
\item we restrict our attention to the subspace 
of $K$-invariant vectors, where $K\subset \bG(\A)$ is the standard
maximal compact subgroup. \end{itemize}

\subsubsection{}

While the real complexities of the problem can be appreciated only for
groups $\bG$ with an interesting root system, certain basic features
of the problem, together with the different approaches to their
analysis, may already be illustrated in the most basic example
when $\bG=SL(2)$ and $\F=\Q$. In this case
\begin{align}
  \label{eq:182}
 K \backslash \bG(\bA) / \bG(\F) & = \left\{ \textup{lattices
                                   $\Gamma\subset \R^2$, 
                                   $\vol(\R^2/\Gamma) =1$} \right\}
                                   /\, \textup{isometry}\,, \\
&= \left\{ \textup{upper half-plane} \right\} / SL(2,\Z) \,.    \notag                        
\end{align}
In \eqref{eq:182}, we associate the lattice
\[
\Gamma = \Z \tfrac{1}{\sqrt{y}}  \oplus \Z \tfrac{\tau}{\sqrt{y}}
\]
to a point $\tau=x+iy$, $y>0$, in the upper half plane. 
One is interested in the spectrum of the Laplace operator
$$
\Delta  = - y^2 \left( \frac{\partial^2}{\partial x^2}
  +\frac{\partial^2}{\partial y^2}
  \right) \,, 
$$
and commuting 
Hecke operators acting in $L^2$-functions on $ K \backslash \bG(\bA) /
\bG(\F)$.

Note that the volume of \eqref{eq:182} is finite, hence the constant
function is a trivial eigenfunction of $\Delta$. Other $L^2$-eigenfunctions of
$\Delta$ are more mysterious. They are known as the Maa\ss\ forms and their
study is the study of cuspidal automorphic forms for the whole group
$\bG$.

There is only one nontrivial Levy subgroup in $\bG$, namely the group $\bM \cong GL(1)$ of
diagonal matrices. Starting
from the trivial representation of $\bM$, one constructs the following
classical Eisenstein series
\begin{equation}
E(\Gamma,s) = \tfrac12 \sum_{\textup{primitive $\gamma \in \Gamma$}}
\|\gamma\|^{-s-1} \,.\label{eq:171}
\end{equation}
This converges for $\Re s>1$ and has a meromorphic analytic
continuation in $s$. The point $s=1$ is a simple pole for this
analytic continuation with residue the constant function
$\frac{6}{\pi}$.


The sum \eqref{eq:171} may be interpreted as
the count of 
$\F$-rational points of $\bP^1 = \bG/\bB$ according the height induced
by the metric on $\Gamma$. Here $\bB$ is the parabolic subgroup
corresponding to $\bM$, that is, the group of upper-triangular
matrices. Parallel interpretation exists for general Eisenstein
series.

It is easy to see that
\begin{equation}
  \label{eq:184}
  \Delta E(s) = \left(\frac{1}{4} - \frac{s^2}{4} \right) E(s)
  \,, 
\end{equation}
but it is also easy to see that $E(s)
\notin L^2$ for any value of $s$. Still, the Eisenstein sereis 
contribute to the spectrum of $\Delta$, and this
contribution can be analysed as follows.

\subsubsection{}


Consider linear combinations of the Eisenstein series of the following
form
\begin{equation}
  \label{eq:185}
  E_{f^\vee} (\Gamma) = \tfrac{1}{4\pi}
  \int_{\Re s = s_0 \gg 0} f(s) \, 
  E(\Gamma,s) \, ds = \tfrac12 \sum_{\gamma}
f^\vee(\|\gamma\|) \,,
\end{equation}
where the functions $f$ and $f^\vee$ are related by the Mellin
transform in the variable $s$ and $ds$ is the standard Lebesgue
measure on the vertical line. We can choose $f^\vee$ to be
with compact support in $\R_{>0}$ and infinitely
differentiable. This makes $f(s)$ a Paley-Wiener function, in the
sense that $f(s)$ is an entire function that grows at most
exponentially as $\Re s \to \infty$ and decays faster than any
polynomial as $\Im s \to \infty$. The commutative diagram
\begin{equation}
  \label{eq:186}
  \xymatrix{\textup{PW-functions $f(s)$} \ar[d]_{\frac{1}{4} - \frac{s^2}{4}}\ar[rr]^{\qquad E_{f^\vee}} && L^2\ar[d]^{\Delta}\\
  \textup{PW-functions $f(s)$} \ar[rr]^{\qquad E_{f^\vee}}&& L^2
  } 
\end{equation}
conjugates the operator $\Delta$ to 
multiplication by a function.
We denote by $L^2_\Eis$ the closure of the image in \eqref{eq:186}.
It is the orthogonal complement of the span of the cuspidal
eigenfunction of $\Delta$.

While \eqref{eq:186} may at first look like a spectral theorem for
$\Delta$ in $L^2_\Eis$, the pullback of the $L^2$-norm via the map $E_f$
does \emph{not} have the form required in the
spectral theorem. For this pullback, there is a beautiful
general formula \eqref{eq:100} going back to R.~Langlands.
In this particular instance, it 
takes the form  
\begin{equation}
\| E_{f^\vee}\|^2 = \tfrac{1}{4\pi} \int_{\Re s = s_0 \gg 0} \left(f(s) \bar f(-s) +
f(s) \bar f(s) \frac{\xi(s)}{\xi(-s)} \right) \, ds \,,  \label{eq:187}
\end{equation}
where $\bar f(s) = \overline{f(\bar s)}$ and 
\begin{equation}
  \label{eq:188}
  \xi(s) = \pi^{-s/2} \Gamma(s/2) \zeta(s) \,, \quad \xi(1-s) = \xi(s)
  \,,
\end{equation}
is the completed $\zeta$-function of $\Q$.



One can draw a precise analogy
between the coefficients in \eqref{eq:187} and scattering for
$(1+1)$-dimensional
Schr\"odinger equation. In particular, the pole at $s=1$ in the
coefficient 
$\frac{\xi(s)}{\xi(-s)}$ correspond to the
constant eigenfunction of $\Delta$. See e.g.\  \cite{Takh} for a detailed
recent survey. Similarly, the products of $\xi$-functions that appear in
\eqref{eq:100} are analogous to S-matrices in integrable
$(1+1)$-dimensional many-body systems. 

The classical strategy for turning \eqref{eq:187} into a standard
$L^2$-norm required in the spectral theorem is to move
the contour of integration, picking up the residues. This
approach is also classical in quantum mechanics, see \cite{Takh}. We get
\begin{equation}
  \label{eq:182_}
  \| E_{f^\vee}\|^2 = \tfrac{3}{\pi} \, |f(1)|^2 + \tfrac{1}{2\pi}
  \int_{i\R} 
  |\cP_{\mathfrak{sl}_2} f(s)|^2 \, ds \,, 
\end{equation}
where the spectral projector
\begin{equation}
  \label{eq:190}
  \cP_{\mathfrak{sl}_2} f(s) = \frac12 \left( f(s) +
    \frac{\xi(-s)}{\xi(s)} f(-s)
    \right)
  \end{equation}
reflects the functional equation $E(-s) =  \frac{\xi(-s)}{\xi(s)}
E(s)$. Via \eqref{eq:184}, the two terms in 
\eqref{eq:182_} correspond to the zero eigenvalue and
the continuous spectrum filling $[\frac14,\infty)$ with
multiplicity one. 

For groups of higher rank, the analysis of residues in multivariate
integrals meets with considerable problems, which one may describe as
either combinatorial or conceptual, depending on one's point of view.
For instance, one can open pages 181--195 of Langlands'
monumental treatise \cite{L} to see his analysis of the 
integral for the group $G_2$. One could also recommend the
exposition in \cites{Labesse,MW} and the state-of-the-art
discussion in \cite{DHO2}. 

\subsubsection{}

In this paper, we use entirely different 
\emph{topological} ideas for the analysis of formulas like \eqref{eq:187}.
To get started, let us rewrite it as follows
\begin{equation}
  \label{eq:191}
  ( E_{f_1^\vee}, E_{{f^*_2}^\vee})_{L^2}= \underbrace{\tfrac{1}{4\pi} \int_{\Re s
    = s_0 \gg 0} f_1(s) ds\phantom{\Bigg|}}_{-T^*\La\bA}\, 
\underbrace{\,\,\frac{1}{\xi(-s)} \phantom{\Bigg|} }_{-T^*\La\bU}
\underbrace{ \,\left(f_2(-s) \xi(s) + 
f_2(s) \xi(-s) \right)\phantom{\Bigg|}}_{T^* (\La \bG/\La \bB)} \,, 
\end{equation}
where $f^*(s) = \bar f(-s)$. 
The three parts of this formula have the following topological
meaning.

The innermost term in \eqref{eq:191} is the equivariant
localization formula for a certain integral in
$(\La\bA\times \Ct_q)$-equivariant cohomology of 
\begin{equation}
T^*\bP^1 = T^* \big(\La \bG/\La \bB\big)\,, \quad \La \bG  = PGL(2,\C)
\,. \label{eq:196}
\end{equation}
The $\Ct_q$-factor acts by scaling the cotangent fibers with weight
$-1$. We will evaluate all formulas at the point $1\in \Lie \Ct_q$,
which is a special point that
replaces in the number field situation the cardinality of the finite
field $\Bbbk$. 
We denote by $0 \in \bP^1$ the origin of $\La \bG/\La \bB$, and
by $\infty \in \bP^1$ the other $\La\bA$-fixed point, corresponding
to the opposite Borel subgroup $\La\bB_{-}$.

The variable $s$ is the positive root of $PGL(2)$, which
means that tangent space $T_0\bP^1$  has weight $-s$.
Taking into account the $\Ct_q$-action, the weight of the cotangent
direction at 0 is thus $-1+s$. At $\infty$, we get
the weights $s,-1-s$, respectively. 

Given a function $\psi(s)$, we can define a genus $\psi(M)$ of any
manifold $M$ by
multiplying the values of $\psi$ over Chern roots of the tangent
bundle of $M$. Concretely, $\psi(T^*\bP^1)$ is an equivariant cohomology class whose
restriction to $0$ and $\infty$ equals $\psi(-s) \psi(-1+s)$ and
$\psi(s) \psi(-1-s)$, respectively. By equivariant localization,
\begin{equation}
  \label{eq:192}
  \int_{T^*\bP^1} f_2(c_1)  \psi(T^*\bP^1)  =
  \frac{f_2(-s) \psi(-s) \psi(-1+s)}{s(1-s)}  \, + \, \left(\textup{same
    with $s\mapsto -s$} \right) \,, 
\end{equation}
where $c_1$ is the 1st Chern class of the tangent bundle to $\bP^1$. 
Since $\xi(s)$ has poles at $s=0,1$ and satisfies $\xi(s) = \xi(1-s)$,
one can find an entire function $\psi$ such that \eqref{eq:192}
is equal to the innermost sum in \eqref{eq:191}. In \eqref{eq:192} we
have the integral of an \emph{analytic} cohomology class. These
are obtained by extension of scalars from equivariant cohomology of a
point, and therefore obey the usual rules of integration. 

The rest of the RHE in \eqref{eq:191} has to do with taking the
symplectic reduction by the opposite Borel subgroup $\La \bB_{-}$.
Referring to the bulk of the paper for details, we mention that this
can be done in stages. First, the reduction by unipotent radial
$\La \bU_{-}$ amounts to a combination of a Koszul and
Chevalley-Eilenberg complexes and leads to division by
$$
\xi(-s) = \int_{T^* \Lie \La \bU_{-}} \psi(T^* \Lie \La \bU_{-}) \,. 
$$
This subtracts the bundle $T^*\La\bU_{-}$ from the tangent bundle of
$T^*\bP^1$, whence the notation in \eqref{eq:191}. 

Finally, the symplectic reduction by the torus $\La\bA$ leads to
integration along a translate of the Lie algebra $i\R$ of the compact
torus. This is related to how K-theory integrals over quotients by a
reductive groups are related to invariants,
and hence to Riemann integrals over 
a maximal compact subgroup. See
Section \ref{s_geom} for details. Before we take the quotient by
$\La\bA$, we are free to multiply by any class in $\Hd_{\La\bA}(\pt)$,
that is, by any function of $s$. This is the meaning of $f_1(s)$ in
\eqref{eq:191}.

In the end, \eqref{eq:191} is a certain characteristic
class of the stack 
\begin{equation}
\cT = T^*(\La \bB \backslash \La \bG / \La \bB) \,.\label{eq:193}
\end{equation}
The gain in casting \eqref{eq:191} into a topological language is that
it shifts the focus from the particular functions and
contours of integration entering the formula to the geometry of the space
\eqref{eq:193} itself. In particular, the classical map
\begin{equation}
T^*(\La \bG / \La \bB)  \to \textup{nilpotent elements in $\Lie \La
  \bG$}\label{eq:195}
\end{equation}
induces the map of stacks 
\begin{equation}
  \label{eq:194}
  \cT  \to \left. \left\{\textup{nilpotent elements in $\Lie \La
  \bG$}\right\}\right/ \La
  \bG \,. 
\end{equation}
Any characteristic class in the source in \eqref{eq:194} can be pushed
forward to the target in \eqref{eq:194}. Since there are only two nilpotent
conjugacy classes in $\mathfrak{sl}(2)$, this means that the
integral of any characteristic class of $\cT$ is a sum of two terms.
Not surprisingly, for \eqref{eq:191} these two terms are exactly
the two terms in \eqref{eq:182_}. 

While this result is compatible with the one obtained by residue
calculus, it is logically fully independent of any contour
integral computations. One can compare and contrast the two approaches
by looking at the $G_2$-example, which we consider from our points of
view in Section \ref{sectG2}. The difference between the two
approaches becomes very noticeable for
more complicated groups. The map \eqref{eq:195} exists in
full generality and it is known as the Springer resolution. Using the 
induced map \eqref{eq:194} one can write any characteristic class of
$\cT$, in particular, the Langlands' formula \eqref{eq:100} as a
sum over nilpotent conjugacy classes in $\Lie \La \bG$.
As we show in this paper, the spectral decomposition of 
$L^2_\Eis$ for general $\bG$ is parametrized by
precisely this data.

\subsection{The automorphic form setup}

\subsubsection{}

While the main idea explored in this paper may be generalized in a multitude
of directions, some of its key features are already very interesting in
the following most basic situation. Further
generalizations are discussed
in \cites{KO,KO2}, see also Section \ref{s_future} below. 

\subsubsection{}

We denote by $\F$, $\A$, and $\I$ a global field, its adeles, and
ideles, respectively. We consider a split reductive group over $\F$,
which we denote $\La\bG$ because the bulk of our computations will be
with the Langlands dual complex reductive group $\bG$. On both sides,
we denote by the letters $\bA$, $\bB$, and $\bU$
(resp. $\La\bA,\La\bB,\La\bU$) a maximal torus, a
Borel subgroup containing it, and its unipotent radical. We have
\[
\bA = \Lambda \otimes_\Z \Ct \,,
\quad
\La \bA = \Lambda^\vee \!\otimes_\Z \Gm \,, 
\]
for pair of dual lattices $\Lambda$ and $\Lambda^\vee$. 

\subsubsection{}

Let $\cQ_\F \subset \R_{>0}$ denote the image of
the norm map $\I \to \R_{>0}$. We have
\begin{equation}
  \cQ_\F =
\begin{cases}
  q^\Z \,, &\F = \Bbbk(\textup{curve})\,, |\Bbbk|=q\,,\\
  \R_{>0}\,, & \chr \F = 0 \,. 
\end{cases}\label{eq:182-7}
\end{equation}
For uniformity in formulas, one should interpret $q>1$ as the generator
of this group, including the infinitesimal generator $t\frac{d}{dt} \in \Lie \R_{>0}$,  in the number field case.

We have the dual map 
\[
  \C \cong\Hom(\R_{>0},\Ct)  \to \Hom(\cQ_\F,\Ct)
\overset{\textup{\tiny def}}=\cQ_\F^\vee \,,
\]
where the first isomorphism sends $t\in \R_{>0}$ to
$t^s$, $s\in \C$. In other words,
\begin{equation}
\cQ_\F^\vee =
\begin{cases}
  \C\left/\frac{2\pi i}{\ln q} \, \Z \right. \,, &\chr \F >0 \,,\\
  \C\,, & \chr \F = 0 \,. 
\end{cases}\label{eq:116}
\end{equation}
In both cases we have a well-defined map
\[
  \Re: \cQ_\F^\vee \to \R 
\]
which takes the real part.

\subsubsection{}\label{s_Eis_def}

The characters of $\La\bA(\A)$ that factor through the norm map
\begin{equation}
  \label{eq:96}
  \xymatrix{
    \La\bA(\A) \ar[rr]^{\|a\|^\lambda} \ar[dr]^{\|a\|}&& \Ct \\
    & \Lambda^\vee \otimes \cQ_\F
    \ar[ru] 
    }
\end{equation}
are parametrized by
\[
  \lambda \in \Lambda \otimes \cQ_\F^\vee \cong
\begin{cases}
  \bA \,, &\chr \F >0 \,,\\
  \Lie \bA\,, & \chr \F = 0 \,. 
\end{cases}
\]
Given such $\lambda$, one defines the Eisenstein series
\begin{equation}
E(\lambda,g) = \sum_{\gamma \in \La\bG(\F)/\La\bB(\F)}
\| a(g \gamma)\|^{\lambda+\rho} \,,\label{eq:97}
\end{equation}
where $a(g)\in \La\bA$ is the torus part in the Iwasawa
$\La\mathsf{K} \La\bA \La\bU$ decomposition of $g$, and $\rho\in \Lambda$ is the half sum
of the roots in $\La\bB$. The series \eqref{eq:97}
converges for $\Re \lambda$ inside the positive cone of
$\Lambda \otimes \R$, which we express by the
notation $\Re \lambda > 0$. 

\subsubsection{}
More generally, given a
function
\[
  \phi^\vee: \Lambda^\vee \otimes \cQ_\F \to \C
\]
with compact support, one may consider the corresponding
pseudo-Eisenstein series
\begin{equation}
E_{\phi^\vee}(g) = \sum_{\gamma \in \La\bG(\F)/\La\bB(\F)}
\phi^\vee(\| a(g \gamma)\|) \,.\label{eq:98}
\end{equation}
The series \eqref{eq:97} and \eqref{eq:98} are related by the Mellin
transform on the group $\Lambda^\vee \otimes \cQ_\F$.
If $\phi$ is a function on $\Lambda \otimes \cQ_\F^\vee$ which is
the Mellin transform of $\phi^\vee$ then
\[
  E_{\phi^\vee}(g) = \int_{\Re \lambda = \lambda_0 \gg 0}
E(\lambda,g) \phi(\lambda) d \lambda  \,.
\]

\subsubsection{}

The following fundamental result is due to R.~Langlands

\begin{Theorem}[\cite{L}] \label{t_Langl} 
  The Hermitian product
  \begin{equation}
    \label{eq:99}
    (f_1,f_2)_\Eis \overset{\textup{\tiny def}}=
    (E_{f_1^\vee}, E_{f_1^\vee})_{L^2\left(\La\bG(\A)\,/\,\La\bG(\F)\right)}
  \end{equation}
  is given by the following
  formula
  \begin{equation}
    \label{eq:100}
   (f_1,f_2)_\Eis =
  \int_{\Re \lambda = \lambda_0 \gg 0}
  d\lambda \sum_{w\in W} f_1(\lambda) \, \overline{f_2}(- w
    \cdot \lambda) \!\!
  \prod_{\substack{\alpha > 0 \\ w\cdot \alpha <0 }}
  \frac{\xi_\F(\alpha(\lambda))}{\xi_\F(1+\alpha(\lambda))} \,,
\end{equation}
where $\overline{f}(x) = \overline{f(\overline{x})}$,
$\{\alpha\}$ are the roots of $\bG$, and
$\xi(x)$ is the completed $\zeta$-function of $\F$.
\end{Theorem}

\noindent
The summation over $W$ in \eqref{eq:100} 
comes from the Bruhat decomposition of the group
$\La\bG(\F)$ and 
the ratios of $\zeta$-functions appear as the
spherical matrix elements of intertwining operators
associated to $w\in W$.
See Appendix \ref{s_pT3} for a discussion of the normalization of the
Lebesgue measure in \eqref{eq:100}.


\subsubsection{}

Our object of interest in this paper is
\begin{equation}
  \label{eq:117}
  \cH_\Eis = \overline{\Span\left\{E_{\phi^\vee}(g)\right\}} \subset
  L^2\!\left(\,\La\bG(\A)\,\big/\,\La\bG(\F)\right) \,. 
\end{equation}
This is a Hilbert space with an action of a large commutative $*$-algebra
generated by the
spherical Hecke operators at all nonarchimedian places (and also
invariant differential operators at Archimedean places). Since
Eisenstein series are joint eigenfunctions of all these operators, we
may realize a dense (in the sense of having the same commutant)
subalgebra $\sH$ of this algebra concretely as
\begin{align}
  \mathbf{sH}  &= \textup{$W$-invariant functions of $\lambda$} \notag \\
                     & = \C \left[ \uA\,/\,W\right] \,, \label{eq:118}
\end{align}
where
\begin{equation}
  \label{eq:110}
  \uA =  \Lambda \otimes \cQ_\F^\vee \cong
\begin{cases}
  \bA \,, &\chr \F >0 \,,\\
  \Lie \bA\,, & \chr \F = 0 \,. 
\end{cases}
\end{equation}
The problem that we solve in this paper is to determine the spectrum
of $\sH$ in $\cH_\Eis$.

\subsubsection{}\label{s_sl2}
As envisioned by Langlands, $\cH_\Eis$  breaks up into a direct sum, the pieces in which are indexed by the conjugacy
classes of nilpotent elements $e\in \fg = \Lie \bG$. Here $\bG$ is the
Langlands dual group.

Recall that there are only finitely many conjugacy classes of
nilpotent elements in $\fg$. Further, 
by the Jacobson-Morozov theorem, every $e$ has the form
\begin{equation}
e = \phi \left(
  \begin{bmatrix}
    0 & 1 \\ 0 & 0 
  \end{bmatrix}
\right)\label{eq:119}
\end{equation}
for a homomorphism $\phi: \mathfrak{sl}_2 \to \fg$. The element
\begin{equation}
h = \phi \left(
  \begin{bmatrix}
    1 & 0 \\ 0 & -1
  \end{bmatrix}
\right)\,, \label{eq:124}
\end{equation}
known as the \emph{characteristic} of $e$, satisfies
\[
  [h,e]=2e\,.
\]
Note this includes the case
\[
  (e,h) = (0,0)
\]
which corresponds to $\phi=0$. 
The nilpotent element $e$
and its characteristic $h$ determine each other up to the action of
the respective centralizers. 

We denote by
\[
  \bC_e \supset \bC_\phi 
\]
the centralizers of $e$ and of the pair $\{e,h\}$. The group $\bC_\phi$ is
reductive and is a maximal reductive subgroup of $\bC_e$. We denote by
$\bC_{\phi,\comp} \subset \bC_\phi$ a maximal compact subgroup. It is
also a maximal compact subgroup of $\bC_e$.

\subsubsection{}
Given $\phi$, we define
\begin{equation}
  \label{eq:122}
  \bbSt_e =
\begin{cases}
  q^{h/2} \bC_{\phi,\comp}\,, & \chr \F  > 0 \,, \\
  \frac12 h + \Lie \bC_{\phi,\comp} \,, & \chr \F  =  0 \,. 
\end{cases}
\end{equation}
These subsets of $\bG$, resp.\ $\Lie \bG$, are determined by $e$ 
up-to conjugation. 
Since
\[
  \uA \, / \, W =
\textup{conjugacy classes of semisimple elements in}
\begin{cases}
 \bG\,, & \chr \F  > 0 \,, \\
  \Lie \bG \,, & \chr \F  =  0 \,, 
\end{cases}
\]
we have a well-defined real-algebraic subvariety 
\begin{equation}
  \label{eq:123}
  \bbS_e = \textup{Image} \left(\bbSt_e  \to \uA \, / \, W \right)
  \,. 
\end{equation}
We equip $\bbS_e$ with the push-forward of the Haar measure
$d_\textup{Haar}$ from $\bbSt_e$, which we denote by the same symbol. 

\subsubsection{}
The following is our main result. It represents the next step in a very long development of the
subject, see  in particular \cites{MW,Labesse} for a discussion of
prior results. 

\begin{Theorem}\label{t_spectral} 
  We have the decomposition
  \begin{equation}
    \cH_\Eis = \bigoplus_{\textup{nilpotents $e\in \Lie \bG$}} L^2\left( \bbS_e, \bmu_e\right) \,,
    \label{e_spectral}
    \end{equation}
    consistent with the inclusion $\bS_e \subset \spec \sH$, 
    where $\bmu_e$ is a measure with a density with
    respect to $d_\textup{Haar}$. 
\end{Theorem}

\noindent
We also give an explicit formula for spectral projectors and for the
measure $\bmu_e$, see Theorem \ref{t_L2}. The proof of Theorem \ref{t_spectral}  is
completed in Section \ref{s_proof}.

In the special case when $\F =\Bbbk(t)$ is the field of rational functions
with coefficients in a finite field $\Bbbk$,
the results of our paper follow from the work of Prasad
\cite{Prasad}. The focus of Prasad's work are representations that
have Iwahori ramifications at $t=0,\infty$. The everywhere unramified
representations considered in this paper constitute a very simple special
case of those.

\subsection{Large-scale structure of the argument}

\subsubsection{}

In his fundamental work \cite{L}, R.~Langlands introduced a contour
shift and the analysis of residues as a way to describe the discrete
spectrum of $\cH _{\Eis}$ as an $\sH$-module. In \cite{L}, he
computed the discrete spectrum for groups for groups of rank 2.
In the work \cites{Moe1,Moe2} of C.~M\oe glin and J.-L.~Waldspurger this approach
was used to describe the discrete  spectrum of $\cH _{\Eis}$ for all 
classical group. In recent work  of Volker Heiermann, Marcelo de Martino and Eric Opdam \cite{DHO} 
this approach was extended to general split groups over number fields (computer 
aided in the exceptional cases). The most recent progress in that
direction is reflected in \cite{DHO2}.


In this paper, we do something rather different, and our starting point
is the following geometric interpretation of the pairing 
\eqref{eq:100}. In an equivalent and slightly more convenient
language, we will start with a geometric interpretation of 
the linear functional, or distribution,  $\Psi_\Eis$
defined by 
\begin{equation}
  \label{eq:81}
   (f_1,f_2)_\Eis = (f_1 \boxtimes f_2^*, 
   \Psi_\Eis)  \,. 
\end{equation}
Here
\[
  f^* (\lambda) = \overline{f} (-\lambda) \,, \quad \lambda\in \uA
\,,
\]
where minus refers to the group structure of $\uA$. See Section
\ref{s_distr} for a discussion of what we mean by distributions in
this paper. 

\subsubsection{}

Recall the group $\cQ_\F^\vee$ defined in \eqref{eq:116} 
and note this is a 1-dimensional connected affine
algebraic group over
$\C$. To such an object one can associate an equivariant complex-oriented
cohomology theory. This is a contravariant
functor, which we denote by
\begin{equation}
  \label{eq:112}
  \begin{matrix}
  \textup{topological spaces} \\
  \textup{with an action of} \\
  \textup{a compact group}
\end{matrix} \quad 
\xrightarrow{\quad \Hd_{\textup{group}}\left(\textup{space}\big|  \,
    \cQ_\F^\vee \right)\quad} \quad 
\begin{matrix}
  \textup{supercommutative} \\
  \textup{algebras over $\C$} 
\end{matrix} \quad. 
\end{equation}
Concretely, this is ordinary equivariant cohomology if $\chr \F=0$ and
equivariant 
topological K-theory if $\chr \F > 0$. 

In particular, we have 
\begin{equation}
  \label{eq:111}
  \uA = \Spec \Hd_{\bA_\comp} \left(\pt\big| \, \cQ_\F^\vee\right) 
  \,, 
\end{equation}
where $\uA$ was defined in \eqref{eq:110} 
and $\bA_\comp \subset \bA$ is the compact torus.

\subsubsection{}\label{s_limit}

While there is a fundamental unity between equivariant K-theory and
equivariant cohomology, there are also differences in details. 
To avoid considering cases all the time and to
streamline the exposition,  we will focus on
the more difficult case of
K-theory, the group law for which we write multiplicatively.
Equivariant cohomology we treat as the $q\downarrow 1$
limit, in which we substitute
\[
  x_{\textup{old}} = q^{x_{\textup{new}}}
  \]
  for all other variables, see Section \ref{s_limitH} for details. 
In practice, this amounts to the following
  instruction: 
\begin{enumerate}
\item[(1)] replace the multiplicative group law by the additive group
  law, 
\item[(2)] replace $1$ by $0$, $q$ by $1$, and $1-x^{-1}$ by $x$,
\item[(3)] define $\psi_{\textup{new}}(s) = \lim_{q \to 1}
  \psi_{\textup{old}}(q^s;q)$, where $q^s$ is the argument of $\psi_{\textup{old}}$
  and $q$ is a parameter on which $\psi_{\textup{old}}$ is allowed to
  depend. 
\end{enumerate}
see Table \ref{tab1} in Section \ref{s_limitH} for more on this.

\subsubsection{}
The cohomology theory \eqref{eq:112} has pushforwards for
proper equivariant complex oriented maps
\[
  f: \bX \to \bY \,.
\]
For instance, if $f$ is a holomorphic map between complex manifolds, then
$f$ is automatically complex oriented. As long as $f$ is proper,
this setup can be extended  without difficulty
to suitable maps between complex algebraic stacks,
as we recall in Section \ref{s_geom} below. 

In particular, if $f$ is a proper $\bA$-equivariant map and 
 $\bY$ is a point then $f_*$ takes values in
 polynomial function on $\uA$. Remarkably, if one weakens the
 properness assumptions to \emph{cohomological properness}, which is a
notion recalled from \cite{DHLloc} in
 Appendix \ref{a_proper}, then
 $f_*$ is still defined with values in distributions on $\uA$. This is
 completely analogous to the fact that the character of an
 infinite-dimensional representation, when defined, is typically a
 distribution on the group rather than a function.

 The geometric interpretation of the
 distribution $\Psi_\Eis$ will be in these terms. Namely, the
 distribution $\Psi_\Eis$ on $\uA \times \uA$ 
 will be interpreted as the
 pushforward under the diagonal map in 
 %
 %
 %
 \begin{equation}
   \xymatrix{
     \textup{Springer stack } \cT \ar[rrd] \ar[d] \\
    \pt\big/\left(\bA \times \bA
 \right)\ar[rr] && \pt}\label{eq:82}
 \end{equation}
 of a certain equivariant genus $\psi$ related to the completed
 $\zeta$-function $\xi_\F(s)$ of the field $\F$. See \eqref{eq:86}
below  for the definition of $\cT$. 

 \subsubsection{} 

As we recall in Section \ref{s_geom} below, 
genera are usually considered for either smooth compact stably complex
manifolds or smooth proper algebraic varieties and involve the choice of an
arbitrary function $\psi(x)\in \C\left[\cQ_\F^\vee\right]$, such that
\begin{equation}
\psi( \textup{origin} ) =1 \,.\label{eq:113}
\end{equation}
If $x$ is the
equivariant Chern class of a line bundle, then the origin $x=0$ is the
Chern class of a trivial line bundle. 

For manifolds with an
action of a group $\bA$, genera give ring homomorphisms
\begin{equation}
  \label{eq:189}
  \left\{
    \begin{matrix}
      \textup{$\bA$-equivariant} \\
      \textup{bordism classes $[\bX]$}
    \end{matrix}
  \right\}
  \xrightarrow{\quad \psi(\bX) \quad}
       \C[\uA] \,,
\end{equation}
given by 
\[
  \psi(\bX) = (\bX \to \pt)_* \, \psi(T\bX) \,. 
\]
Here 
\[
  \psi( \textup{a vector bundle $V$}) = \prod_{\textup{Chern roots $v_i$}}
\psi(v_i) \in \Hd_\bA\left( \bX \big|  \,
    \cQ_\F^\vee \right) \,. 
\]
This may by
extended by continuity to much larger classes of functions
$\psi(x)$, see Section \ref{s_K_an}. Also, the condition \eqref{eq:113} may be relaxed if we
work with manifolds of a given dimension, or stacks of given virtual
dimension.

\subsubsection{}

We define the Springer stack as
\begin{equation}
\cT = T^*( \bB \backslash \bG /\bB)\,, \label{eq:86}
\end{equation}
see Section \ref{s_stack}, 
and give it an action of the group $\Ct$ scaling cotangent directions
with weight $-1$. The element of this multiplicative group $\Ct$ will
be denoted by $q$. We will eventually evaluate it to a prime power in the function
field case. In the number field case, we will take $1\in \Lie \Ct$ as
a distinguished element, as already mentioned in Section
\ref{s_limit}.

Since $\cT$ is presented as a quotient by $\bB \times \bB$, it
comes equipped with the vertical map in \eqref{eq:82}. 

\subsubsection{}

The starting point of our analysis is the following reformulation of
\eqref{eq:100}. 

\begin{Theorem}\label{t_T3} 
  The stack $\cT$ is cohomologically proper and 
  \begin{equation}
  \Psi_\Eis =\psi_\F(\cT) \,, \quad \cT = T^*( \bB \backslash \bG
  /\bB)\,, 
\label{eq:68}
\end{equation}
  where the function $\psi_\F(x)$ satisfies the equation
  \begin{equation}
  \xi_\F(s) = \begin{cases}
  \dfrac{\psi_\F(q^{-s}) \psi_\F(q^{-1+s})}{(1-q^s)(1-q^{1-s})} \,,
  &\cQ_\F = q^\Z \,, \vspace{3mm}\\ 
 \dfrac{\psi_\F(-s) \psi_\F(-1+s)}{s(1-s)} \,, & \chr \F = 0 \,. 
\end{cases}\label{eq:108}
\end{equation}
\end{Theorem}

\noindent
Note that since the zeros and poles of $\xi_\F(s)$ are symmetric with
respect to $s \mapsto 1-s$, a function $\psi_\F(x)$ solving these
equations can always be found by the Weierstra\ss\ factorization
theorem.

For instance, in the function field case, we have
\begin{equation}
  \zeta_\F(x) = \frac{\prod_{i=1}^{g} (1- \alpha_i/ x) (1-
    \alpha_i^{-1} q / x)}
  {(1-1/x)(1-q/x)}\,,
\label{eq:159}
\end{equation}
where we use the variable $x=q^{s}$ instead of the more conventional
$t=q^{-s}$, and where $\{\alpha_i\}$ are half of the Frobenius
eigenvalues, $|\alpha_i|=q^{1/2}$. Thus
\begin{equation}
  \label{eq:161}
  \xi_\F(x) = x^{g-1} \zeta_\F(x) = \frac{(-1)^{g-1}}{\prod \alpha_i} 
  \frac{\prod_{i=1}^{g} (\alpha_i - x     
    ) }{(1-x)}
  \frac{\prod_{i=1}^{g} (\alpha_i - q/x) }{(1-q/x)} 
\end{equation}
gives the desired factorization \eqref{eq:114}, in which we can
choose either square root of the prefactor. 

\subsubsection{}

In general, the function
\begin{equation}
  \label{eq:114}
Z_\psi(x) = \Psi(x^{-1}) \Psi(x/q) \,, \quad \Psi(x) = \frac{\psi(x)}{1-x^{-1}} \,.
\end{equation}
plays a special role in the computation of genera for manifolds $\bX$ having
a polarization $T^{1/2}$ of the tangent bundle $T\bX$. By definition, a
polarization is a solution of the equation
\[
  (T^{1/2})^* + q^{-1} T^{1/2} = T\bX
\]
in equivariant K-theory of $\bX$. A polarization makes the Chern roots
of $T\bX$ appear in pairs $\{x_i^{-1}, x_i/q\}$.

Theorem \ref{t_T3}  is an easy consequence of Theorem
\ref{t_Langl} and Theorem \ref{t_f1f2} , see Corollary \ref{c_T3}.

\subsubsection{}

The next key point is that the spectral decomposition for
Eisenstein series corresponds to a certain decomposition of $\cT$
in the additive group of equivariant cohomologically proper
bordism classes.

In equivariant cohomologically proper situation, there are certain
convenient scissor relations, which look as follows. 

\subsubsection{}

To fix ideas, let
\[
  f: \bY \to \bX
\]
be an equivariant
embedding of a complex subvariety into a complex manifold. We now
stress that, while $\bX
\setminus \bY$ is not proper, it may very well be cohomologically
proper and so have a well-defined bordism class $[\bX \setminus \bY]$. We then define
\[
  \Thom(\bY \xrightarrow{\,\, f \,\,} \bX) = [\bX] - [\bX \setminus \bY] \,, 
\]
which we will usually abbreviate to just $\Thom(\bY)$, the rest of the
data being understood. Clearly, it depends only on the neighborhood of
$\bY$ in $\bX$, and only on the normal bundle if $\bY$ is smooth. 

Concretely, by the long exact sequence for local cohomology, we have 
\begin{align}
 \psi(\Thom(\bY)) 
  & = \sum (-1)^i H^i_\bY(\psi(TX))\,,  \label{eq:101}
\end{align}
in  equivariant K-theory,
were $H^i_\bY$ denotes cohomology with support in $\bY$.

\subsubsection{}
We denote by $\cB = \bG/\bB$ the flag variety of $\bG$. 
It is easy to see from the definitions that, set-theoretically
(in the sense of set-theoretic support of the structure sheaf),
we have 
\begin{equation}
\cT = \bigsqcup_{e}
\left(\cB^e \times \cB^e\right) \big/ \bC_e\,,\label{eq:115}
\end{equation}
  where the sum is over the conjugacy classes of nilpotent
  elements $e\in
  \fg = \Lie \bG$,
  \begin{align*}
    \cB^e & = \textup{the Springer fiber of $e$} \\
    & = \textup{fixed locus of $e$ acting on $\cB$}\,, 
  \end{align*}
and $\bC_e$ is the
centralizer of $e$. For instance, the piece corresponding to
the regular nilpotent $e$ is open in $\cT$, while the piece
corresponding to $e=0$ is zero section
\[
  \bB \setminus \bG / \bB \subset \cT \,.
\]
Needless to say,
Springer fibers are objects of paramount importance in geometric
representation theory, and it is not surprising that they also play
a key role in our work. 

\subsubsection{}
We prove the following cobordism upgrade of the decomposition
\eqref{eq:115}. In equation \eqref{eq:120}, $[q^{-1}]$ denotes the
$\Ct_q$-equivariant cobordism class of $\C$, where the action is
by $q^{-1}$. By our assumptions $\Psi(q^{-1})$ is invertible, and
it is convenient to introduce
the inverse of the class $[q^{-1}]$ itself. 

\begin{Theorem}\label{t_proper}  For each $e$, the
  $\Ct_q$-action on the space 
  \begin{align}
 \sThom(e) & \eqdef \Thom \left( \left(\cB^e \times
             \cB^e\right) \big/ \bC_e  \to \cT \right) \notag \\
    &= \frac{1}{[q^{-1}]^{2r}} \left([\cB^e] \times [\cB^e] \times \Thom(0 \to \Slice_e
      \fg)\right)
      \big/ \bC_e 
\label{eq:120}
  \end{align}
  is cohomologically proper,  and we
  have the decomposition
\begin{equation}
[\cT] = \bigsqcup_{e}
\left[\sThom (e) \right]\,.\label{eq:116_Thom}
\end{equation}
\end{Theorem}
\noindent
Here $q\in \Ct_q$ acts on $\fg$ by
\[
  \xi \mapsto q^{-1} \ad(q^{h/2}) \xi\,, \quad \xi \in \fg\,, 
\]
and by $q^{h/2}$ on $\cB^e$. This action fixes $e$ and preserves
$\cB^e$. Since the weights of this action are negative on
$\Slice_e \fg$, it is important for our formulas that $|q|>1$ in the
function field case and that $q \downarrow 1$ in the number
field case.

In \eqref{eq:120}, $[\cB^e]$ denotes the bordism class of $\cB^e$ with its
virtual structure, namely that of the zero locus of a vector
field: 
\begin{equation}
  \label{eq:142}
  [\cB^e] = \textup{Euler class of $q^{-1} T\cB$}  \,. 
\end{equation}
In particular, it is defined in $\bG$-equivariant
bordism and is independent of the choice of $e$.

Our proof of Theorem \ref{t_proper} uses the stratification
of $\cB^e$ studied by De Concini, Lusztig, and Procesi in
\cite{dCLP}.

\subsubsection{}

Theorem \ref{t_proper} results in the following geometric description
of the spectral projectors and the measure $\bmu_e$ in Theorem
\ref{t_spectral}.

It was noted by Langlands in \cite{L} that by the functional
equation for $\xi_\F$ the operator
\begin{equation}
\cP_\Langl = \frac{1}{|W|} \sum_{w \in W} \prod_{\substack{\alpha > 0 \\ w^{-1} \alpha <0 }}
\frac{\xi_\F(\alpha(\lambda))}{\xi_\F(1+\alpha(\lambda))} \,  w
\label{eq:118_cPlang}
\end{equation}
is a projection operator, acting on meromorphic functions.
In terms of the function $\Psi(x)$, this
fact may be 
explained as follows. In the K-theory case, for instance, we have
\begin{equation}
  \cP_\Langl = \bPi_-^{-1} \,  \left(\frac{1}{|W|} \sum_{w \in W} w \right)\,
  \bPi_- \,,
  \quad \bPi_{\mp} = \prod_{\alpha \gtrless 0}
  \frac{\Psi(x^\alpha)}{\Psi(q^{-1} x^{\alpha})} \,, 
\label{eq:1182}
\end{equation}
where the operator in the middle is the standard projector onto
$W$-invariant functions. We define
\begin{equation}
  \label{eq:141}
  \cP_{\pm} = \left(\sum_{w \in W} w \right)\,
  \bPi_\pm \,. 
\end{equation}
These take functions on $\bA$
to $W$-invariant functions and are
proportional to a projection operator.
In contrast to $\cP_\Langl$,
these operators do not introduce undesirable poles.

Viewed as operators
\begin{equation}
  \label{eq:143}
  K_\bG(\bB) \xrightarrow{\quad \cP_\pm \quad}  K_\bG(\pt) \,, 
\end{equation}
The operators $\cP_\pm$ have the following clear
geometric meaning
\[
  \cP_\pm \cdot f = \chi\left( f \otimes \psi\left(\big[\textup{Euler class of
  $q^{-1} T(\bG/\bB_\pm)$}\big]\right)\right)  \,, 
\]
the difference between $\pm$-choices being the
identification of $K_\bG(\bB)$ with functions on $\bA$.

\subsubsection{}

For a nilpotent element $e$, we define
\begin{equation}
  \label{eq:18}
  \Psi_e = \frac1{\Psi(q^{-1})^{2r}}
  \, \Psi \left( \Thom\left(0 \to \Slice_e
      \fg)\right)
      \big/ \bC_e \right) \,, 
\end{equation}
which can be seen to be convergent on $q^{h/2}
\bC_{\phi,\comp}$ for $|q|>1$. Very explicit formulas 
for the spectral measure can be found in Proposition
\ref{p_spectral_Z} below. 

With some additional analysis of the effect of the zeros of the
function $\psi(x)$, Theorem \ref{t_proper} implies
the following generalization of the 
decomposition \eqref{e_spectral}. It explains
the role of $\cP_+$ as spectral projectors and the role
of $\left|\Psi_e\right|$ as the spectral measure. It applies
to an arbitrary function $\psi(x)$ satisfying the hypothesis
of the theorem, not just functions of the form \eqref{eq:108}. 

\begin{Theorem}\label{t_L2} 
  Suppose $q>1$ and that the function $\psi(x)$ satisfies:
  \begin{enumerate}
  \item[(1)] it is holomorphic
  in $\Ct$ and nonvanishing outside the critical strip $1/q<|x|<1$; 
  \item[(2)] the function $- \psi(x^{-1}) \psi(x/q)$ is real and positive 
    for $x> q$.  
  \end{enumerate}
Then 
  \begin{equation}
    \label{eq:19}
    \left(f \boxtimes f^*, \psi(\cT)\right) =
      \sum_{e} 
    \int_{\bbS_e} \left| \cP_+ f \right|^2 
    \left|\Psi_e \right|\, 
      d_\textup{Haar}  \,, 
    \end{equation}
    The conclusion \eqref{eq:19} remains valid in cohomology provided: 
\begin{enumerate}
  \item[(1$'$)] $\psi(s)$ is entire with no zeroes outside $-1 < \Re s
    <0$; 
  \item[(2$'$)] the function $-\psi(-s) \psi(s-1)$ is real and positive 
    for $s>1$.  
  \end{enumerate}
  \end{Theorem}

  For the $\zeta$-functions \eqref{eq:108}, the above positivity
  means the positivity for positive real arguments in the region of
  convergence, which is immediate.

  \subsubsection{Remark}
\label{sec:remark_growth}

  In cohomology, the finiteness of the integral in \eqref{eq:19} may
  be guaranteed for Paley--Wiener functions $f$ by at most polynomial
  growth of the ratio $Z_\psi(x)/Z_\psi(x\pm 1)$ in the imaginary
  direction. Such bounds for $\xi_\F$-functions are classical, see for
  instance Chapter XIII.5 in \cite{Lang}. If the function
  $Z_\psi(x)/Z_\psi(x\pm 1)$ grows faster then one should
  demand faster decay from the test functions $f$. In this text, 
  we make the simplifying assumption that the functions
  $Z_\psi(x)^{\pm 1}$ grow at most
  exponentially in the imaginary direction. This hypothesis is
  satisfied by the $\xi_\F$-functions of global fields (for which the exponential
  factors come from the Archimedean places). Faster growing
  functions $Z_\psi(x)^{\pm 1}$ may, in principle, be handled by modifying the
  cutoff function in the proof of Proposition \ref{p_convolve}. 

\subsection{Example}

\subsubsection{}

To get a better feeling for statement of our main result, we
consider the following example.
Let $V$ be a representation of $\mathfrak{sl}(2,\C)$ of the form
\begin{equation}
  \label{eq:1-1}
  V = S^4 \C^2 \oplus \textup{trivial 2-dimensional} \,. 
\end{equation}
Since odd-dimensional representations of  $\mathfrak{sl}(2,\C)$ are
orthogonal, this defines a map
\begin{equation}
  \label{eq:2-1}
  \phi: \mathfrak{sl}(2,\C) \to \mathfrak{so}(7,\C) 
\end{equation}
with centralizer
\begin{equation}
  \label{eq:3-1}
  \bC_\phi = S( O(1) \times O(2))  \subset SO(7,\C) \,. 
\end{equation}
A maximal compact subgroup of $\bC_\phi$ has the form 
\begin{equation}
  \label{eq:4}
  \bC_{\phi,\textup{cmp}} = \{1\} \times \{\textup{rotations in $\R^2$}\}
  \sqcup \{-1\} \times \{\textup{reflections in $\R^2$}\} \,. 
\end{equation}
By Theorem \ref{t_spectral}, in the function field case,
the representation $\phi$ contributes to the
spectrum the conjugacy classes of
\begin{equation}
  \label{eq:7-1}
  \phi\left(\left[
    \begin{matrix}
      q^{1/2} \\
      & q^{-1/2} 
    \end{matrix} \right] \right) \bC_{\phi,\textup{cmp}} \subset
SO(7,\C) \,. 
\end{equation}
In particular, the second component in \eqref{eq:4} contributes the conjugacy classes
of
\begin{equation}
- \phi\left(\left[
    \begin{matrix}
      q^{1/2} \\
      & q^{-1/2} 
    \end{matrix} \right] \right) \oplus \{\textup{reflections in
  $\R^2$}\}  \subset SO(7,\C)\label{eq:5-1} \,. 
\end{equation}
All of the elements in \eqref{eq:5-1} are conjugate in $SO(7,\C)$ with eigenvalues
\begin{equation}
(-q^2,-q,-1,-q^{-1},-q^{-2},-1,1)\label{eq:6-1}
\end{equation}
and hence produce a point in the spectrum, that is, an
$L^2$-eigenfunction of the Hecke operators. 

\subsubsection{}

This point in the spectrum is specific to function fields and has no
analog in the number field situation.

Recall that, by construction, the spectrum is a
subset of $\uA/W$, where $\uA$ was defined in \eqref{eq:110}.
In \eqref{eq:111}, we identified $\uA$ with the spectrum of
$\bA_\comp$-equivariant K-theory of a point, in the function field case, and with
the spectrum of $\bA_\comp$-equivariant cohomology of a point, in the
number field case. Those are the cohomology theories
associated to the curve $\cQ^\vee$ from \eqref{eq:11}. 

Equivariant cohomology deals only with Lie algebras and completely 
ignores finite groups. 
Since the eigenvalue \eqref{eq:6-1} comes from the nonidentity connected
component in \eqref{eq:4},  it has no analog in
the number field case.

  \subsection{Future directions}
  \label{s_future}

  \subsubsection{}

Since the focus of the current subsection is on the automorphic side, we
will make the $\bG \leftrightarrow \La \bG$ switch to shorten the
notation. To avoid dealing with  the characters of the center, we
will assume that $\bG$ is semisimple.

The general Langlands spectral decomposition of
\[
  \cH = L^2\!\left(\bG(\A)\,\big/\,\bG(\F)\right)
\]
is organized according to the data of a parabolic subgroup
\[
  \bG \supset \sP =
  \underbrace{\bM}_{\textup{Levi}} \ltimes
\underbrace{\bU}_{\textup{unipotent}}
\]
and
a cuspidal automorphic form
\begin{equation}
\bPhi: \bG(\A)\,\big/\,\sP(\F) \bU(\A) \to \C  \,,  \label{eq:167}
\end{equation}
see \cites{Labesse, L}. This means that $\bPhi(g m)$
is a cusp form on $m\in \bM$ for any $g$ and, in
particular, an eigenfunction of the center of  $\bM$.

\subsubsection{}
To construct an Eisenstein series from $\bPhi$, we consider the torus
\[
  \bA = \sP/[\sP,\sP]
  \]
  and the natural maps
  \begin{alignat}{2}
    a:&& \sP & \to \bA \\
    \| \, \cdot \, \|: && \quad \bA(\A) & \to \Lambda \otimes
    \cQ \label{eq:166_}\,. 
  \end{alignat}
 We denote by  $\Lambda^\vee$ and  $\Lambda$ the character and the
 cocharacter lattices of $\bA$,
 respectively. 

  The function $\| a(\, \cdot \, )\|$ is invariant under
  $\bK \cap \sP$, and hence extends to a unique
  function on $\bG$ via 
  the $\bG = \bK \sP$ decomposition. We denote this function
  of $\bG$ by the same symbol. The definition
  \eqref{eq:97}  generalizes as follows: 
\begin{equation}
E(\bPhi, \lambda,g) = \sum_{\gamma \in \bG(\F)/\sP(\F)}
\| a(g \gamma)\|^{\lambda+\rho} \, \bPhi(g
\gamma)\,. \label{eq:97bis} 
\end{equation}
Here $\| a(\, \cdot \,)\|^{2\rho}$ is the character coming
from the scaling of the Haar measure on $\bU$ and 
\[
  \lambda \in \La \uA \eqdef \Lambda^\vee \otimes \cQ^\vee =
  \begin{cases}
  \La \bA \,, &\chr \F >0 \,,\\
  \Lie \La \bA\,, & \chr \F = 0 \,. 
\end{cases}
\]
In the function field case, the dependence on $\lambda$ may
factor through a quotient $\La  \bA / \Fix_{\bPhi}$ by
a finite group. The group $\Fix_{\bPhi}$ is trivial in the number field
situation. 

\subsubsection{} 

As in
\eqref{eq:98}, Mellin transform in the variable $\lambda$
produces series of the form
\begin{equation}
  \label{eq:169}
  E(\bPhi, \phi^\vee,g) = \sum_{\gamma \in \bG(\F)/\sP(\F)}
\phi^\vee(\| a(g \gamma)\|)\, \bPhi(g
\gamma)\,,
\end{equation}
where $\phi^\vee$ can be taken to be a smooth compactly
supported function on the image in  \eqref{eq:166_}.

The series \eqref{eq:169} are in $\cH$ and their norm is computed by a generalization of Theorem \ref{t_Langl}, also due
to R.~Langlands, see \cite{L, Labesse}. This generalization  
involves intertwining operators $\cR_w(\lambda)$
that enter into the functional equation 
\begin{equation}
E(\bPhi, \lambda,g) = E(\cR_w(\lambda) \bPhi, w \cdot 
\lambda,g)\,, \quad w \in W\,, \label{eq:168}
\end{equation}
for the series \eqref{eq:97bis}. Recall that the ratios of the
$\xi$-function in \eqref{eq:100} is how these intertwining
operators act on 
$\bPhi(g)=1$ for $\sP=\bB$.

\subsubsection{}

In full generality, our knowledge of the singularities of the
intertwining operators $\cR_w(\lambda)$ is incomplete. However,
we can hypothesize that, in all situations, there is the following
generalization of Theorem \ref{t_T3}, see \cites{KO,KO2} for
more on this. 

General Langlands philosophy links the automorphic form $\bPhi$ to
a representation
\begin{equation}
\pi_{\bPhi}: \Gamma \to \La \bM \,, \label{eq:183}
\end{equation}
where $\Gamma$ is a group of the
form
\[
  1 \to \Gamma' \to \Gamma \to \cQ\to 1\,. 
\]
For a
function field $\F$, the group
$\Gamma$ is the Weil group of $\F$. The nature of the group $\Gamma$
in the number field case is more mysterious, see for instance
\cite{Lmarch}.

We will assume that the Zariski closure of the image 
$\pi_{\bPhi}(\Gamma)$ in $ \La \bM$ is reductive and also
irreducible, which means that the commutant of $\pi_{\bPhi}(\Gamma)$ in $\La \bM$ is finite
modulo the center $\La \bA \subset \La \bM$.

We note that, while it is extremely beneficial to be guided by correspondence with
representations \eqref{eq:183},  in practical terms, formulas of the form \eqref{eq:170} below may be stated,
proven, and used without referring to representation of a specific
group $\Gamma$. 

\subsubsection{}

We denote by
\[
  \bG^{\triangledown} = \Big(\La \bG\Big)^{\pi_{\bPhi}(\Gamma')}
\]
the commutant of the image of $\Gamma'$ in $\La \bG$. This is an
algebraic group with an action of $\cQ$ by automorphisms. It contains
a $\cQ$-stable Borel subgroup
\[
  \bG^{\triangledown} \supset \bB^{\triangledown} =
  \left(\bG^{\triangledown} \cap \La \sP \right)_0 \,, 
\]
where the subscript denotes the connected component of the identity. 
We define
\begin{equation}
  \label{eq:181}
  \cT = T^*( \bB^{\triangledown}  \backslash \bG^{\triangledown} /
  \bB^{\triangledown} )  \,. 
\end{equation}
Note that the group $\bG^{\triangledown}$ and hence the stack $\cT$
may in general be disconnected.
We make the group $\cQ\subset \R_{>0}$ acts on $\cT$
by the combination of action induced from the base with
the scaling of the cotangent fibers with the inverse of the
tautological character.

We expect the following generalization of
Theorem \ref{t_T3}
\begin{equation}
  \label{eq:170}
  \| E(\bPhi, \phi^\vee) \|_\cH^2 \overset{?} =
  \int_{\cT} \bbc(\bPhi) \cup \psi_\F(T \cT) \cup \phi \boxtimes \phi^*\, \,, 
\end{equation}
where
\begin{itemize}
\item the integral and the cup product denote the pushforward and the
  product in the cohomology
  theory corresponding to $\cQ^\vee$, 
\item $\bbc(\Phi)$ is a certain characteristic class that involves,
  among other things, automorphic $L$-functions for $\bM$,
\item the test function $\phi \boxtimes \phi^*$ is pulled back via
  the composed map
  \[
  \cT \to \pt/ \left(\bA^\triangledown\right)^2 \to \pt/ \left(\La \bA/
      \Fix_{\bPhi}\right)^2\,, 
  \]
  where $\bA^\triangledown = \bB^{\triangledown}/[\bB^{\triangledown},
 \bB^{\triangledown}]$ and the map $\bA^\triangledown \to \bA/
      \Fix_{\bPhi}$ takes an element $a$ to the conjugacy class of
      $qa$, where $q\in \cQ$ is the generator. 
\end{itemize}

As before, the stack $\cT$ maps to the stack of nilpotent elements in
$\Lie \bG^\triangledown$ and the basic logic of our paper is fully
applicable to the analysis of distributions of the form \eqref{eq:170}. 

\subsubsection{}

In the unramified function field situation, the characteristic class
$\bbc(\bPhi)$ may be identified with what we call the L-genus
\cite{KO} of $\cT$, viewed as the fixed locus of $\Gamma'$ inside
a certain ambient stack. The spectral analysis in that situation is
the subject of the paper \cite{KO2}.

Beyond the unramified situation, particularly
attractive geometric formulas for the intertwining operators
$\cR_w(\lambda)$ are available when $\sP=\bB$ and the
function $\bPhi$ is invariant under an Iwahori subgroup.

A formula of the form \eqref{eq:170} is equivalent to a formula for the constant 
term of the the Eisenstein series
\eqref{eq:97bis}.  
Going beyond the constant term, we expect to find more involved
geometric formulas for all Fourier coefficients of the Eisenstein series
\eqref{eq:97bis}.  Their shape will be discussed elsewhere.

\subsection{Acknowledgements}

\subsubsection{}

We are grateful to R.~Bezrukavnikov, M.~Kapranov, I.~Krichever,
I.~Loseu, H.~Nakajima, and especially A.~Braverman for the discussions
we had during our work on this project. 

\subsubsection{}

 The paper was
completed while A.O.\ was working at the Kavli IPMU, University of
Tokyo. He  thanks all Japanese colleagues for warm hospitality,
friendship,
and support.

D.K.\ thanks ERC grant 669655 for financial support. A.O.\ thanks
the Simons Foundation for being supported as Simons Investigator.

\section{Integrals and residues from K-theory point of view}\label{s_geom}

\subsection{Equivariant pushforwards as distributions}

\subsubsection{} 

Let an algebraic group $\bG$, not necessarily
connected or reductive, act on an algebraic variety  $\bX$ and let $\cF$ be a
$\bG$-equivariant coherent sheaf on $\bX$. For our goals in this paper,
it is enough to consider this situation over the field $\C$ of
complex numbers. We will further assume that $\bX$
is quasiprojective and the $\bG$-action on $X$ is linearized, that is,
we will assume an existence of an equivariant embedding
\begin{equation}
  \label{eq:37}
  \iota_\bX : \, \bX \hookrightarrow \bP(V)\,, 
\end{equation}
where $V$ is a linear representation of $\bG$. Below, we will meet
generalizations in which $\bX$ can be a scheme, a
derived scheme, or a stack. These are technical, but convenient
if one wants to broaden the applicability of the geometric tools reviewed
in this section.

Many basic ideas discussed in this section are also discussed, from a
different angle, in our paper \cite{KO}. Some readers may prefer that
exposition. 
\subsubsection{} 

Our main interest will be in the situation when $\chi(\bX,\cF)$, while
possibly 
infinite-dimensional, has finite multiplicities of irreducible
$\bG$-modules. This is conveniently formalized by the concept of
cohomologically proper actions, see  \cite{DHLloc} and
Appendix \ref{a_proper}.

In what follows we assume that the action of $\bG$ on $\bX$ is
cohomologically proper. 

\subsubsection{}

For $V \in K_\bG(\pt)$ we define 
\begin{equation}
(V, \chi(\bX,\cF))  = \chi(\bG, V\otimes \chi(\bX,\cF)) 
\,, \label{eq:57}
\end{equation}
see Appendix \ref{a_coh} for our notational conventions related to
group cohomology.

Since we are in characteristic zero, we have the analog
of the Levi-Malcev decomposition
\begin{equation}
\bG = \bG_\red \ltimes \bG_\uni\,,\label{Most}
\end{equation}
due to G.~Mostow \cite{Mostow}. The irreducibles for $\bG$ are the
same as irreducibles for $\bG_\red$. We will view
\[
  \C[\bG_\red]^{\bG_\red} = K_\bG(\pt) \otimes_\Z \C
\]
as our supply of \emph{test functions} and will realize $\chi(\bX,\cF)$
as a conjugation-invariant distributions on
$\bG_{\textup{red}}$ via the pairing \eqref{eq:57}. 

The tools of equivariant K-theory impose many
geometric relations on the totality of pushforwards $\chi(\bX,\cF)$.
These can be thus translated to relations among distributions.

\subsubsection{}

More generally, we may have an action of a larger group 
\begin{equation}
  \label{eq:173}
  \bGt = \bG \times \bAut
\end{equation}
on $\bX$, in which case the Euler characteristic \eqref{eq:57}
takes values in $K_{\bAut}(\pt)$, that is, in 
finite-dimensional virtual representations of $\bAut$. 

An example of such group $\bAut$ relevant to us
is the scaling action of $q\in \Ct$
that appears everywhere in this paper. Since all constructions
in this paper will be equivariant with respect to this scaling action,
such extra equivariance will always be understood.

\subsubsection{}

In the discussion of cohomological properness, it is important to pay
attention to both the space and the group. In the current discussion,
$\bX/\bG$ is assumed to be cohomologically proper and
distributions \eqref{eq:57}  take values in functions on $\Aut$, or in numbers
after evaluation at $q\in \Aut$. This is the
setting in Theorem \ref{t_T3} with $\Aut=\Ct_q$.

A weaker hypothesis is the cohomological properness of $\bX/\bGt$.
In this case, $\chi(\bG, V\otimes \chi(\bX,\cF))$
is typically an infinite-dimensional $\Aut$-module. Its trace may converge
for certain $q\in \Aut$, but, in general, it should be interpreted as a
distribution on $\Aut$. This is the setting in Theorem \ref{t_proper}, where
the assumption $|q|>1$ is important for the convergence of the trace
in \eqref{eq:18}. Evaluated at $q$ in the region of convergence,
\eqref{eq:57} defines a $\C$-valued distribution on $\bG_\red$. 

\subsubsection{}\label{s_distr}
More precisely, the word 
\emph{distribution} 
means the following in this paper. 
For basic reasons of finite generation, the function
\[
  \mu \mapsto \tr q \Big|_{\Hom(V^\mu,
  \chi(\bX,\cF))}\,, \quad q\in \bAut\,, 
\]
where $V^\mu$ is an irreducible representation with highest weight
$\mu$, grows at most exponentially, and the exponential growth
is the typical situation.

Therefore, all test functions in this paper will be functions on
$\bG_\red$ analytic in a certain sufficiently large (depending on the
geometric problem and on the norm of 
$q\in \bAut$) neighborhood of a maximal compact subgroup.
Similarly, for a Lie algebra we will consider functions analytic in
a neighborhood of the Lie algebra of the compact form. Analyticity
of the test functions means that the usual rules of operating
with contour integrals apply to distributions given by such integrals
with an analytic density.

Support of distributions will be also interpreted in the sense of
analytic functions. Namely, the support of a distribution $\chi$ is the
reduced subvariety defined by largest ideal annihilated by $\chi$.

\subsubsection{}
We observe that
\begin{equation}
  K_\bGt(\bX)
  \xrightarrow{\quad \chi \quad} \Hom(K_\bG(\pt),K_{\bAut}(\pt)) \label{eq:38}
\end{equation}
is a map of $K_\bG(\pt)$-modules. Whence the following 

\begin{Lemma}
  \begin{equation}
    \label{eq:41}
    \supp \chi(\bX,\cF) \subset \{g \in \bG_{\textup{red}}, \left(\supp
      \cF\right)^g \ne \varnothing \} \,. 
  \end{equation}
\end{Lemma}

\begin{proof}
Without loss of generality, we can assume that $\bX$ is the reduced
support of $\cF$. By localization in equivariant K-theory, the module $K_\bGt(\bX)$
is supported over the points $g\in \Spec K_\bG(\pt)$ that have a fixed
point on $\bX$, which proves the lemma. 
\end{proof}

\subsubsection{}\label{s_finite}
In this paper, we compute Euler characteristics. Therefore, without
loss of generality, we can replace $K_\bG(\bX)$, defined algebraically
or topologically, with its \emph{numerical version}, where one mods
out by $\cF$ such that
\[
  \chi(\bX, \cF \otimes \cV) = 0
\]
for all equivariant vector bundles $\cV$ on $\bX$. Even in the algebraic
K-theory situation, this quotient is often finitely generated over 
$K_\bG(\pt)$. 

\begin{Lemma}
If $\bX$ is cohomologically proper
and its numerical K-theory is finitely generated over $K_\bG(\pt)$ then the distributions
$\chi(\bX,\cF)$ span a finitely generated module over functions on
$\bG_\red$. 
\end{Lemma}

\noindent 
For concrete $\bX$ of importance to us in this paper, the
algebraic, topological, and numerical 0th K-groups all coincide and
are finitely generated over $K_\bG(\pt)$.

\subsubsection{Example}\label{s_eGH} 

If $\bX=\bG/\bH$, with $\bG$ and $\bH$ reductive then 
\[
  \chi(\cO_{\bG/\bH}) = \delta_{\bH_\comp} \,,
\]
restricted to conjugation-invariant functions, where
\[
  (f, \delta_{\bH_\comp}) = \int_{\bH_\comp} f \, d_{\textup{Haar}}
\]
is the integral with respect to the probability Haar measure.
Indeed, by the Peter-Weyl theorem, we have
\begin{equation}
  \label{PW}
  \chi(\cO_{\bG}) = \bigoplus_{\lambda} \left(V^\lambda\right)^* \boxtimes
  V^\lambda \,, 
\end{equation}
as $\bG\times \bG$-module, where $V^\lambda$ ranges over irreducibles.
Denoting by $\chi^\lambda$ the character of $V^\lambda$, we have 
\[
  (\chi^\lambda, \chi(\cO_{\bG/\bH})) = \dim (V^\lambda)^\bH 
= \int_{\bH_\comp} \chi^\lambda \, d_{\textup{Haar}}\,.
\]
Note that this distribution is indeed supported on elements that have
a fixed point on $\bX$, that is, elements conjugate to elements of
$\bH$.   

\subsubsection{Example}\label{s_eGHY} 

Extending Example \ref{s_eGH} , suppose 
\[
  \bX = \bG \times_\bH \bY
\]
where $\bG$ is reductive, $\bH \subset \bG$ is an \emph{arbitrary}
algebraic subgroup, and $\bY$ is a $\bH$-variety. An $\bH$-equivariant
coherent sheaf $\cF$ induces a $\bG$-equivariant coherent sheaf on
$\bX$ which we denote $\Ind \cF$. Since $\bX$ is a free quotient by
$\bH$, 
we have from \eqref{PW} and \eqref{eq:12} 
\[
  \chi(\bX,\Ind \cF) = \sum_{\lambda} 
\left(V^\lambda\right)^*  \boxtimes \chi(\bH,  V^\lambda \otimes
\chi(\bY,\cF))\,. 
\]
Therefore, 
\[
  (f, \chi(\bX,\Ind \cF))_{K_\bG(\pt)} = (\Res f, \chi(\bY,
\cF))_{K_\bH(\pt)}
\]
where $\Res$ denotes restriction of functions from $\bG$ to
$\bH$. Denoting by $\Ind$ the linear map on distributions that is dual
to restriction of functions, we obtain 
\begin{equation}
  \label{eq:1}
  \chi(\bX,\Ind \cF) = \Ind \chi(\bY, \cF) \,. 
\end{equation}
It is instructive to see what this formula says for line bundles on
$\bG/\bB$, where $\bB$ is the Borel subgroup. In this case, $\bY$ is a
point and $\cF$ is a $1$-dimensional representation of $\bB$. 

\subsubsection{Example}\label{ex_C1} 

Consider the defining action of $\bG=\Ct$ on $\bX=\C$. We have
\begin{equation}
\tr_{\chi(\cO_\C)} x  = 1 + x^{-1} + x^{-2} + \dots \,, \quad  x\in
\Ct\,, \label{eq:53}
\end{equation}
and this converges for $|x|> 1$. Thus
\begin{equation}
  \label{eq:43}
\left(f(x), \chi(\cO_\C)\right) = \int_{|x|=\varsigma }
\frac{f(x)}{1-x^{-1}} \, \frac{dx}{2\pi i x}  \,, \quad \varsigma > 1
\,, 
\end{equation}
which we can write as 
\begin{equation}
  \label{eq:2}
  \chi(\cO_\C) = \frac{1}{1-x^{-1}}\delta_{\varsigma U(1)} \,, \quad |\varsigma| > 1 \,. 
\end{equation}

\subsection{Characteristic classes}\label{sGdqs}

\subsubsection{}

If $\cV$ is a locally free sheaf on $\bX$ then we can take its 
characteristic classes, that is, tensor functors as the integrand
$\cF$ in \eqref{eq:57}. 
 More precisely, given a univariate Laurent polynomial 
 \[
   \psi(x) \in R[x,x^{-1}] 
\]
with coefficients in a ring $R$, we define
\[
  \psi(\cV) = \prod \psi(x_i) \in K_\bG(X) \otimes_\Z R \,, 
\]
where $x_i$ are the Chern roots of $\cV$.
\subsubsection{}

In particular, if $\bX$ is smooth, we can take $\cV = T\bX$ and define 
\[
  \psi_\bG(\bX) = \chi(\bX, \psi(T\bX)) \in \Hom_\Z(K_\bG(\pt),R) \,. 
\]
While integrality is always a very important consideration, for the
particular application we have in mind it will be enough to take
$R=\C$, the field of complex numbers, or $R= K_{\Aut}(\pt) \otimes \C$
in the presence of automorphisms. 

\subsubsection{}

For $\bX$ smooth and  proper,
the Hirzebruch-Riemann-Roch formula gives 
\[
  \psi_\bG(\bX) = \int_\bX \prod \frac {\xi_i}{1-e^{-\xi_i}} \,
\psi(e^{\xi_i}) \,,
\]
where $\{\xi_i\}$ are the cohomological Chern roots of $T\bX$. 
From this it is clear that the totality of $\psi_\bG(\, \cdot \,)$ differs by
reindexing only from the usual equivariant complex genera and is a
function of the complex cobordism class of $\bX$. 



\subsubsection{Example}\label{ex_C2}

Continuing with Example \ref{ex_C1}, we have
\begin{equation}
  \label{eq:45}
  \psi_{\Ct}(\C)  = 
\Psi(x) \, \delta_{\varsigma U(1) }   \,, \quad |\varsigma| > 1 \,, 
\end{equation}
where the function 
\begin{equation}
  \label{eq:46}
  \Psi(x) = \frac{\psi(x)}{1-x^{-1}}
\end{equation}
has a simple pole at $x=1$. Since $\Ct$-action on $\C$ is the
fundamental building block on the whole theory, it will be convenient
to use $\Psi$ instead of $\psi$ in many formulas below.

\subsubsection{}

Note that \eqref{eq:45} makes sense and depends analytically on
$\psi(x)$, as long as it is analytic in some neighborhood of $|x|=1$ that
contains $\varsigma \, U(1)$.

In the same way, in what follows, we can always replace a Laurent
polynomial $\psi(x)$ by a function analytic in a suitable neighborhood
of $|x|=1$. This fits into the following general definition. 

\subsubsection{Definition}\label{s_K_an} 

Let $\cU \in \Spec K_\bG(\pt)$ be an open set and consider
the diagram
\begin{equation}
  \label{eq:172}
  \xymatrix{
    \pi^{-1}(\cU) \ar[rr] \ar[d] && \Spec K_\bG(\bX) \ar[d]^\pi \\
    \cU  \ar[rr] && \Spec K_\bG(\pt)
    }\,, 
\end{equation}
in which the natural projection $\pi$ is assumed to be
a finite map as in Section \ref{s_finite}. We define the sheaf of 
\emph{analytic K-theory classes} by
\begin{equation}
K_\bG(\bX)_{\an} \Big|_{\pi^{-1}(\cU)}= K_\bG(\bX) \otimes_{ K_\bG(\pt)}
\cO_{\cU,\an} \,. \label{eq:174}
\end{equation}
Since $\pi$ is finite, this is the same as analytic functions on
$\pi^{-1}(\cU)$.

Suppose that $\cU$ is so large that  $\chi(\cF)$ is a well-defined linear functional
on $\cO_{\cU,\an}$ for all $\cF \in K_\bG(\bX)$, as in Section
\ref{s_distr}. 
Then by \eqref{eq:174} pushforward of
$K_\bG(\bX)$ uniquely extends to 
all $\cF\in K_\bG(\bX)_{\cU,\an}$.

In addition to analyticity, one may also impose growth conditions on
analytic functions. For instance, in equivariant cohomology, that is,
the number field situations, one very natural
space to consider is
\begin{equation}
  \label{eq:177}
  \Hd_\bG(\bX)_{PW^*} \Big|_{\pi^{-1}(\Re^{-1}(\cU))}= \Hd_\bG(\bX) \otimes_{\Hd_\bG(\pt)}
PW^*_{\Re^{-1}(\cU)} \,, 
\end{equation}
where the real part $\Re$ is the projection along the Lie algebra of the
compact form of $\bG$, $\cU$ is an open set in the target of this
projection, and analytic functions in $PW^*$ are required
to grow at most polynomially along the fibers of $\Re$. That makes
their pairing with the Paley--Wiener functions
well-defined.

\subsubsection{}\label{s_psi_vir}

Let $\psi(x)$ be an entire function. We define 
$\psi(\cV)$ also for virtual
vector bundles $\cV = \cV_1 - \cV_2$ by
\[
  \psi(\cV_1 - \cV_2) = \frac{\psi(\cV_1)}{\psi(\cV_2)} \,, 
\]
provided the denominator is invertible on an open dense subset
of $\Spec K_\bG(\bX)$. This is an example of a meromorphic $K$-theory
class.

In what follows, we will assume that $\psi(x)$ is analytic and
nonvanishing for all $x\in \Ct$, until these assumptions are
explicitly relaxed in Section \ref{s_zeros}. With this assumption, $\psi(\cV)$
is an analytic K-theory class for virtual vector bundles.

When $\psi(x)$ does have zeroes, the pushforward of
$\psi(\cV_1-\cV_2)$
should be interpreted as a product of two distributions 
(or as the Fourier transform of their convolutions, since we define
pushforwards in terms of the Fourier transforms).
This product may or may not be well-defined.

In Section \ref{s_zeros}, for our concrete stack $\cT$, we prove
that the equivariant genus $\psi(\cT)$ is well-defined as a product
of two distributions provided all zeros of $\psi(x)$ lie in a certain critical
annulus (respectively, strip). Further, we show that this product is
given by the same formulas as in the case when $\psi(x)$ is nowhere
vanishing.

\subsubsection{} 

Many kinds of spaces considered in algebraic geometry have a natural
virtual tangent bundle. For instance, we may
consider a subvariety $\bX' \subset \bX$ cut out by a section $s$ of a vector
bundle $V$. In this case,
\[
  T \bX' = \left(T\bX - V \right)|_{\bX'} \,.
\]
If $s$ is transverse to the zero section then $\bX'$ is smooth, and this
reduces to the previous formula.

It is more intrinsic to view $\bX'$ as a derived scheme corresponding
the Koszul complex DGA
\begin{equation}
  \dots \xrightarrow{\,\, \iota_s \,\,} \Lambda^{2} V^* \otimes \cO_\bX
  \xrightarrow{\,\, \iota_s \,\,} V^* \otimes \cO_\bX\xrightarrow{\,\, \iota_s \,\,}
  \cO_\bX \,, \label{Koszul} 
\end{equation}
in which the differential is the contraction $\iota_s$ with the
section $s$. When the section $s$ is regular, this complex is acyclic
except in the final $\cO_{\bX'}$ term. Another extreme is when the
section $s$ vanishes. Note, however, that $\psi_{\bG}(\bX')$ does not depend on
$s$ if $\bX$ is proper. 

\subsubsection{}
A particularly important example for us are the Spinger fibers
$\cB^e$, which are the scheme-theoretic fixed loci of a vector field
$e$ on the flag manifold $\cB$, see Appendix \ref{a_Lie}.
{}From the above, we see that their bordism classes $[\cB^e]$ are
independent of $e$ and satisfy \eqref{eq:142}. 

\subsection{Quotient stacks} 

\subsubsection{}

Another natural way to weaken the smoothness requirement
is to replace $\bX$ with a quotient stack
$\bX/\bH$, where $\bX$ is smooth with an action of an extension
\[
  1 \to \bH \to \widetilde{\bG} \xrightarrow{\,\, \pi_\bH \, \,} \bG \to
1 \,. 
\]
Here we denote by $\pi_\bH$ the quotient by $\bH$. From the definition of a
quotient stack, 
\begin{equation}
\chi(\bX/\bH, \cF) = \chi(\bH, \chi(\bX, \cF)) \,,\label{eq:12}
\end{equation}
for any $\cF \in K_\bH(\bX)$. 
Further, 
\[
  T \bX/\bH = T \bX - \Lie \bH \,, 
\]
where the second term is a trivial bundle with a nontrivial action of
$\widetilde{\bG}$.

\subsubsection{}

\begin{Proposition}\label{p_stack} For any $\cF \in K_{\widetilde{\bG}}(\bX)$, we
  have 
  \[
    \chi(\bX/\bH, \cF) = \pi_{\bH,*} \, \chi(\bX, \cF) \,,
  \]
  where $\pi_{\bH,*}$ denotes push-forward of distributions under
  $\pi_{\bH}$. 
\end{Proposition}

\begin{proof}
  For $V \in K_\bG(\pt)$ we have
  \begin{align}
  (V, \chi(\bX/\bH, \cF))_{K_\bG(\pt)} &=
                                         \chi(\bG, V \otimes \chi(\bH,
                                         \chi(\bX, \cF))) \notag\\
    &=
      \chi(\bG, \chi(\bH, \pi_\bH^*(V) \otimes  \chi(\bX, \cF))) \label{eq:10b} \\
    &=
      \chi(\widetilde{\bG}, \pi_\bH^*(V) \otimes  \chi(\bX, \cF)) \,, 
\label{eq:10}
  \end{align}
where \eqref{eq:10b} uses linearity of cohomology
over invariants and \eqref{eq:10} uses \eqref{eq:7}. 
\end{proof}

\subsubsection{}

It is often convenient to factor \eqref{eq:12} as follows
\begin{equation}
  \label{eq:50}
  \chi(\bH, \chi(\bX, \cF)) = \chi(\bH_\red, \chi(\bH_\uni,\chi(\bX, \cF)))
\end{equation}
where 
\[
  1 \to \bH_\uni \to \bH \to \bH_\red \to 1 \,, 
\]
are the unipotent radical and the reductive part of $\bH$.
Correspondingly, we have an exact sequence
\begin{equation}
  \label{eq:55}
  1 \to \bH_\red \to \widetilde{\bG}' \xrightarrow{\,\, \pi_\red \, \,} \bG \to
1 \,, 
\end{equation}
where $\widetilde{\bG}'=\widetilde{\bG}/ \bH_\uni$.



\subsubsection{}\label{s_push_uni} 
We
denote by $\fh_\uni$ and $\fh_\red$ the Lie algebras of $\bH_\uni$ and
$\bH_\red$,  respectively. Both have well-defined classes in
$K_{\widetilde{\bG}'}(\pt)$. 

\begin{Proposition}
 We have 
  \begin{equation}
    \label{eq:48}
    \chi(\bX/\bH, \psi(T \bX/\bH)) = \pi_{\red,*}  \left( \frac{\chi(\bX, \psi(
                             T \bX))}{\psi(\fh_\red) \Psi(\fh_\uni)}
                         \right) 
                           \,. 
                         \end{equation}
\end{Proposition}

\begin{proof}
  Denoting
  \[
    V=\chi(\bX, \psi(TX)) \,, 
  \]
  we
  have the following chain of equalities in $K_{\widetilde{\bG}'}(\pt)$
  \begin{align}
    \chi(\bH_\uni,\chi(\bX, \psi(TX-\fh))) &= \psi(-\fh) \otimes
                                             \chi(\bH_\uni,V) \label{eq:236}\\
    &= \psi(-\fh) \otimes
                                             \chi(\fh_\uni,V) \label{eq:234}
    \\
     &= \psi(-\fh) \otimes V \otimes \sum
       (-1)^i \Lambda^i \fh_\uni^* \label{eq:235} \\
    &= V \otimes \psi(-\fh_\red) \otimes \Psi(-\fh_\uni) \notag\,. 
  \end{align}
  Here \eqref{eq:236} follows from $\fh$ being a multiple of the
  identity in $K_{\bH_\uni}(\bX)$, \eqref{eq:234} follows from
  \eqref{eq:8}, and \eqref{eq:235} uses the K-theory class of the 
  Chevalley-Eilenberg complex \eqref{ChE}. 
\end{proof}

\subsubsection{}

Now suppose $\bH_\red$ is connected and has a $\bG$-invariant maximal
torus $\bT$. Then we have a diagram
\begin{equation}
  \label{eq:56}
  \xymatrix{
   1 \ar[r] & \bH_\red \ar[r] & \widetilde{\bG}' \ar[rr]^{\pi_\red} && \bG \ar[r] &
   1 \\
   1 \ar[r] & \bT \ar[r] \ar[u] & \widetilde{\bG}'' \ar[rr]^{\pi_\bT} \ar[u] && \bG \ar[r] \ar@{=}[u]&
   1}\,, 
\end{equation}
and we can compute $\chi(\bX, \psi(
                             T \bX))$ in
                             $\widetilde{\bG}''$-equivariant
                             K-theory without loss of information. 

\begin{Proposition} If $\bH_\red$ is connected and has a $\bG$-invariant maximal
torus $\bT$ then 
  \begin{equation}
    \label{eq:49a}
   \chi(\bX/\bH, \psi(T \bX/\bH)) = \frac{1}{ \psi(1)^{\rk} |W|} \, \pi_{\bT,*} 
                           \left( \frac{\chi(\bX, \psi(
                             T \bX))}{\Psi(\fh-\ft)}
                           \right) 
\,. 
                         \end{equation}
where $\rk$ is the rank, and $|W|$ is the order of the Weyl group. 
\end{Proposition}
\begin{proof}
This follows from \eqref{eq:48} and 
Weyl integration formula. 
\end{proof}

\subsubsection{Definition}

There is, obviously, only a loss of information in pushing forward
distributions from $\widetilde{\bG}$ to $\bG$. In what follows, we
will remember the action of the automorphism in computations on
quotient stack.

Concretely, we define the $\psi$-genus of a quotient stack by
\begin{equation}
\psi_\bG(\bX/\bH) = (\bX/\bGt \to \pt/\bGt)_*\, 
\psi(T\bX/\bH) \,. \label{eq:175}
\end{equation}
Among other things, this is a convenient tool to include
characteristic classes of principal vector bundles over $\bX$ in our
computations. Indeed, any principal bundle gives a presentation of
$\bX$ as a free quotient by the corresponding group.



\subsection{Localization}

\subsubsection{}

We begin by recalling how equivariant localization is typically used
for computation of Euler characteristics. 

Suppose $\bX$ is smooth and let $g\in \bG$ be a semisimple element. Recall that
the smoothness of $\bX$ implies the smoothness of $\bX^g$. We decompose
\begin{equation}
T\bX\big|_{\bX^g} = T\bX^g + N_{\bX/\bX^g}\in K(\bX^g) \label{TMhN}\,, 
\end{equation}
where 
$N_{\bX/\bX^g}$ the normal bundle to the fixed locus. Localization in equivariant K-theory implies the following

\begin{Proposition}
  If $\bX^g$ is proper and $\Psi(x)$ is analytic at the weights
  of $g$ in the normal bundle $N_{\bX/\bX^g}$ then $\chi(\bX,\psi(T\bX))$ is analytic at $g$ and 
\begin{equation}
  \label{QMloc}
  \chi(\bX, \psi(T\bX))(g) = \chi(\bX^g, \psi(T \bX^g) \otimes \Psi(N_{\bX/\bX^g}))\,. 
\end{equation}
\end{Proposition}

\subsubsection{} 
The decomposition \eqref{TMhN} also 
  makes sense also for a virtual bundle $T\bX$. If
  \[
    N_{\bX/\bX^g} = V_+ - V_-\,,
  \]
  where $V_i$ are vector bundles on $\bX^g$, then we need to require
  that $\Psi(x)^{\pm 1}$ is regular at the weights of $V_\pm$.

\subsubsection{}
A fundamental property of localization formulas like \eqref{QMloc} is
that they are regular as functions of $g$ if $\bX$ itself is proper.
Note that terms in corresponding to connected components of $\bX^g$ have
poles in $g$, due to the pole of $\Psi(x)$ at $x=1$. However, these
poles cancel when we take all components of $\bX^g$ into account.

\subsubsection{}\label{s_C}
If $\psi(x)$ is meromorphic and $\bX^g$ is proper for some $g$ in the identity
component of $\bG$ then \eqref{QMloc} has a meromorphic extension to
$\bG$ and 
\begin{equation}
  \label{eq:49}
  \psi_\bG(\bX) \equiv \chi(\bX^g, \psi(T \bX^g) \otimes
  \Psi(N_{\bX/\bX^g})) \, 
   \delta_\gamma \! \mod \textup{torsion} \,, 
\end{equation}
where
\begin{itemize}
\item[$\bullet$] torsion refers to the torsion submodule for the action of
  $K_\bG(\pt)$,
\item[$\bullet$] the integration cycle $\gamma$ is a deformation of
$\delta_{\bG_\comp}$ 
that avoids the poles of the
integrand.
\end{itemize}
Here $\bG_\comp \subset \bG$ is a maximal compact
subgroup. 
Note that all choices of $\gamma$ are equivalent modulo
torsion.

Dropping the properness assumption on $\bX^g$ in \eqref{eq:49} as well as
specifying the torsion components of $\psi_\bG(\bX)$ requires a
certain care, in particular, attention to the
sign of the weights in the normal bundle to
$\bX^g$.

This leads to the analysis of the attracting and
repelling directions for the action of $g$ in $N_{\bX/\bX^g}$
in Sections \ref{s_Attr} and \ref{s_lc_s} below, respectively.

\subsection{Attracting manifolds}\label{s_Attr}

\subsubsection{} \label{s_sigmaG}
Consider the following model situation. Let
\begin{equation}
\sigma: \Ct \to \bG \label{sigma}
\end{equation}
be a 1-parameter subgroup such that there exists an action-preserving
fibration 
\begin{equation}
\pi: (\bG,\bX) \xrightarrow{\quad (\pi_\bG,\pi_\bX) \quad}
(\bG^\sigma,\bX^\sigma)\label{eq:54}
\end{equation}
with
\[
  \pi = \lim_{z\to 0} \, (\textup{conjugation by $\sigma(z)$},
\sigma(z)) \,. 
\]
Note this implies that $\ker \pi_\bG$ lies in the unipotent radical of
$\bG$ and that $\pi_\bX$ is a fibration in affine spaces if $\bX$ is
smooth.

Also note that since the semisimple conjugacy classes are closed,
every semisimple conjugacy class of $\bG$ intersects
$\bG^\sigma$. Put it differently, $\bG^\sigma$  contains a maximal
torus of $\bG$. Therefore, the distributions of interest to us on $\bG$ are
induced from $\bG^\sigma$. 

\subsubsection{}

A slightly more general setup is when
\begin{equation}
\widetilde{\bX} = \widetilde{\bG} \times_{\bG} \bX\,, \quad \sigma:  \Ct \to
\bG\,,\label{eq:62}
\end{equation}
and such that $(\bG,\bX)$ have a fibration as in \eqref{eq:54}.  This reduces to the situation from Section \ref{s_sigmaG}
as discussed in Section \ref{s_eGHY}.

\subsubsection{}

Let $\eta$ be a distribution on $\bG^\sigma/\sigma$. We define its
\emph{attracting lift} $\eta^\wedge$ to a distribution on $\bG^\sigma$ as follows. By construction, we
have a principal bundle
\begin{equation}
  \label{eq:60}
  \xymatrix{
    \Ct \ar@{^{(}->}[r]  & \bG^\sigma \ar[d] \\
    & \bG^\sigma/\sigma } \,, 
\end{equation}
for which we can choose a local trivialization with coordinates $x$ in
the base and $z$ along the fiber. We denote by $\eta_x$ the
distribution $\eta$ in the base coordinates and define 
\[
  (f(x,z), \eta^\wedge) = \int_{|z|=\const \gg 0} (f(x,z),\eta_x)
\frac{dz}{2\pi i z} \,.
\]
In other words, this extracts the constant term in $z$ in the $z\to
\infty$ expansion of $(f(x,z),\eta_x)$. Note this distribution is
invariant under rescaling of $z$, and therefore its definition does
not depend on the local trivialization. This procedure is illustrated
in Figure \ref{lift_eta}. 

Example \ref{ex_C1} is the most basic example of attracting lift
to $\bG^\sigma = \sigma = \Ct$.

\begin{figure}[!h]
  \centering
  \def\svgwidth{7cm}
\begingroup%
  \makeatletter%
  \providecommand\color[2][]{%
    \errmessage{(Inkscape) Color is used for the text in Inkscape, but the package 'color.sty' is not loaded}%
    \renewcommand\color[2][]{}%
  }%
  \providecommand\transparent[1]{%
    \errmessage{(Inkscape) Transparency is used (non-zero) for the text in Inkscape, but the package 'transparent.sty' is not loaded}%
    \renewcommand\transparent[1]{}%
  }%
  \providecommand\rotatebox[2]{#2}%
  \newcommand*\fsize{\dimexpr\f@size pt\relax}%
  \newcommand*\lineheight[1]{\fontsize{\fsize}{#1\fsize}\selectfont}%
  \ifx\svgwidth\undefined%
    \setlength{\unitlength}{443.5915354bp}%
    \ifx\svgscale\undefined%
      \relax%
    \else%
      \setlength{\unitlength}{\unitlength * \real{\svgscale}}%
    \fi%
  \else%
    \setlength{\unitlength}{\svgwidth}%
  \fi%
  \global\let\svgwidth\undefined%
  \global\let\svgscale\undefined%
  \makeatother%
  \begin{picture}(1,0.91325595)%
    \lineheight{1}%
    \setlength\tabcolsep{0pt}%
    \put(0,0){\includegraphics[width=\unitlength,page=1]{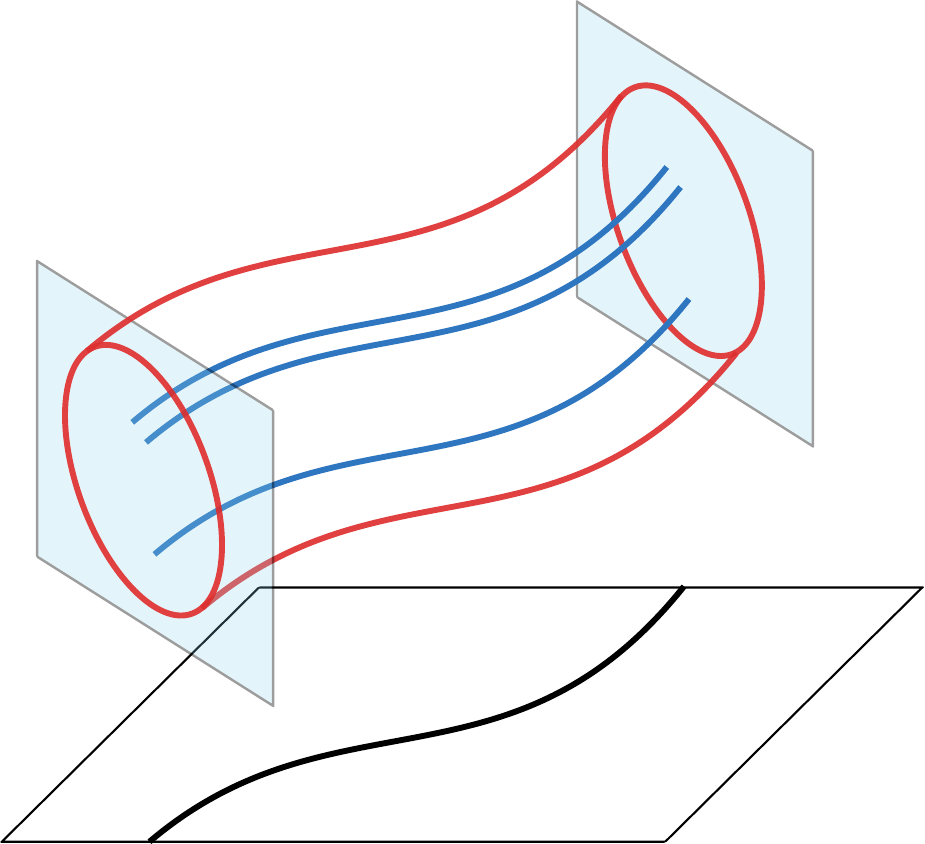}}%
    \put(0.17464669,0.7435186){\makebox(0,0)[lt]{\lineheight{1.25}\smash{\begin{tabular}[t]{l}$\bG^\sigma$\end{tabular}}}}%
    \put(0.77370981,0.20126326){\makebox(0,0)[lt]{\lineheight{1.25}\smash{\begin{tabular}[t]{l}$\bG^\sigma/\sigma$\end{tabular}}}}%
    \put(0.48192478,0.06957261){\makebox(0,0)[lt]{\lineheight{1.25}\smash{\begin{tabular}[t]{l}$\eta$\end{tabular}}}}%
    \put(0.44605684,0.72475563){\color[rgb]{0.87058824,0.18039216,0.18039216}\transparent{0.9137255}\makebox(0,0)[lt]{\lineheight{1.25}\smash{\begin{tabular}[t]{l}$\eta^\wedge$\end{tabular}}}}%
    \put(0.41478836,0.4853018){\color[rgb]{0.18431373,0.4627451,0.75294118}\transparent{0.81568629}\makebox(0,0)[lt]{\lineheight{1.25}\smash{\begin{tabular}[t]{l}{\small poles}
    \end{tabular}}}}%
  \end{picture}%
\endgroup%

  \caption{In the attracting lift $\eta^\wedge$ of a distribution $\eta$, we
    integrate around all possible singularities in the fiber
    direction in \eqref{eq:60}.}
\label{lift_eta}
\end{figure}
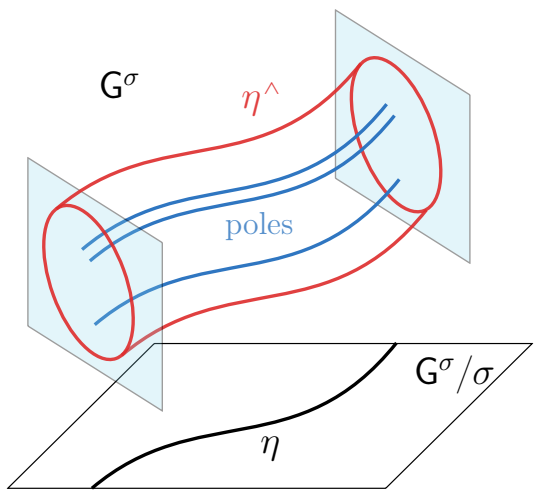

\subsubsection{}

For $\cF\in K_{\bG}(\bX)$ the push-forward
\begin{equation}
  \label{eq:71}
  \pi_{\bX,*} \cF \in K_{\bG^\sigma/\sigma}(\bX^\sigma)[z,z^{-1}]] 
\end{equation}
may be computed by localization and equals the expansion at
$z=\infty$ of a rational function in $z$. In \eqref{eq:71}, $z$ is the
character of the defining representation in the $\C^\times$ fiber in
\eqref{eq:60}, and $[z,z^{-1}]]$ denotes Laurent power series in
$z^{-1}$.

\begin{Proposition}\label{p_attr_lift}
Let $\widetilde{\bG}$ and $\widetilde{\bX}$ be as in \eqref{eq:62},
where the action of $\sigma$ on $\bX$ is attracting as in
\eqref{eq:54}.
If the action of $\bG^\sigma$ on
$\bX^\sigma$ is cohomologically proper then the action of $\widetilde{\bG}$ on
$\widetilde{\bX}$ is cohomologically proper and 
\[
  \chi(\widetilde{\bX}, \Ind \cF) = \Ind^{\widetilde{\bG}}_{\bG^\sigma}
\chi(\bX^\sigma,
\pi_{\bX,*} \cF)^\wedge \,. 
\]
\end{Proposition}

\begin{proof}
  Since the action is attracting, the expansion of the character of
  $\pi_{\bX,*}(\cF)^\wedge$ at infinity represents the character of
  the $\bG^\sigma$ action on \eqref{eq:71}. 
%
\end{proof}







\subsection{Local cohomology and stratifications}\label{s_lc_s}

\subsubsection{}

Let $\cF$ be a coherent sheaf on $\bX$ and $\bX_1\subset \bX $ a closed
subset. Associated to $\bX_1$ we have the local cohomology groups
$H^i_{\bX_1}(\bX,\cF)$ and the corresponding Euler characteristic
$\chi_{\bX_1}(\bX,\cF)$. The local cohomology groups fit in a long exact sequence that gives
\begin{equation}
  \label{longlocal}
  \chi(\bX,\cF) = \chi(\bX \setminus \bX_1, \cF) + \chi_{\bX_1}(\bX, \cF) \,, 
\end{equation}
provided all terms in \eqref{longlocal} are well-defined. While 
$\bX\setminus \bX_1$ is typically not proper, it can be cohomologically
proper for an action of a certain group $\bG$. In the equivariant
situation, we always assume that $\bX_1$ is invariant.

\subsubsection{}\label{s_weights_loc_coh} 

Local cohomology depends only on the neighborhood of $\bX_1$ in $\bX$ and
is particularly amenable to  computations if $\bX$ and  $\bX_1$ are smooth. If
$f_1,f_2, \dots$ are generators of the ideal of $\bX_1$ then the local
cohomology is computed by tensoring $\cF$ with the \v{C}ech complex
\begin{equation}
\cO_\bX \to \bigoplus_i  \cO_\bX \left[f_i^{-1}\right] \to  \bigoplus_{i,j}
  \cO_\bX \left[f_i^{-1},
    f_j^{-1}\right] \to \dots \,. \label{Cech}
\end{equation}
The Cech complex and the Chevalley-Eilenberg complex \eqref{ChE} are further examples of a DGAs (this time, positively
graded) that naturally appear our computations.

\subsubsection{}

If $\bX_1$ is smooth then
\begin{equation}
  \label{eq:9}
  \textup{K-theory class of \eqref{Cech}} = (-1)^{\codim \bX_1}
  \cO_{\bX_1} \otimes \det{N_{\bX/\bX_1}} \otimes \Sd N_{\bX/\bX_1} \,, 
\end{equation}
where
\[
  \Sd V = 
\bigoplus_{d\ge 0} S^d V
\]
denote the symmetric algebra of a vector space or a vector bundle $V$.

\subsubsection{}

Note that for any vector space $V$, we have the following equalities 
in the localization of $K_{GL(V)}(\pt)$
\begin{align}
  \frac{1}{\det(1-g^{-1})} & = \Sd V^* \label{SdVvee} \\
                                    & = (-1)^{\dim V} \det V \otimes
                                      \Sd V  \label{SdV}  \,. 
\end{align}
Further note the traces in the right-hand side of \eqref{SdVvee}
and \eqref{SdV} converge when
\[
  | \textup{eigenvalues of $g$} |  \gtrless 1 \,,
\]
respectively.

\subsubsection{}\label{s_oBB} 

\begin{Proposition}\label{l_lc}
  Suppose $\widetilde{\bG}$ acts on $\bX$ and 
  \[
    \bX_1 = \widetilde{\bG} \times_\bG \Attr(\bX_1^\sigma)\,, 
  \]
  where $\bG \subset \widetilde{\bG}$ is a subgroup
  containing $\sigma$ and the limits
  \[
    \pi_\bG(g) = \lim_{z\to 0} \sigma(z) g \sigma(z)^{-1} \in \bG^\sigma
  \,, 
  \]
  for all $g\in \bG$. If the action of
  $\bG^\sigma/\sigma$ on $\bX_1^\sigma$ is cohomologically proper and $N_{\bX/\bX_1}$ is $\sigma$-repelling
%
%
then
\begin{equation}
\chi_{\bX_1}(\bX, \psi(T\bX)) = \Ind^{\bG}_{\bG^\sigma} \, 
\chi(\bX_1^\sigma, \psi(T\bX_1^\sigma) \otimes
\Psi(N_{\bX/\bX_1^\sigma}))^\wedge
\,,\label{eq:5}
\end{equation}
where ${}^\wedge$ denotes the attracting lift of a distribution from
$\bG^\sigma/\sigma$. 
\end{Proposition}

\noindent
Here repelling means that it has limits to $\bX_1^\sigma$ under
the action of $\sigma(z)$ as $z \to \infty$. 

\begin{proof}
This is a combination of Proposition \ref{p_attr_lift} and the above
analysis of the local cohomology complex for $\bX_1$. 
\end{proof}




\subsubsection{}\label{s_BB}
Obviously, one can iterate this procedure for a sequence
\[
  \bX \supset \bX_1 \supset \bX_2 \supset \dots
\]
in which each $\bX_i$ is closed and
\[
  \bX_i \setminus \bigcup_{j>i} \bX_j \subset \bX \setminus \bigcup_{j>i} \bX_j
\]
is smooth.

Typical example of such stratification are Bia\l{}ynicki-Birula
stratification \cite{BB} coming from torus actions on $\bX$, as well as the
more general stratifications constructed for reductive group
actions in the GIT context in \cites{Bogo,Hess,Kempf,Ness,Rous}, see
e.g.\ Chapter 5 in \cite{VinPop} for an exposition.

\subsubsection{}

The main geometric input for this paper is contained in the elementary
observation that different $\bG$-invariant stratifications of $\bX$ will
lead to \emph{different} integral formulas for $\psi_\bG(\bX)$.

\subsubsection{Example}\label{s_exCC}

This can be seen already in the simplest example
\[
  (\bG,\bX) =  (\Ct ,\C) \,, 
\]
with the defining action.

On the one hand, for $\sigma(z) = z$ and $\bX_1 = \bX$, we saw in the
Example \ref{ex_C2} that
\[
  \psi_{\Ct}(\C) = 
 \Psi(x) \, \delta_{(1+\varepsilon)U(1)} \, .
\]
One the other hand, for $\sigma(z) = z^{-1}$ and $\bX_1 = \{0\}$, we get
from Proposition \ref{l_lc}
\[
  \psi_{\Ct}(\C) = \Psi(x) \, \delta_{(1-\varepsilon)U(1)} + \psi_{\Ct}(\Ct) \,. 
\]
Since
\[
  \psi_{\Ct}(\Ct)  = \psi(1) \, \delta_1 \,, 
\]
this is in agreement with the usual residue formula.


\subsubsection{}

Hopefully, applications related to the topic of this paper, as well as many
potential applications make a compelling case for the development of
the general cobordism theory for cohomologically proper virtually
smooth derived stacks, encompassing both positively and negatively
graded DGAs. 

\subsection{Equivariant cohomology}\label{s_limitH} 

\subsubsection{}

In this section, we will compare and contrast the constructions of
equivariant genera in K-theory and cohomology. To be able to
distinguish between the genera in the two theories, in this section,
we will stick to the following notation. We will denote by
\[
  x \in \Ct, \quad s \in \C = \Lie \Ct
\]
the elements in the corresponding groups, and by $\bpsi(x)$ and
$\psi(s)$ the values of the genera on $\C/\Ct$. We denote by
$\bG$ a connected reductive group acting on $\bX$, and by $\bpsi(\bX)$
the corresponding distribution on $\bG$. We denote by $\xi\in \fg$ an
element of $\fg = \Lie \bG$. 

Recall we assume that $\bpsi(x)$
is analytic and nonvanishing in some neighborhood of $U(1) \subset
\Ct$. The size of this neighborhood needs to be correlated with the
image in $\bG/\bG_\comp$ of those group elements at which we
want to evaluate  $\bpsi(\bX)$. Specifically, this involves the
element $q\in
\Ct_q$ which will be eventually evaluated to a prime power in the
function field context. 

In parallel, we require $\psi(s)$ to be analytic and nonvanishing on some neighborhood of
$i \R = \Lie U(1)$. We assume that at infinity $\psi(s)$ behaves so
that $\psi(\bX)$ is a well-defined linear functional on Paley-Wiener functions
$PW(\fg)$. In the case of interest to us, it suffices to assume that
the ratios $Z(x)/Z(qx)$ grow at most polynomially along $\Re x =
\const$.

\subsubsection{}

In contrast with K-theory, cohomology is a $\Z$-graded theory,
implying that for a smooth proper $\bX$ and a test function
$f(\xi)$, we have
\begin{equation}
\langle \psi(\bX), f(\xi) \rangle
\Big|_{\psi(s) \mapsto \psi(k s), f(\xi) \mapsto f(k \xi) } =
k^{\dim \bX/\bG} \langle \psi(\bX), f(\xi) \rangle\,, \quad k \in \R
\,. 
\label{eq:150}
\end{equation}
Since the argument $s$ is also an element in a certain
Lie algebra, namely the Lie algebra of the group defining the cohomology
theory, the scaling by $k$ in the RHS of \eqref{eq:150} scales
all Lie algebra elements by the same constant. This is the
infinitesimal form of raising all group elements to the
power $k\in \Z$, and thus \eqref{eq:150} may be interpreted as being an
eigenfunction of Adams operations.

Our goal in this section is to define distribution-valued genera
\begin{equation}
  \label{eq:152}
  \psi(\bX) \in PW(\fg)^* 
\end{equation}
for cohomologically  proper
derived stacks $\bX$ that are relevant to this paper, with the
same functoriality as $\bpsi(\bX)$ and with the homogeneity
\eqref{eq:150}. This will be done by a certain limit
procedure\footnote{In the context of Adams operations, this limit
  may be compared to the situation when a large power of an operator converges
  to the the projector on the maximal eigenvalue.} 
starting from $\bpsi(\bX)$. Readers who are only interested in
integral identies resulting from the geometric considerations, can
bypass the latter by taking the well-defined limit
\begin{equation}
  \label{eq:151}
  x = q^s\,, \quad q \downarrow 1 \,, 
\end{equation}
in the actual integral identities.

Specifically, the combination of our Theorems
\ref{t_f1f2} and \ref{t_L2} may be phrased as a specific contour
integral identity. Its $q\to
1$ limit, spelled out in Section \ref{s_spell_out} below, 
may be taken directly and independently of any geometric 
considerations.

\subsubsection{}

In the analytic category, indices of \emph{transversally elliptic}
operators \cite{Atiyah} naturally lead to distributions on compact
groups. For the analysis of such operators, as well as for applications to 
many other problems of geometric 
or combinatorial nature, a powerful theory of equivariant cohomology with
generalized coefficients has been developed. See the survey
\cite{Vergne} as well as \cite{BerlVerg, BravM, CPV2,  CPV,
 ParVerg} for a thin sample of further 
references.

Our treatment here is probably closest in spirit to \cite{CPV}, but
with emphasis on different generality and different tools. Like the
authors of  \cite{CPV}, 
we start on the Fourier transform side $\fg^*$. In place of de Rham
representative constructed using an action form,
we use an ample line bundle on $\bX$ as an auxiliary piece of data. 

\subsubsection{} 
As in Section \ref{s_K_an} we note that
the $\Hd_\bG(\bX)$ is finitely-generated over $\Hd_\bG(\pt)$. 
Therefore, in order to define $\psi(\bX)$,  it enough to define the push-forwards for
finitely many Chern classes of $T \bX$.

Slightly more generally, consider the push-forward
$\chi(\bX, \bbc)$, where $\bbc$ is a K-theory class of the form 
\begin{equation}
\bbc = \prod_i \bc_{k_i} (\cV_i)\,, \label{eq:166}
\end{equation}
where $\cV_i$ are vector bundles on $\bX$ and
\[
  \bc_k(\cV) = e_k(v_1-1,v_2-1,\dots) \,. 
\]
Here $e_k$ is the elementary symmetric function and $\{v_i\}$ are the
Chern roots of $\cV$.

\subsubsection{}

By definition, the Fourier transform of $\chi(\bX, \bbc)$ is given by the
multiplicities $\bbch$ in
\begin{equation}
  \label{eq:153}
\chi(\bX, \bbc)  = \sum_{\lambda} \,
  \bbch(\lambda) 
 \, 
  V^\lambda \,, 
\end{equation}
where $V^\lambda$ denotes the irreducible $\bG$-module with highest
weight $\lambda$. 
%
%
%
%
%
The virtual module in the LHS of \eqref{eq:153}
is the cohomology of a complex of finitely-generated
modules over finitely-generated
$\bG$-algebras, therefore the multiplicities $\bbch(\lambda)$ are
piecewise quasi-polynomial functions on the weight lattice.
By definition, the algebra of quasi-polynomials on a lattice is
generated by polynomials and characters. Here, only characters
of finite order appear, their order being bounded in terms of the degrees of
generators of the finitely-generated algebras above.

Integral
formulas for $\chi(\bX, \bbc)$ thus involve generating functions for these
quasi-polynomials
by the very definition of the Fourier transform. 

\subsubsection{}

The cohomology analog of \eqref{eq:166} is the class
\begin{equation}
  \label{eq:176}
  c= \prod_i c_{k_i} (\cV_i)\,, 
\end{equation}
where $c_k(\cV)$ are now ordinary Chern classes. The distribution
$\int_\bX c$, to be introduced shortly,
describes the $k\to \infty$ asymptotics
of the multiplicities in 
\begin{equation}
  \label{eq:155}
  \chi(\bX,\cL^k \otimes \bbc) = \sum_{\lambda} \,
  \bbch(\lambda,k) 
 \, 
  V^\lambda 
\end{equation}
where $\cL\in \Pic_\bG(\bX)$ is ample. Recall we assumed from the
beginning that such bundle $\cL$ exists. The ampleness assumption,
satisfied in all examples of interest to us, can
certainly be relaxed in general.

The $k\to\infty$ asymptotics will be
taken in the weak sense, that is, after pairing with a sufficiently
smooth slow-changing function. Equivalently, it may be described via 
the asymptotics of the Fourier transform near the origin.

In fact, since
\eqref{eq:155} can be computed on the cone over the corresponding
projective embedding of $\bX$, the coefficients $\bbch(\lambda,k)$
are piecewise quasi-polynomials in $(\lambda,k)$ and their weak asymptotics is
their polynomial part in regions that are unbounded in the $k$
direction.

\subsubsection{}

It is convenient to pick out this asymptotics by pairing $\chi(\bX,\bbc)$ with the
generating function
\begin{equation}
e^{t (\cL-1)} = e^{-t} \sum_{k\ge 0} \cL^k \,
\frac{t^k}{k!} \label{eq:157}
\end{equation}
and letting $t\to \infty$. 
Note that
\[
  e^{t(e^{i \phi}-1)} \approx e^{i t \phi} e^{-t \phi^2/2} 
\]
for $\phi$ near the origin of $\R$.
Analytically, pairing with \eqref{eq:157} is simply a convenient way
to introduce a Gaussian localization on both sides of the Fourier
transform, see Figure \ref{localGauss}. 

\begin{figure}[!h]
  \centering
 \includegraphics[scale=0.3]{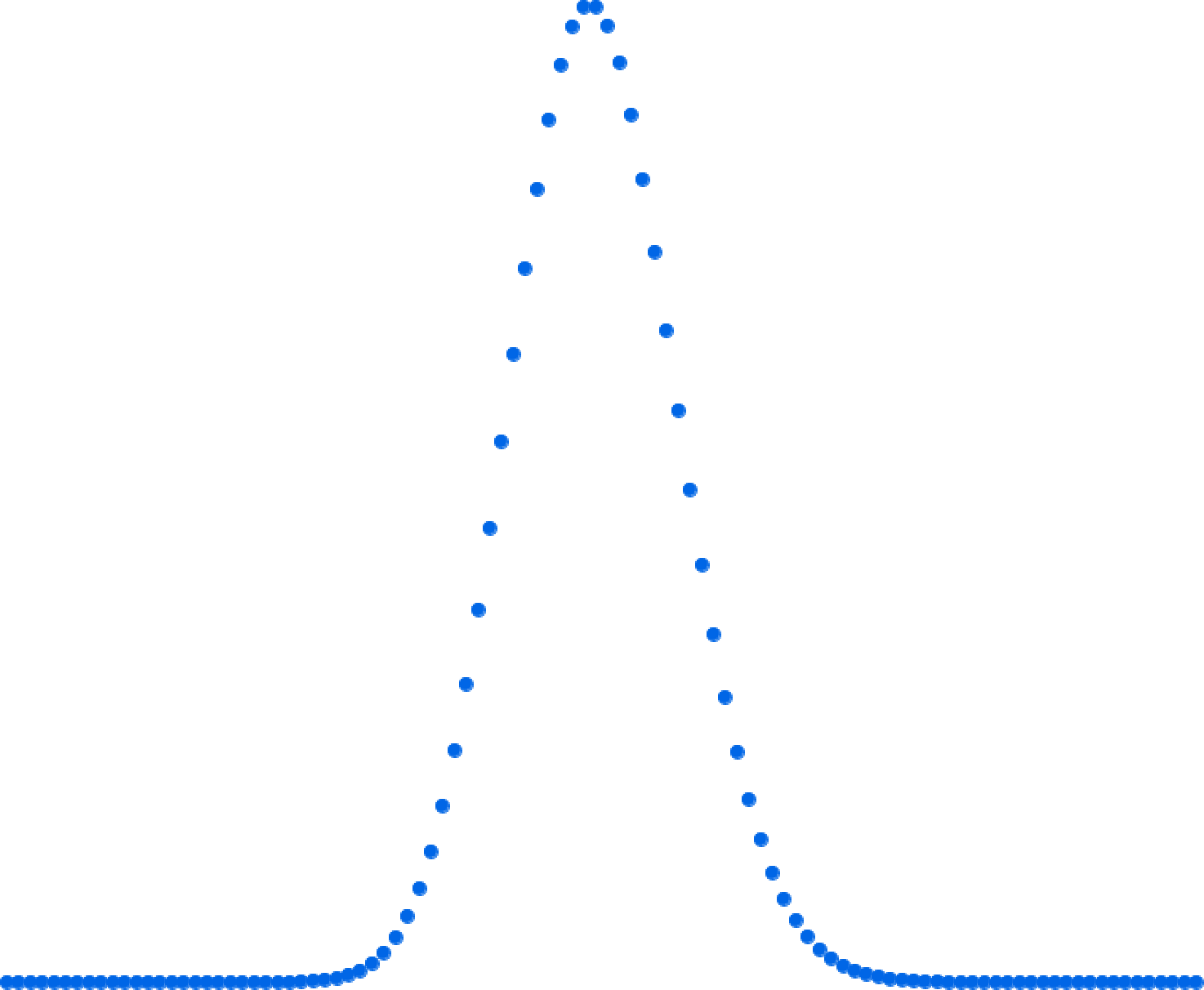}\includegraphics[scale=0.3]{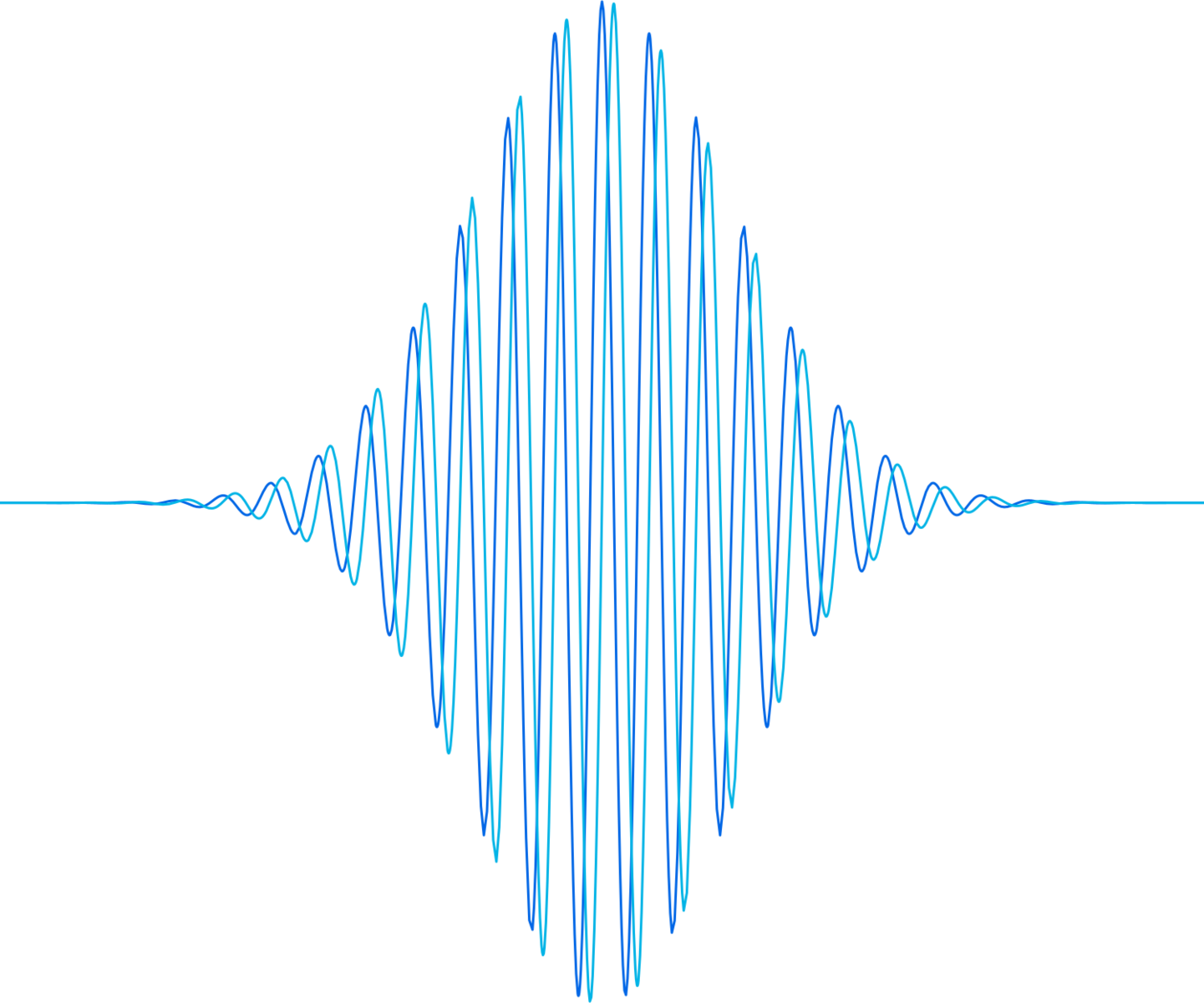}
  \caption{The Taylor coefficients of the function $e^{t(z-1)}$ and the
    plot of its real and imaginary part on the unit circle for large $t$. The former
    is the Poisson distribution with mean $t$ and standard deviation
    $\sqrt{t}$. The latter is localized on the $\frac{1}{\sqrt t}$
    scale. }
  \label{localGauss}
  \end{figure}

  \subsubsection{}

  Let $\boldsymbol{f}_t$ be a sequence of functions\footnote{
    Note that such sequence always exists. Indeed, for a
    univariate PW function $f(s)$ we can just set
    $\boldsymbol{f}_t(e^{s/t}) = \sum_{k\in \Z} g_t(s+2\pi i k t)$,
    where $g_t=f$. But even very fast growing functions can be
    approximated by such sum if we take e.g.\ 
    \[
      g(s) = f(s) \exp(- \alpha(t) \exp(\dots (\exp(-s^2))\cdots))\,, \quad \alpha(t)
    \to 0 \,. 
    \]
}
  such
that
\begin{equation}
  \label{eq:158}
  \boldsymbol{f}_t(e^{s/t}) \to 
      f(s) 
\end{equation}
on compact sets as $t\to \infty$.  We define $\int_\bX c$ by 
\begin{equation}
  \label{eq:155_}
  \left\langle \int_\bX e^{c_1(\cL)} c, f \right\rangle = \left[ t^{\vir \dim
      \bX/\bG-\sum k_i}\right]
  \left\langle \chi(\bX,e^{(\cL-1)t}  \, \bbc),  \boldsymbol{f}_t\right\rangle \,, 
\end{equation}
where bracket means we extract the corresponding term in the
asymptotic expansion. It is easy to see this is independent of $\cL$,
see e.g.\  Proposition \ref{p_int_limit} below.

Since $\bpsi(\bX)$ is the generating function for the equivariant
Chern classes of the tangent bundle, the equality \eqref{eq:155_}
implies
\begin{equation}
  \label{eq:155__}
  \left\langle \int_\bX e^{c_1(\cL)} \psi(T\bX), f \right\rangle = \left[ t^{\vir \dim
      \bX/\bG}\right]
  \left\langle \chi(\bX,e^{(\cL-1)t} \bpsi_t(T\bX)),  \boldsymbol{f}_t\right\rangle \,, 
\end{equation}
where $\bpsi_t(e^{s/t}) \to  \psi(s)$, and  where we assume that
$\psi(T\bX)$ grows at most polynomially in the imaginary direction. 

\subsubsection{Example $\bX=\bP^1$}

Consider $\bX=\bP^1$ with the standard action of $\Ct$. We denote by
$x$ and $s$ elements of $\Ct$ and $\Lie \Ct$, respectively. Since
$\bP^1$ is smooth and proper, we have
\begin{equation}
  \label{eq:83}
  \bpsi(\bP^1) = \frac{\bpsi(x)}{1-x^{-1}} +
  \frac{\bpsi(x^{-1})}{1-x}\,, \quad
  \psi(\bP^1) = \frac{\psi(s)-\psi(-s)}{s} \,, 
\end{equation}
by localization. More generally, if $\cL$ is a line bundle on $\bP^1$
with $\Ct$-weights $l_0,l_\infty\in \Z$ at the two fixed
points, then
\begin{align}
 \cL^k \bpsi(\bP^1) &= x^{k l_0}  \frac{\bpsi(x)}{1-x^{-1}} +
  x^{k l_\infty} \frac{\bpsi(x^{-1})}{1-x}\,, \notag \\ 
  e^{c_1(\cL)} \psi(\bP^1) &= \frac{e^{l_0 s}\psi(s)-e^{-l_\infty
                             s}\psi(-s)}{s} \,. \label{eq:ecL} 
\end{align}
Let us see how the formula \eqref{eq:ecL} appears from the limit
\eqref{eq:155__}.

Ampleness of $\cL$ is equivalent to $l_0 > l_\infty$. Then, by the
geometric series expansion, we have 
\[
  \widehat{\chi(\cL^k (T\bP^1)^{\otimes n})} = \delta_{[k l_\infty -n, k
  l_0 +n]}\,, 
\]
viewed as function on $\Z$. Thus in the $k\to \infty$ limit, we have
\begin{alignat}{2}
  \label{eq:160}
  \widehat{e^{c_1(\cL)}} &= 
  \delta_{[ l_\infty, 
  l_0 ]}&&=\lim_{k\to\infty} k^{-1} \widehat{\chi(\cL^k)} \Big|_{\textup{scale the
      argument by $k^{-1}$} }\,, 
  \\
  \widehat{c_1(T \bP^1)}& = \frac{\partial}{\partial l_0} -
  \frac{\partial}{\partial l_\infty}
  &&= \lim_{k\to \infty} k (\widehat{T\bP^1}-1) \,, \label{eq:160_}
\end{alignat}
where we interpret \eqref{eq:160_} as an operator acting on
\eqref{eq:160}. Upon Fourier transform, this becomes \eqref{eq:ecL}.

\subsubsection{Example $\bX=\C^1$ and $\bX=\Thom(0\to \C^1)$}
Note that if we take $\bX= \C^1 = \bP^1 \setminus \infty$, the same reasoning
as in the previous example will yield the distribution 
\[
  \widehat{e^{c_1(\cL)_{\C^1}}} = \delta_{[-\infty,l_0]} \Rightarrow
e^{c_1(\cL)_{\C^1}} = \frac{e^{l_0
    s}}{s+0} \,, 
\]
where the shift of the pole indicate the direction of the deformation
of the integration contour. In this case, the contour goes to
the region $\Re s > 0$, which is the infinitesimal analog of the
shift in Example \ref{ex_C2}.

For $\bX=\Thom(0\to \C^1)$,  we get
\[
  \widehat{\chi(\Thom,\cL^k)} = - \delta_{[k
  l_0 +1,\infty]} \,, 
\]
and thus
\[
  \widehat{e^{c_1(\cL)_{\Thom}}} = - \delta_{[l_0,\infty]} \Rightarrow
e^{c_1(\cL)_{\Thom}} = \frac{e^{l_0
    s}}{s-0} \,. 
\]
Since
\[
  \bP^1 = \C^1 + \Thom(\infty \to \bP^1)
\]
and $s$ acts on $T_\infty \bP^1$ with weight $-1$, this agrees with
the previous computation.

\subsubsection{}
In this paper, we frequently use integral formulas that express
$\bpsi(\bX)$ as integrals over certain cycles $\gamma \subset \bA$
in the maximal torus of $\bG$ with density built from the functions
$\bpsi(x)$ and $(1-x^{-1})$. From such  formulas, one can easily derive
an integral formula for $\psi(\bX)$ by making the change of variables
prescribed in the definition \eqref{eq:155__} and letting $t \to
\infty$.

Note that in this limit, one should send all Lie groups elements to
the identity at comparable rates. Since we use $q\in \Ct$ as an
important parameter in this paper, we may parametrize the limit 
in \eqref{eq:155__} by
\[
  t = \left(\ln q\right)^{-1} \,, \quad q \downarrow 1 \,.
\]
This means we replace the exponential map by the map $\xi \mapsto
q^\xi$.

\subsubsection{}

\begin{Proposition}\label{p_int_limit}
  Integral formulas for the  K-theoretic genus $\bpsi(\bX)$
  and the cohomological genus $\psi(\bX)$ are related by
  the correspondence in Table \ref{tab1}. 
\end{Proposition}

\begin{table}[!h]
  \centering
  \small
\begin{tabularx}{0.8\textwidth}
 { 
  | >{\centering\arraybackslash}X 
  | >{\centering\arraybackslash}X 
  | >{\centering\arraybackslash}X | }
 \hline
 & $\chr\F >0$ & $\chr\F=0$ \\
 \hline
  cohomology theory  & K-theory  & cohomology  \\
 \hline 
 $\cQ_q^\vee\cong $ & $(\Ct,\, \cdot \,,1)$  & $(\C,+,0)$\\
 \hline
 Haar measure & $ (2\pi i x)^{-1}  dx$& $ (2\pi i)^{-1}  d s$ \\
 \hline 
$\Psi=$  &  $\bpsi(x)/(1-x^{-1})$ & $\psi(s)/s$\\
\hline
$Z=$  &  $\Psi(1/x) \Psi(x/q)$ & $\Psi(-s) \Psi(s-1)$ \\
\hline
poles of $Z$ &  $x=1,q$  & $s=0,1$ \\
\hline
integration cycle & $\gamma \subset \bA$ & $\log_q(\gamma)_0 \subset
\fa$ \\
\hline 
\end{tabularx}
\caption{Conversion of formulas between the cases of finite and zero
  characteristic. The Haar measure $(2\pi i)^{-1}  d s$ on Lie algebras of
  compact tori is normalized
  so that $\Ker \exp$ is a unimodular lattice. The subscript
  $\log_q(\gamma)_0$ means the component that does not escape to
  infinity as $q\to 1$. The cohomology column is the limit of the 
K-theory column in the sense explained in this Section.}
\label{tab1}
\end{table}

\begin{proof}
Substitute $g=q^\xi$ in the integral formula for $\langle \bpsi(\bX),
e^{(\cL-1)t} \, \boldsymbol{f}\rangle$. By definition, this transforms
$\bpsi(x)$ and $\boldsymbol{f}(g)$ into the corresponding $\psi(s)$
and $f(\xi)$, while affecting the other terms in the formula by
\begin{equation}
  \label{eq:156}
  \frac{dx}{2\pi i x} \mapsto \ln q \, \frac{ds}{2\pi i}\,, \qquad 
  \Psi(x) \mapsto \frac{1}{\ln q} \frac{\psi(s)}{s} \,.  
\end{equation}
The difference between the
numbers of terms each kind is precisely minus the virtual dimension of
$\bX/\bG$, thus we get the correct power of $t$ as a prefactor.

The line bundle $\cL$, being a function on the spectrum of K-theory,
is represented by a regular function of the form
\[
  \cL \big|_{q^\xi} = 1 + (\ln q ) c_1(\cL)\big|_{\xi} + \dots
\]
in the integral formulas. Therefore, $e^{(\cL-1)t}\to e^{c_1(\cL)}$ as
$q\to 1$. 

The integration cycles $\gamma$ appearing in K-theoretic formulas have
the form
\[
  \gamma = q^{\eta} \, \bH_\comp \big / \textup{conjugation}\,, 
\]
where $\bH \subset \bG$ is a reductive group, not necessarily
connected, which commutes with $\eta$. Since
\[
  \textup{dist}_\bG(1, \bH_\comp \setminus \bH_{\comp,0}) >
\const > 0 \,, 
\]
the only component of $\log_q(\gamma)$ that does not escape to infinity
as $q\to 1$ is
\[
  \log_q(\gamma)_0 = \left(\eta+ \fh_\comp \right) \big/
\textup{conjugation}\,.
\]
This concludes the proof.
\end{proof} 

As a corollary of the proof, we note that in cohomology we always
integrate over a translate of a Lie algebra of a subtorus in $\bA_\comp$.

\section{The stack $\Tau$}\label{s_stack} 

\subsection{Basic properties}

\subsubsection{}
A central role in this paper will be played by the following derived quotient
stack
\begin{align}
  \label{defTau1}
  \Tau &= T^* (\bB \backslash \bG /\bB) \\
       & = T^* \cB \rd \bB \label{defTau2}\\
       & = \{ (e,\fn') , e \in \fn \cap \fn'\} / \bB \label{defTau3}\\
       & = \{ (e,\fn',\fn''), e \in \fn' \cap \fn''\} /
         \bG \label{defTau4} \\
      & = \left(T^*\cB \times_\fg T^*\cB \right) / \bG \label{defTcT} \,.
\end{align}
Here $\rd \bB$ in \eqref{defTau2} denotes the algebraic
symplectic reduction, that is, $\mu_\fb^{-1}(0)/B$, where
$\mu_\fb^{-1}(0)$ is the derived preimage of $0\in \fb^*$
under the moment map. We denote by $\fn$ the nilradical of
$\fb = \Lie \bB$ and we denote by $\fn',\fn''$ pairs of points of
$\cB$, interpreted as  
maximal nilpotent subalgebras of $\fg$, see Appendix \ref{a_Lie}.

The equality between \eqref{defTau1} and
\eqref{defTau2} is the definition of the cotangent bundle to a
quotient stack. The equivalence of \eqref{defTau2} and \eqref{defTau3} follows from
the identification of the Springer resolution
\begin{equation}
  \label{Springer}
  \xymatrix{  (e, \fn')\, \ar@{|->}[d]\ar@{^{(}->}[r]& T^*\cB \ar[d] \ar[dr]^{\mu} \\
    e\, \ar@{^{(}->}[r] & \textup{nilpotent cone} \,\ar@{^{(}->}[r] & \fg
    }
  \end{equation}
with the algebraic moment map $\mu$ for $\bG$. Here, and in
\eqref{TsB}, we use the identification
$\fg \cong \fg^*$ provided by an invariant bilinear form.

We call the stack 
 $\Tau$ the \emph{Springer stack}.

\subsubsection{}

Via the presentation \eqref{defTau1}, we view $\psi(\Tau)$ as a
distribution on $\bA \times \bA$. We will compute it
$\Ct_q$-equivariantly, where $\Ct_q$ scales the cotangent
direction with weight $q^{-1}$.

\subsubsection{}\label{s_tan_Tau} 

{}From \eqref{defTau4} we conclude that
\begin{equation}
  \label{decTau1}
  \Tau
  = \bigcup_{e/\textup{conj}} \left(\cB^e
  \times \cB^e \right) / \bC_e 
\end{equation}
where the union is over conjugacy classes of nilpotent elements of
$\fg$, $\cB^e$ is 
is the Springer fiber \eqref{Spri} of $e$, and $\bC_e= \bG^e$ is the centralizer
of $e$. Compare with the formula 
\eqref{eq:11c} below.

\begin{Lemma}\label{l_Ttau}
  We have the following equivalent expressions for the tangent bundle
  of $\cT$
\begin{align}
  \label{eq:11}
  T \cT \big|_{(e,\fn,\fn')} & = \underbrace{(1+q^{-1})
                               \fg}_{T_{(T^*\bG)}}-
                               \underbrace{(\fb+ q^{-1}
                               \fb^*)}_{\rd \bB}
                               - \underbrace{(\fb' + q^{-1}
                               {\fb'}^* )}_{\rd \bB'} \\
                             &  = \underbrace{(\fn^* - q^{-1} \fn^*)}_{T_\fn \cB^e} +
                               \underbrace{ ({\fn'}^* - q^{-1}
    {\fn'}^*)}_{T_{\fn'} \cB^e} + \underbrace{(q^{-1}\fg-
                               \fg)}_{T_e(\fg/\bG)} - 2 q^{-1} \fa
  \label{eq:11c}\\
      & = \underbrace{q^{-1} \fn -
        \fb}_{T(\fn/\bB)}  + \underbrace{\left(\fn'\right)^* -q^{-1}
        \left(\fb'\right)^*}_{\textup{$\mu$-relative tangent}} \,.\label{eq:11b}
\end{align}
\end{Lemma}

\noindent
Here $\fa = \Lie \bA$ and the virtual tangent bundle
\begin{equation}
T_\fn \cB^e = \fn^* - q^{-1} \fn^*\label{eq:127}
\end{equation}
describes $\cB^e$ as the 
zero locus of the vector field $e$, see Appendix
\ref{s_App_Spring}.

\begin{proof}
  Formula \eqref{eq:11} follows directly from \eqref{defTau1},
  while the other two formulas are just restatements of
  \eqref{eq:11}. 
\end{proof}

\subsubsection{}

\noindent
Via the presentation \eqref{defTcT}, we can view $\psi(\cT)$ as a
distribution on
\[
  \bA \times \bA \times \Ct_q \times \bG \,. 
\]
The stabilizers of the nilpotent elements $e$ in the last two factors
have the form
\begin{equation}
  \bCh_e = \{  (q, q^{h/2} \bC_e) \} \subset \Ct_q \times
  \bG\label{eq:28}\,, 
\end{equation}
$h$ being the characteristic of $e$,
and the various $\bCh_e$  will correspond to the pieces of the eventual spectral
decomposition.
The integration over one of the $T^*\cB$ factors in
\eqref{decTau1} will give the corresponding spectral projectors.  





\subsection{Characteristic classes of $\Tau$}\label{s_charT} 

\subsubsection{}

First, we recall the computations of
the $\bB$-equivariant characteristic classes of
\[
  M=T^*\cB \,. 
\]
This is the same as $\bA$-equivariant characteristic
classes for a maximal torus $\bA\subset \bB$.

\subsubsection{}

We proceed by localization. The $\bA$-fixed points on $M^\bA$ are
isolated and indexed by the Weyl group $W$. We will index the fixed points so that the smallest
Bruhat cell corresponds to the longest element of the Weyl group. This
means we take
\begin{equation}
M = T^*(\bG / \bB_{-}) \,, \label{eq:76}
\end{equation}
where $\bB_{-}$ and $\bB_{+}=\bB$ is a pair of opposite Borel subgroups, and set
\[
  p_w = w \bB_{-}  \in \bG / \bB_{-} \subset M \,.
\]


\subsubsection{}

The positive roots of $\bG$ will be, by definition, those contained in
$\bB_+$. 
\[
  T_{p_w} M = w \cdot T_{p_1} M = w \left( \sum_{\alpha > 0}  (\alpha + q^{-1} / \alpha)
\right) \,. 
\]

\subsubsection{}\label{s_f1f2}

Recall that we view $\psi(\cT)$ as a distribution on $\bA \times \bA$,
where the two copies act by the left and right regular action on
$\bG$. The point $w\in \bG$ is fixed by a subtorus
\[
  \bA \cong \textup{stabilizer}(w) = \{(a,w^{-1} \cdot a)\} \subset
\bA \times \bA \,.
\]
Thus, pairing the distribution with a test function
\[
  f_1 \boxtimes f_2 \in K_{\bA \times \bA}(\pt)
\]
means integrating $f_1(a) f_2(w^{-1} a)$ over $\bA$.

Put it differently, the identification $f_2 \in K_\bA(\pt) = K_{\bB}(\pt) =
K_\bG(\bG/\bB)$ takes a $\bB$-module $V$ to a vector bundle
$\cV = \bG \times_\bB V$ over $\bG/\bB$. The torus $\bA$ acts on
its fiber 
\[
  \cV_{p_w} = w \bB \times_\bB V \cong V 
\]
by the action of $w^{-1} \cdot a$ on $V$, whence the term $f_2(w^{-1} a)$ above. Here the choice of the Borel subgroup
containing $\bA$ is not material, which is why we have dropped the subscript
in $\bB_-$.

\subsubsection{}

The tangent bundle to $\Tau$ has the following description: 
\begin{align}
  T \Tau &= T \left(M \rd \bB_+\right) \notag \\
  & = TM - \fb_+ - q^{-1} \fb_-  \,. \label{TTau} 
  \end{align}
Therefore,
\begin{align}
  \label{eq:51}
  T_{p_w} ( T^*(\bB_+ \backslash \bG / \bB_-)) 
  & = 
  \sum_{\substack{\alpha > 0 \\ w^{-1} \alpha <0 }} (\alpha^{-1} -
  \alpha)(1-q^{-1}) - (1+q^{-1}) \fa\,. 
\end{align}
and hence 
\begin{equation}
  \psi(\Tau) \big|_{p_w} = \frac{1}{\cz^r}  \prod_{\substack{\alpha > 0 \\ w^{-1} \alpha <0 }}
\frac{Z(x^\alpha)}{Z(q x^\alpha)}\,.  \label{eq:181-3}
\end{equation}
Where the function
\begin{equation}
  \label{eq:114bis}
Z(x) = \Psi(x^{-1}) \Psi(x/q)
\end{equation}
was introduced in \eqref{eq:114}. Here and in what follows we drop the
subscript in $Z_\psi$ for brevity. We also define
\begin{equation}
   \cz= \psi(1) \Psi(q^{-1}) \,.  \label{eq:23}
\end{equation}
%


\subsubsection{}\label{s_t_f1f2}

\begin{Theorem}\label{t_f1f2}
  The stack $\cT$ is cohomologically proper. 
  For a function $f_1 \boxtimes f_2$ on $\bA \times \bA$, we 
  have
  \begin{equation}
  \left(f_1 \boxtimes f_2, \psi(\cT)\right) =
\frac{1}{\cz^r} \int_{\varsigma A_\comp} d_\textup{Haar}
  \, 
   \sum_w f_1(x) \, f_2(w^{-1} x) \prod_{\substack{\alpha > 0 \\ w^{-1} \alpha <0 }}
   \frac{Z(x^\alpha)}{Z(q x^\alpha)} \,,\label{eq:3}
 \end{equation}
   for any  $\varsigma \in \bA$ such that
   \[
     \forall \alpha >0 \,, \quad |\varsigma^\alpha|  > |q|\,. 
  \]
  \end{Theorem}

  \begin{proof}
    Consider the diagram: 
  \begin{equation}
    \label{TaunB}
    \xymatrix{ \Tau \ar[d] &\mu_\fb^{-1}(0) \ar[r] \ar[l] \ar[d] & T^*\cB \ar[d]^\mu\\
       \fn/\bB & \fn \ar[r] \ar[l] & \fg 
    }
  \end{equation}
  The vertical maps in this diagram are proper. Therefore, it is enough
  to show the cohomological properness of $\fn/\bB$. 

We observe that $\fn$ is a linear space, $\fn^\bB = \{0\}$, and the normal
$\bT$-weights at $0$ have the form $\{q^{-1} x^\alpha\}$, where
$\alpha >0$. With our assumption,
all of these weights are attractive on $\varsigma
A_\comp$. The cohomological properness thus follows.

The RHS in \eqref{eq:3} is the computation of the genus $\psi(\cT)$ by
equivariant localization. The Weyl group elements $w$ index the fixed
points. The weight of $\psi(\cT)$ at each fixed point were determined
in \eqref{eq:181-3}.  The pairing of the corresponding distribution with a test
function $f_1 \boxtimes f_2$ was described in Section \ref{s_f1f2}. 
\end{proof}

To shorten the notation, one is tempted to  replace $w$ by $w^{-1}$
in \eqref{eq:3}. We leave it as is because in this way we have
the action of $w$ on parts of the integrand.

\begin{Corollary}\label{c_T3}
  Theorem \ref{t_T3} holds for a certain
  normalization of measures. 
\end{Corollary}

\noindent 
We discuss the precise normalization of measures in Appendix 
\ref{s_pT3}.

\subsection{Stratification of $\cT$}

\subsubsection{}

We now consider the stratification of $\cT$ from \eqref{decTau1}. On
each piece, we will consider the action of the stabilizer $\bCh_e$
from \eqref{eq:28}. We begin with 
\begin{equation}
\Slice_e \fg = e + \fg^{e_-}\label{eq:125}\,, 
\end{equation}
where
\[
  e_- = \phi \left(
  \begin{bmatrix}
    0 & 0 \\ 1 & 0 
  \end{bmatrix}
\right)\,. 
\]
Concerning the action of $(q,q^{h/2})  \in \bCh_e$, we
have the following 

\begin{Lemma}\label{l_slice} 
The weights of the $(q,q^{h/2})$-action on the slice
\eqref{eq:125} are repelling. 
\end{Lemma}

\begin{proof}
The $h$-weights of $\fg^{e_-}$ are nonpositive by the representation
theory of $\mathfrak{sl}_2$, and we have the additional factor of $q^{-1}$
coming from the scaling action of $\Ct_q$. 
\end{proof}

\subsubsection{}
We now consider the $\bCh_e$-action on the
fibers of projection to $\fg$ and note that the $\Ct_q$ acts trivially
on these fibers. We have the following classical
result of De Concini, Lusztig, and Procesi. 

\begin{Theorem}[\cite{dCLP}]\label{t_dCLP} The fixed loci and the attracting
  manifolds of the $q^{h/2}$-action on $\cB^e$ are smooth.
\end{Theorem}

\subsubsection{Proof of Theorem \ref{t_proper}}

Stratify both copies of $\cB^e$ by attracting manifolds of the
$q^{h/2}$-action. Each piece is smooth, and the normal weights to it
are repelling in the base $\fg$ by Lemma \ref{l_slice} and in the
fiber by construction. Since the fixed loci are proper, and hence
cohomologically proper, we can
apply the results of Sections \ref{s_oBB} and \ref{s_BB} to this
stratification, proving the claim. \qed

\subsection{Explicit formulas}

\subsubsection{} 

By its definition \eqref{eq:142}, the bordism class of $\cB^e$ is independent of $e$
and well-defined $\bG$-equivariantly. Its characteristic classes can
be easily computed by equivariant localization. To set it up, we fix a Borel subgroup $\bB \subset \bG$ with a maximal
torus $\bA \subset \bB$ and define
\begin{equation}
  \label{eq:145}
  \cP_\bB f = \left(\sum_{w \in W} w \right)\,
  \bPi_\bB \, f \,, \quad 
  \bPi_\bB = \prod_{\alpha \notin \Lie \bB}
  \frac{\Psi(x^{\alpha})}{\Psi(q^{-1} x^{\alpha})} \,. 
\end{equation}
From definitions and \eqref{eq:142}, we obtain the following 

\begin{Proposition}
  We have
  \begin{equation}
    \chi( f \otimes \psi([\cB^e])) = \cP_\bB f \in K_\bG(\pt)_\an\,, 
\label{e_cPbB}
  \end{equation}
  where we use the identification $f\in K_\bA(\pt) = K_\bG(\bG/\bB)
  \cong K_\bG(\cB)$ \,. 
\end{Proposition}

\noindent
In particular, the operators $\cP_\pm$ from \eqref{eq:141} correspond
to the two Borel subgroups $\bB_\pm$ used in Section
\ref{s_charT}.

\subsubsection{}\label{s_W(e)} 
For a given nilpotent $e$ we will be interested in the restriction
$\left(\cP_\pm \right)|_{\bCh_e}$ where $\bCh_e$ is the
stabilizer \eqref{eq:28}.

We recall from Appendix \ref{A_Bh} that $W^h \backslash W$ cosets
parametrize connected components of $\cB^h$. We define
\[
  W(e) = \left\{ w \in W \, \big| \, \cB^e  \cap \{
  \textup{corresponding component of $\cB^h$} \} \ne
\varnothing \right\} \,.
\]
This set varies from $W(0) = W$ to
\[
  |W( \textup{regular nilpotent})| = 1 \,.
\]

\begin{Lemma}\label{l_We} 
  \begin{equation}
    \label{eq:79}
    \Big(\cP_\pm \Big)\Big|_{\bCh_e} =
    \Big( \sum_{w \in W(e)} w \,
  \bPi_\bB \, f \Big) \Big|_{\bCh_e} \,. 
  \end{equation}
\end{Lemma}

\begin{proof}
In the equivariant localization computation, the 
contribution of those components of $\cB^h$ that do not meet
$\cB^e$ vanishes in $\bCh_e$-equivariant K-theory. 
\end{proof}

\subsubsection{}\label{s_Psie}

Consider the Lie algebras
$\fc_e = \Lie \bC_e$ and
its $h$-grading 
\[
  \fc_e = \fc_{e,0} \oplus \fc_{e,>0}\,, \quad
\fc_{e,0}  = \fc_\phi = \Lie \bC_\phi \,. 
\]
The $\fc_{e,>0}$ term above is the Lie
algebra of the unipotent radical of $\bC_e$. We define 
\begin{equation}
  \label{eq:139}
  \Psi_e = \frac1{\Psi(q^{-1})^{2r}}
  \frac{\Psi(\Slice_e q^{-1} \fg)}{\psi(\fc_\phi)
    \Psi(\fc_{e,>0})} \,. 
\end{equation}
Since
\begin{equation}
  \label{eq:1113}
  q^{-1} T \Slice_e \fg - \fc_e = (q^{-1} -1 ) \fg 
\end{equation}
the density $\Psi_e$ is always the restriction to $\bCh_e$ of a
$\bG$-invariant quantity. 

\begin{Theorem}\label{t_spectral_Psi} 
  We have
  \begin{equation}
    \label{eq:137}
    (f_1 \boxtimes f_2, \psi(\cT)) =
    \sum_{\substack{{\textup{nilpotent elements $e\in \fg$ \quad}}\\
        \textup{up to conjugation}}} 
    \int_{q^{h/2} \bC_{\phi,\comp}} \left(\cP_+ f_1 \right)
    \left(\cP_{-} f_2 \right) \Psi_e \, 
      d_\textup{Haar}  \,. 
    \end{equation}
\end{Theorem}

\begin{proof}
We apply Theorem \ref{t_proper}, and specifically equation
\eqref{eq:120}. The result of pushforward along two copies of
$\cB^e$ is given by $\left(\cP_+ f_1 \right)
\left(\cP_- f_2 \right)$. The contribution of the slice is
the numerator in \eqref{eq:139}. Finally, we compute the
$\bC_e$-cohomology as described in Section \ref{s_push_uni}. 
\end{proof}

\subsubsection{}

Since \eqref{eq:137} is an equality between distributions having
compact support, by analytic continuation we conclude the following

\begin{Corollary}\label{c_ac} 
The equality \eqref{eq:137} holds for functions $\psi$ analytic and
nonvanishing 
in the annulus $1/R_\bG < |x| < R_\bG$, where $R_\bG$ is a
sufficiently large number that depends on $\bG$. 
\end{Corollary}

\subsection{Example: zero nilpotent}

\subsubsection{}
In the case $(h,e) = (0,0)$ we have
\[
  \cB^e = \cB\,, \quad \bC_\phi = \bG \,, 
\]
and thus the integration is over 
\[
  (\cB)^{\times 2} /\bG 
\cong \bB \backslash \bG / \bB
\,.
\]
By definition of $\cT$, we have 
\begin{align}
  \label{eq:78}
  T \cT \big|_{(0, \fn, \fn')}
                               & = \underbrace{(\fn^* + \fn'^* - \fg)}_{T (\bB \backslash \bG / \bB)} + q^{-1}
\underbrace{(\fn+\fn' - \fg)}_{T^*(\bB \backslash \bG / \bB)} \,. 
\end{align}
Therefore, the $e=0$ term in \eqref{eq:137} is given by
the same integral as in \eqref{eq:3}, except the cycle of the
integration is the \emph{compact} torus.

\subsubsection{}
We now trace this equality through formulas. We have
\[
  \Slice_0 \fg = \fg, \quad \fc_\phi = \fg \,,
\]
and therefore 
\[
  \Psi_e = \frac{\Psi(q^{-1} \fg)}{\psi(\fg)} \,.
\]
Echoing \eqref{eq:1113}, we note that $\Psi_e$  is $W$-invariant. The
Weyl integration formula gives 
\begin{align}
  \int_{\bG_{\comp}} \left(\cP_+ f_1 \right) 
    \left(\cP_- f_2 \right) \Psi_\phi \, 
    d_\textup{Haar}  = &\notag \\
    = \frac{1}{|W|}
    \frac{1}{\psi(1)^r \Psi(q^{-1})^r}
    &\int_{\bA_{\comp}} \left(\cP_+ f_1 \right)
    \left(\cP_-f_2 \right) \prod_{\alpha\ne 0} \frac{\Psi(q^{-1}
      x^\alpha)}
    { \Psi(x^{\alpha})}\, 
    d_\textup{Haar} \notag \\
    = 
    \frac{1}{\psi(1)^r \Psi(q^{-1})^r}
    &\int_{\bA_{\comp}} f_1 
    \left(\cP_- f_2 \right)
\prod_{\alpha> 0} \frac{\Psi(q^{-1}
      x^\alpha)}
    { \Psi(x^{\alpha})}
    \, 
    d_\textup{Haar} \notag \\
    =
    \frac{1}{\psi(1)^r \Psi(q^{-1})^r}
    &\int_{\bA_{\comp}} f_1 
    \left(\cP_\Langl \, f_2 \right)
    d_\textup{Haar}\,,  \label{eq:140}
  \end{align}
where we used the $W$-invariance in going from the second line to the
third, and \eqref{eq:1182} to pass from the third line to the fourth. Clearly, this is the
integral \eqref{eq:3} taken over the compact torus.

\subsection{Example: regular nilpotent}\label{s_ex_reg}

\subsubsection{}

In this case we have
\[
  h= 2 \rho^\vee \,, \quad \fc_\phi = 0 \,, \quad W^h = \{1\} \,,
\]
and $\bC_\phi$ is the center of $\bG$ 
\[
  \bC_\phi = \bC(\bG) =\left\{ x \in \bA \, \big|\,  \forall \alpha\,,
  x^\alpha =1 \right\} \,.  
\]

\subsubsection{}

{}We have 
\[
  \alpha(h/2) = \hht(\alpha) \,, \quad \alpha > 0 \,, 
\]
where $\hht(\alpha)$ denotes the height of a positive root.
For negative roots, we define height by the above equality. 
We see that 
\begin{equation}
  \label{eq:102}
  \Big(w \bPi_\pm \Big)\Big|_{\bCh_e} =
  \begin{cases}
    \displaystyle{\prod_{\alpha<0} \dfrac{\Psi(q^{\hht(\alpha)})}
      {\Psi(q^{-1+\hht(\alpha)})}}\,, & \textup{$w=1$ for $\bP_+$, 
      $w=w_0$ for $\bP_-$}\,,
    \textup{resp.}\,, \\
    0 \,, & \textup{otherwise} \,, 
  \end{cases} 
\end{equation}
where $w_0$ is the longest element. By Lemma \ref{l_We}, this 
reflects the fact that $\cB^e$ is a fat point and meets only one
point in $\cB^h$.

\subsubsection{}

The classical identity, see e.g.\ \cite{Kostant, Macdonald, Steinberg}, 
\begin{equation}
  (1-q^{-1}) \sum_{\alpha > 0} q^{\hht(\alpha)} =
  \sum_{\textup{exponents $m_i$}} q^{m_i} - r  
\label{eq:14}
\end{equation}
connecting the heights of the positive roots and the exponents of a
  Lie algebra $\fg$ implies that 
  \begin{align}
\displaystyle{\prod_{\alpha<0} \dfrac{\Psi(q^{\hht(\alpha)})}
  {\Psi(q^{-1+\hht(\alpha)})}} &=
\Psi\left( (1-q^{-1}) \sum_{\alpha > 0} q^{-\hht(\alpha)}
  \right) \notag \\
    &=\Psi\left( - \sum_{\alpha > 0} q^{-m_i-1} + r q^{-1} 
  \right) \label{eq:13} 
  \end{align}

\subsubsection{}
Similarly, we observe that $\Psi_e$ equals $\Psi(\dots)$ where
dots stand for the following expression 
\begin{align}
(q^{-1} - 1) \fg \Big|_{q^{h/2}} - 2 r q^{-1} 
& = - (1+q^{-1}) r + (q^{-1} - 1) \sum_{\alpha \ne 0} 
                                   q^{\hht(\alpha)} \notag \\
  & = \sum \left(q^{-m_i-1} - q^{m_i}\right) - 2 r q^{-1} \,. 
\label{eq:15}
\end{align}
This means that
\begin{equation}
\left(\cP_+ f_1 \right) \left(\cP_- f_2 \right) \Psi_e
\Big|_{x \in q^{h/2} \bC_e}=
\frac{1}{\prod Z(q^{m_i+1})}  f_1(x) f_2(w_0 x)  \label{eq:17}
\end{equation}
and
\begin{equation}
  \label{eq:16}
  \int_{q^{h/2} \bC_{\phi,\comp}} \textup{\eqref{eq:17} }
  \, 
  d_\textup{Haar} =
  \frac{1}{|\bC(\bG)| \prod Z(q^{m_i+1})}
  \sum_{x\in \bC(\bG)}
  f_1(q^{\rho^\vee} x)  f_2(q^{-\rho^\vee} x) \,, 
\end{equation}
where we have used the fact that $w_0$ takes $\rho^\vee$ to
$-\rho^\vee$ and fixes the elements of $\bC(\bG)$.

\subsubsection{}
Formula \eqref{eq:17}
  may be interpreted as picking up the residues in \eqref{eq:3} at the
  $x^\alpha = q$ 
  poles of the factor
  \begin{equation}
  \frac{1}{\psi(1)^r \Psi(q^{-1})^r} \prod_{\textup{simple $\alpha>0$}}
  Z(x^\alpha) \prod \frac{dx_j}{2\pi i x_j}\,.\label{eq:26}
\end{equation}
  The simple roots $x^\alpha$ may be chosen as \'etale coordinates on
  $\bA$ and the factor $|\bC(\bG)|^{-1}$ appears as the Jacobian of
  this transformation. After that, each residue in \eqref{eq:26} equals one since
  \[
    Z(x) = \frac{\psi(1) \Psi(q^{-1})}{1-q/x} + O(|x-q|) \,.
  \]

 \section{Spectral Decomposition}

 \subsection{Symmetric dependence on $\Psi$}

 \subsubsection{}

 Since the function $\Psi(x)$ appears in Theorem \ref{t_f1f2} only
 through the function $Z(x)$, it is worthwhile to recast Theorem
 \ref{t_spectral_Psi} in a symmetric form.

 To avoid creating matching pairs of zeros and poles, we introduce
 the entire invertible function
 \begin{equation}
   \label{eq:118_Z1}
   Z^1(x) = (1-1/x)(1-q/x) Z(x) = - x^{-1} \psi(1/x) \psi(x/q) \,.
 \end{equation}
 Eventually, we will allow the function $Z$ to have zeros in the
 critical annulus, in which case $Z^1(x)$ will be an entire function
 with those zeros. In particular, in the geometric case, it will
 constitute the contribution of $H^1$ of the curve
 to the $\zeta$-function, whence the notation.

 \subsubsection{}

 We define
 \begin{equation}
   \label{eq:24}
   \cP^Z_{\pm}  = \left(\prod_{\alpha \ne 0} \psi(q^{-1}
  x^\alpha) \right) \cP_\pm  = \left(\sum_{w \in W} w \right)\,
   \prod_{\alpha \gtrless 0 } \frac{1-q x^\alpha}{1-x^{-\alpha}} Z^1(x^\alpha) 
   \,. 
\end{equation}
and
\begin{align}
Z_e &= \frac{\Psi_e}{\prod_{\alpha \ne 0} \psi(q^{-1}
      x^\alpha)^2}\notag\\
  &= (-1)^{\dim\fg} (1-q)^{2r}
  \frac{\cO_{\Slice_e q^{-1} \fg
    \big/\fc_{e,>0}}}{Z^1(\fg)}\label{eq:29} \,, 
\end{align}
where $\cO$ in the numerator of \eqref{eq:29}  denotes the character of functions, that is, the product
of $(1-w^{-1})^{-1}$ over all Chern roots $w$ of $\Slice_e q^{-1} \fg -
\fc_{e,>0}$. In deriving \eqref{eq:29} from \eqref{eq:24} we used
the fact that the adjoint representation $\fg$ is self-dual and has
trivial determinant.

\subsubsection{}
The density \eqref{eq:29} can be conveniently expressed in terms of
the multiplicity spaces $\fg_i^e$ in the decomposition
\begin{equation}
  \fg = \sum_{i\ge 0} V^i \boxtimes \fg^e_i \,,
\label{eq:21}
\end{equation}%
  of $\fg$ as a $\phi(SL(2)) \oplus \bC_\phi$-module. Here
  $\fg_i$ are the $\ad(h)$-eigenspaces and
  $V^i$ is the irreducible $\mathfrak{sl}_2$-module with highest weight
  $i$. Recall that the existence of an invariant bilinear form
  on $\fg$ implies that 
  $\fg_i^* \cong \fg_{-i}$ as $\bG^h$-modules.

  \begin{Lemma}\label{l_selfdual}
    The $\bC_\phi$-modules $\fg_i$ and $\fg_i^e$ are selfdual.
    The $\bC_\phi$-module $\fg_1\cong \fg_{-1}$ is symplectic.
    We have
    \[
      \det \bigoplus_{i \ge 0} \fg_i^e = 1 \,. 
    \]
\end{Lemma}

\begin{proof}
  From self-duality of $\fg$ and $V^i$ in \eqref{eq:21}, we conclude the
  self-duality of $\fg^e_i$. Since
  \begin{equation}
  \fg_i = \bigoplus_{j \in i +
    2\Z_{\ge_0}} \fg_j^e \,,
\label{eq:22}
\end{equation}
 it follows these modules are similarly self-dual.

  The module $\fg_{-1}$ has a skew-symmetric form
  \[
    (\xi,\xi') \mapsto ([\xi,\xi'],e) \,,
  \]
  which is nondegenerate because $\ad(e)$ has no kernel in
  $\fg_{-1}$.

  Finally, from \eqref{eq:22} we conclude
  \[
    \bigoplus_{i \ge 0} \fg_i^e = \fg_0\oplus \fg_1 \,.
  \]
  We have $\det \fg_1 = 1$ because $\fg_1$ is symplectic, and
  $\det \fg_0 = 1$ because $\fg_0$ is the Lie algebra of a reductive
  group $\bG^h$. 
\end{proof}

\subsubsection{}
Theorem \ref{t_spectral_Psi} can thus be restated as follows. 

\begin{Proposition}\label{p_spectral_Z} 
  We have
  \begin{equation}
    \label{eq:137bis}
    (f_1 \boxtimes f_2, \psi(\cT)) =
    \sum_{e} 
    \int_{q^{h/2} \bC_{\phi,\comp}} \left(\cP^Z_+ f_1 \right)
    \left(\cP_{-}^Z f_2 \right) Z_e \, 
      d_\textup{Haar}  \,, 
    \end{equation}
    where $\cP^Z_\pm$ is given by \eqref{eq:24} and
    \begin{equation}
      \label{eq:30}
      Z_e \big|_{q^{h/2} c} = q^{-\frac{\dim \fg + \dim \fg^e}{2}} \frac{(1-q)^{2r}}{Z^1(\fg)}
      \frac{\prod_{i>0} \det_{\fg_i^e} (1-q^{-i/2} c )\phantom{\,\,\,\,\,}}
      {\prod_{i\ge 0} \det_{\fg_i^e}(1-q^{-i/2-1} c)} \,, \quad c \in \bC_{\phi} \,. 
    \end{equation}
\end{Proposition}

\begin{proof}
  We compute
  \begin{align}
    (-1)^{\dim\fg} \tr_{\cO_{\Slice_e q^{-1} \fg
  / \fc_{e,>0}}} q^{h/2} c & =
                                  (-1)^{\dim\fg}
 \frac{\prod_{i>0} \det_{\fg_i^e} (1-q^{-i/2} c^{-1} )\phantom{\,\,\,\,\,}}
                                  {\prod_{i\ge 0}
                             \det_{\fg_i^e}(1-q^{+i/2+1} c^{-1})}
    \notag \\
    & =
        q^{-\frac{\dim \fg + \dim \fg^e}{2}}                           
 \frac{\prod_{i>0} \det_{\fg_i^e} (1-q^{-i/2} c )\phantom{\,\,\,\,\,}}
                                  {\prod_{i\ge 0}
      \det_{\fg_i^e}(1-q^{-i/2-1} c)} \,, 
        \label{eq:32}
  \end{align}
  where we have used the self-duality of $\fg_i^e$ in
  making the $c\mapsto c^{-1}$ change in the numerator, the
  triviality of the determinant of $c$ in $\bigoplus_{i\ge 0} \fg_i^e$,
  as we as the observations that
  \[
    \dim \fg - \dim \fg^e = \dim \textup{Orbit}(e) \in 2 \Z \,.
  \]
\end{proof}

\begin{Corollary}\label{cZZ} 
All integrals in \eqref{eq:137bis} depend on $\psi(x)$ only through
the combination
\[
  \frac{Z^1(x^{-1})}{Z^1(x)} = q x^2 \frac{Z^1(qx)}{Z^1(x)} \,.
\] 
\end{Corollary}
\begin{proof}
The terms in the product $Z^1(\fg)$ in the denominator of
\eqref{eq:30} are in bijection with the $Z^1(x^\alpha)$ factors in
\eqref{eq:24} up to the action of the Weyl group $W$. For any
$w,w'\in W$, the sum
$$
\sum_{w\cdot \alpha >0} x^{\alpha} + \sum_{w' \cdot \alpha >0  }
x^{\alpha} - \sum_{\alpha \ne 0} x^{\alpha}
$$
is a linear combination of terms of the form $x^{\beta} - x^{-\beta}$,
whence the conclusion. 
\end{proof}

\subsection{Hermitian form}

\subsubsection{}
 In equivariant K-theory there is a natural involution 
 \[
   \cV^* = \Hom_\C(\cV,\C)\,, 
 \]
 where $\C$ is the trivial bundle with the trivial action of all
 groups involved.
We extend it to $K_\bG(\bX) \otimes_\Z \C$ anti-linearly. 
Then, for instance, for
\[
  f \in K_\bG(\pt) \otimes \C \subset \C[\bG]
 \]
 we have
 \[
   f^*(g) = \bar{f}(g^{-1}) \,. 
 \]
 By definition, we set
 \begin{equation}
   \label{eq:20}
  \|f\|^2_\Psi = (f \boxtimes f^*, \psi(\cT)) \,. 
 \end{equation}

 \subsubsection{}

 In this subsection, we will assume that 
 \begin{align}
   q &>1 \,, \notag \\
   Z(x)&> 0\,, \quad x>q \,,  \label{eq:27} \\
   Z^1(x)&\ne 0\,,\quad  x \in \R \setminus \{1 < |x| < q\} \label{eq:27bis}
         \,. 
 \end{align}
 By the functional equation, this implies that $Z(x)>0$
 for $x \in (-1,1) \setminus \{0\}$. 
 Recall that, so far, we assumed $\psi(x)$ to be nonvanishing. However,
 starting in Section \ref{s_zeros}, we will relax this condition to
 having no zeros in the critical annulus/strip. We reflect this
 weakening of assumptions in \eqref{eq:27bis}. Note that by the functional
 equation, \eqref{eq:27} implies $Z(x)>0$ for $x\in (0,1)$. Also,
 \eqref{eq:27}
 is equivalent to 
 $Z^1(x)> 0$ for $x\in \R_{>0} \setminus (1,q)$.

 \subsubsection{}
 We have the following

 \begin{Proposition}\label{p_even} 
   With the assumptions \eqref{eq:27} we have
   \[
     \cP_{-}^Z f^* \Big|_{q^{h/2} \bC_{\phi,\comp}} =
   \overline{\cP_{+}^Z f} \Big|_{q^{h/2} \bC_{\phi,\comp}}
   \]
 and $Z_e \ge 0$ on $q^{h/2} \bC_{\phi,\comp}$. 
\end{Proposition}

\begin{proof}
Since the function $Z(x)$ is single-valued and real on a part of the
real axis, it follows that it is real, that is, $\overline{Z(x)} =
Z(\overline{x})$.

Consider an element of the form
\[
  x = q^{h/2} c \,, \quad c \in \bC_{\phi,\comp} \,, 
\]
and note that
\[
  \overline{x} =  q^{h/2} \overline{c} \sim q^{h/2} c^{-1} \sim 
q^{-h/2} c^{-1} = x^{-1} \,, 
\]
where similarity means being conjugate in $\bG$. Indeed,
$\overline{c}\sim c^{-1}$ inside a maximal compact subgroup of
$\bG^h$, while $h$ is conjugate to
$-h$ by the image of $\phi: SL(2)\to \bG^c$. 
Since $\cP_{+}^Z f$ is conjugation-invariant, we conclude 
\[
  \overline{\cP_{+}^Z f} \Big|_{x=q^{h/2}c }= \cP_{+}^Z \overline{f} \Big|_{x^{-1}} =
\cP_{-}^Z f^* \Big|_{x} \,.
\]

We now consider $Z_e$. Since eigenvalues of $c$ lie on the unit circle
in complex conjugate pairs, we observe that all terms in \eqref{eq:30}
are positive, with the possible exception of $Z^1(\fg)$.

The
representations $\fg_i$ being self-dual, its $c$-weights appear either
in complex conjugate pairs or are equal to $\pm 1$. We are thus left
with the contributions of $Z^1(\pm q^{i/2})$.

By our assumption, $Z^1(q^{i/2})>0$, except possibly $Z^1(q^{1/2})$.
However, the
representation $\fg_1$ being symplectic by Lemma \ref{l_selfdual}, the
$q^{1/2}$-weights in $\fg_1$ appears with even multiplicity, hence
its contribution is positive. 

Now consider the contribution of $Z^1(- q^{i/2})$. By \eqref{eq:27bis} all of these
numbers are of the same sign, except possibly $Z^1(- q^{1/2})$, which
appears with even multiplicity in the symplectic representation
$\fg_1$. The nontrivial weights in the adjoint
representation appear in pairs given by the opposite roots. In
particular, the total number of negative real weights in $\fg$ is even for
any element of $\bG$. This completes the proof. 
\end{proof}

\noindent
As a corollary of the proof, we note the following

\begin{Corollary}
In the statement of Proposition \ref{p_even}, it suffices to assume
that $Z(x)$ is real for $x>q$ and that $Z(x)^{\rk \bG} > 0$ for
$x>q$. 
\end{Corollary}
\begin{proof}
Indeed, the number of positive real eigenvalues in the adjoint
representation has the same parity as $\rk \bG$. 
\end{proof}

 \subsubsection{}

 \begin{Corollary}
   With the assumptions \eqref{eq:27} we have
   \begin{equation}
    \label{eq:137bisbis}
    \|f \|_\Psi^2 =
    \sum_{e} 
    \int_{q^{h/2} \bC_{\phi,\comp}} \left| \cP_+ f  \right|^2 
   \left| \Psi_e \right|  \, 
      d_\textup{Haar}  \,. 
    \end{equation}
  \end{Corollary}

  \begin{proof}
    We have
    \[
      Z_e = |Z_e| = \frac{|\Psi_e|}{\prod_{\alpha \ne 0} \psi(q^{-1}
      x^\alpha) \overline{\psi(q^{-1}
        x^\alpha)}} \,, 
    \]
    whence the conclusion. 
  \end{proof}






\subsection{Zeros is the critical annulus}\label{s_zeros}

\subsubsection{}

Our constructions so far assumed that the function $\psi(x)$ has no zeros
or poles in the whole complex plane, with the extension by analytic
continuations to functions holomorphic in a sufficiently large annulus
as in Corollary \ref{c_ac}.

Our next goal is to extend this result to functions having zeros in
the annulus $q^{-1} < |x| < 1$. For the function $Z(x)$ this
means a pair of zeros in the critical annulus.

If the function $\psi(x)$ has zeros then $\psi(\cT)$ may be written as
product of two distributions, equivalently, the Fourier transform of the
convolution of two distributions. We will check that this product
is well-defined and given by the same formulas as before if
all zeros of $\psi(x)$ lie in the annulus $q^{-1} < |x| < 1$.

\subsubsection{}

By construction, the tangent bundle $T\cT$ is the cohomology of the
complex
\begin{equation}
  \label{eq:179}
  T\cT = \Big(
\fb[1] \xrightarrow{\quad \textup{action} \quad}  TM \xrightarrow{\quad d\mu \quad}
q^{-1} \fb^* [-1] \Big)\,,
\end{equation}
where $M=T^*\cB$ and $\mu$ is the moment map.
Note that $\fb$ and $q^{-1} \fb^*$ are trivial bundles on $M$
with a nontrivial action of $q$ and $\bA$. In other words, they are
pulled back from $\pt/(\Ct_q \times \bA)$. 

Assuming $\psi(x)$ is nonvanishing outside of the critical strip, we
will first define $\psi(\fb+q^{-1} \fb^*)^{-1}$ as a distribution, 
then 
show that the product 
\begin{equation}
  \label{eq:180}
  \psi(T\cT) \eqdef \psi(\fb+q^{-1} \fb^*)^{-1}  \, \psi(TM) 
\end{equation}
of two distributions is well-defined, and finally show that is given by the same
formulas as before.

\subsubsection{}

Given $a\in \bA$, we define
\[
  \fb = \fb_{<} \oplus \fb_{\ge} 
\]
where $\fb_{<}$ is spanned by eigenvectors of $\ad(a)$ that are
less than $1$ in absolute value.  We begin with the following

\begin{Lemma}\label{lbm} Let $x$ be a point fixed by $a$ and $q$. If 
 $|a^\alpha| \ge  1$ for all roots in $\fb$ or if $x$ lies over a point $e \in \fn$
 and $a \in q^{h/2} \bC_{\phi,\comp}$, then $\fb_{<}$
  intersects trivially the stabilizer of $x$. 
\end{Lemma}

\begin{proof}
  If $|a^\alpha| \ge 1$ for all roots then $\fb_{<}=0$ and there is nothing to check. 
Since the stabilizer of $e$ in $\fg$ is contained in $\fp(h)$, we
conclude  that it intersects $\fb_{<}$ trivially. Therefore, the
stabilizer of $x$ intersects $\fb_{<}$ trivially and the result
follows. 
\end{proof}

\noindent 
Using Lemma
\ref{lbm}, we can write
\begin{equation}
  \label{eq:179-2}
  T\cT = \Big(
\fb[1]_{\ge}  \xrightarrow{\quad \textup{action} \quad}  T_{\ge} M \xrightarrow{\quad d\mu \quad}
q^{-1} (\fb_{\ge})^* [-1]\Big)\,,
\end{equation}
where
\[ 
T_{\ge} M = \Ker( T M \to 
q^{-1} (\fb_{<})^{*} ) / \fb_< \,,
\]
is a vector bundle on the locus where the conclusion of the Lemma
is satisfied. This means the factorization 
\begin{equation}
  \label{eq:179-3}
\psi(TM) = \psi(\fb_<+ q^{-1} (\fb_{<})^{*} ) \psi(T_{\ge} M) 
\end{equation}
where both factors are analytic K-theory classes. 

\subsubsection{}

\begin{Proposition}\label{p_convolve}
Suppose that $q^{-1} < |x_0| < 1$ for all $x_0$ such that
$\psi(x_0)=0$. Then the distribution $\psi(\fb+q^{-1} \fb^*)^{-1}$ and the product
\eqref{eq:180} are well-defined and given by the same formulas
as in the case when $\psi(x)$ is nowhere vanishing. 
\end{Proposition}

\noindent
In number field part of the proof, we will make the simplifying assumption that
$(\psi(s)\psi(-1-s))^{\pm 1}$ grows at most exponentially in the
imaginary direction. This assumption is easily seen to hold for all
number fields, see e.g.\ Chapter XIII.5 in \cite{Lang}. In fact, the
exponential terms in this asymptotics come from the contribution of
the Archimedean places. Faster growing
  functions may be handled by modifying the
  cutoff function $e^{\varepsilon s^2}$ in the proof below. 

\begin{proof}
We discuss the K-theory case, that is, the function field case
first.  Consider the action of $\bA$ on the vector
space $\fb$. This action is cohomologically proper, and the
Euler characteristic of any coherent sheaf on $\fb$ is a well-defined
Laurent series on $\bA$ which converges for all $a\in \bA$ such that
$|a^\alpha|>1$ for all positive roots $\alpha$. By construction, this
means that the distribution $\psi(TM)$ is given by a contour integral against
a convergent series in the same region. We define
the distribution
\[ 
\psi(\fb+q^{-1} \fb^*)^{-1} =  \prod_{\alpha >0} \psi(a^{\alpha})
\psi(q^{-1} a^{\alpha})
\]
using the Laurent
series expansion of the function $\psi(x)^{-1} \psi(q^{-1} x^{-1})^{-1}$ in the region
$|x| > 1$, where it is regular. Here we use our hypothesis about the
zeros of $\psi(x)$. Since the product of two convergent
series is well-defined, the distribution \eqref{eq:180} is
well-defined, and given by the same formula \eqref{eq:3}.

Now consider the pieces in the decomposition \eqref{eq:116_Thom}. On
each piece, using Lemma \ref{lbm}, we may replace the complex \eqref{eq:179} by a smaller
complex \eqref{eq:179-2}. By \eqref{eq:179-3}, this replacement means a repeated
application of the identity
\[
  \psi(x)^{-1} \psi(x) =1 \,.
\]
In the language of convolutions, this identity means
\begin{itemize}
\item the convolution with the coefficients of \emph{any} Laurent
series expansion of $\psi(x)^{-1}$, convergent or not,  is well-defined on the image of the
convolution with the Fourier transform of $\psi(x)$,
\item the composition of the two convolutions is the identity.
\end{itemize}
For the
smaller complex \eqref{eq:179-2}, the eigenvalues of $a$ in $\fb_{\ge 0}$ are all
$\ge 1$ in the absolute value. This makes the 
distribution $\psi(\fb_{\ge 0}+q^{-1} (\fb_{\ge 0})^*)^{-1}$ and its 
product with $\psi(T_\ge M)$ well-defined, as before.
This concludes the proof in the K-theory setting.

In the equivariant cohomology, that is, number field case, the
argument is the same with the following modification. We define the
distribution $\psi(s)^{-1} \psi(-1-s)^{-1}$ by its Fourier transform
on the line  $\Re s = s_0 > 0$. Since
we allow this function to grow exponentially in the imaginary
direction, this may not
be well-defined for arbitrary smooth test functions. 

To regularize,
we will multiply $\psi(s)^{-1} \psi(-1-s)^{-1}$  by $e^{\varepsilon
  s^2}$. After Fourier transform, this is equivalent to preconvolving
the test functions with an arbitrarily sharp Gaussian density. 
We can first prove the equality of convolution operators acting
on such restricted set of test functions and then let $\varepsilon
\downarrow 0$. Recall that by Corollary \ref{cZZ}, 
the functions $\psi(x)$ appear in the the final formulas  in
combinations $Z_\psi(s)/Z_\psi(s\pm 1)$ and that these grow
at most polynomially in the imaginary direction. Thus the
$\varepsilon \downarrow 0$  limit is well-defined for any
Paley-Winer test function, and this concludes the proof. 
\end{proof}

\begin{Corollary}\label{c_crit} 
The equality \eqref{eq:137bisbis} holds for functions $\psi$ analytic 
in the annulus $1/R_\bG < |x| < R_\bG$ and nonvanishing outside
of the critical strip $|q|^{-1} < | x |< 1$, where $R_\bG$ is a
sufficiently large number that depends on $\bG$. 
\end{Corollary}



\subsection{Proof of Theorems  \ref{t_L2} and \ref{t_spectral}}
\label{s_proof} 

\subsubsection{Proof of Theorem \ref{t_L2}}

We first consider the K-theory case. The difference between
the statement of Corollary
\ref{c_crit} and \eqref{eq:19} is the domain of the integration,
namely $\bbSt_e$ versus $\bbS_e$. We recall that $\cP_+ f$ is
the restriction to $\bbSt_e$ of a conjugation-invariant function
on $\bG$. Similarly, as observed in Section \ref{s_Psie},
the density $\Psi_e$ is a restriction of a conjugation-invariant
function. Thus the integrand in \eqref{eq:137bisbis} is
pulled back from $\bbS_e$, while the Haar measure on $\bbS_e$
is pushed forward from $\bbSt_e$. This proves the equality \eqref{eq:19}
in the K-theory case.

\subsubsection{}

The cohomology case follows from the K-theory by taking
a limit as explained in Section \ref{s_limitH}. While this concludes
the proof, it may be instructive to see explicit formulas below. 

\subsubsection{} \label{s_spell_out}

Since, in this paper, we focus on the more difficult K-theoretic
computations, it may be useful to spell out the formulas one obtains
in the specialization to cohomology. 

The combination of Theorems \ref{t_f1f2} and \ref{t_L2} may be
phrased as the following equality of integrals involving general
functions $\psi(s)$ and $f(s)$, analytic in a sufficiently wide
neighborhood of the imaginary axis. We define 
\[
  \Psi(s) = \frac{\psi(s)}{s} \,, \quad Z_\psi(s) = \Psi(-s) \Psi(s-1) 
\]
and denote
\[
  \cz_\psi = - \psi(0) \psi(-1) \,. 
\]
We require that: 
\begin{enumerate}
\item[(1)] $\psi(s)$ is
  nonvanishing outside the vertical strip $-1 < \Re s <0$\,,
\item[(2)] $Z_\psi(s)$ is real and positive for $s>1$\,,
  \item[(3)] The integrals on both sides of \eqref{eq:84}
    converge. For
    example, this can be guaranteed by $f$ being a Paley-Wiener
    function and $\psi(s\pm 1)/\psi(s)$ growing at most polynomially
    in the imaginary direction. See also Remark
    \ref{sec:remark_growth}. 
\end{enumerate}

In this setting, the combination of
Theorems \ref{t_f1f2} and \ref{t_L2} implies the following equality: 
\begin{multline}
  \label{eq:84}
\frac{1}{\cz_\psi^r} \int_{\Re x = x_0 > \rho^\vee} 
  \, 
   \sum_w f(x) \, \overline{f}(-w^{-1} x) \prod_{\substack{\alpha > 0 \\ w^{-1} \alpha <0 }}
   \frac{Z_\psi(\alpha(x))}{Z_\psi(1+\alpha(x))}  \, dx =\\
   \\ =
   \sum_{\phi} \int_{\frac12 \phi(h) + \Lie \bC_{\phi,\comp}} \left|\cP_+ f
   \right|^2 \, |\Psi_e| \,dx \,. 
\end{multline}
Here the summation is over the homomorphisms
\[
  \phi : \mathfrak{sl}_2 \to \fg \,,
\]
considered up to conjugation, and uses the image of the standard basis
\[
  h =  \begin{pmatrix}
    1 & 0\\
    0 & - 1
  \end{pmatrix} \,, \quad
  e = \begin{pmatrix}
    0 &1\\
    0 & 0
  \end{pmatrix} \quad
  e_- = \begin{pmatrix}
    0 &0\\
    1 & 0
  \end{pmatrix}\,, 
  \]
  in $\fg$. The the integration is over a
shift by
$\tfrac 12 \phi(h)$ 
of the Lie algebra of the maximal compact subgroup of the centralizer
$\bC_\phi$ of $\phi$ in $\bG$.

The spectral projectors and the spectral measure in \eqref{eq:84} are
defined by
\[
  \cP_+ f = \left( \sum_{w\in W} w \right) \left(\prod_{\alpha<0}
  \frac{\Psi(\alpha(x))}{\Psi(\alpha(x)-1)} \right) f
\]
and
\[
  \Psi_e = \psi(-1)^{-2r} \frac{\prod_{s\in \fg^{e_-}} \Psi (s-1)}
{\prod_{s\in \fg^{\phi}} \psi (s) \prod_{s\in \fg^e/\fg^{\phi}} \Psi (s)} \,,
\]
where the notation means that we take the product over the
weights $s \in \left(\Lie \bC_\phi\right)^*$ of the corresponding
representations. 


\subsubsection{Proof of Theorem \ref{t_spectral}}

Theorem \ref{t_L2} produces an injection from LHS of
\eqref{e_spectral} to the RHS. To prove the surjectivity, it is
enough to show the map has a dense image. Since it is a
map of $\sH$-modules it is enough to show it is nonzero on
each piece of the spectral decomposition. The latter claim
follows from the geometric definition of $\cP_+$ as integration
over the Spinger fiber and the classical fact that
$\left(\cB^{e,h}\right)^g$
is nonempty for any $g\in \bC_{\phi,\comp}$.

Indeed, any such $g$ is semisimple, as hence the Lie algebra
$\fg^g$, which contains $h$ and $e$, is reductive.
Borel subalgebras of $\fg^g$ parametrize
a component of  $\cB^g$. Since
$h$ and $e$ generate a solvable Lie algebra, they belong to
some Borel subalgebra of $\fg^g$, which is thus fixed by all three
elements.

\section{Example: Subregular in $G_2$}
\label{sectG2}

\subsection{Generalities on the dCLP stratification}

\subsubsection{}
The goal of this section is to present a discussion of a
particularly illustrative example --- that of a subregular nilpotent
in type $G_2$. From the contour deformation point of view, it
was dealt with in fine detail in \cite{L} and thus played a very
important role in the development of the subject.

The geometry of all subregular Springer fibers being well-understood,
see for instance \cite{Slodowy}, the only novelty in this section is in how this geometry
matches Langlands' computations. 
We begin with a few general reminders from \cite{dCLP} before
specializing the discussion to the specific example. For the
convenience of the reader, we remind the proofs. 

\subsubsection{}

Let $e$ and $h$ denote a fixed nilpotent element of $\fg$ and its
characteristic. Consider the grading of $\fg$ by the eigenvalues of
$\ad(h)$
\[
  \fg = \bigoplus \fg_{i} \supset \bigoplus_{i\ge 0} \fg_{i} = \fp\,,
\]
and the corresponding parabolic subalgebra $\fp$. We denote by
$\sP\subset \bG$ the parabolic subgroup with Lie algebra $\fp$.
For instance,
\[
  e = 0 \Leftrightarrow \fp = \fg \,.
\]
One notes that $\fp$ and $\sP$ 
are canonically associated to $e$. 

\subsubsection{} 
We have  $e\in \fg_2$ and $\ad(\sP) \, e \subset \fg_{\ge 2}$. 
The following proposition is classical.

\begin{Proposition}\label{p_open} 
We have the following diagram
\begin{equation}
  \label{eq:22_}
  \raisebox{30pt}{\xymatrix{
    \fg^\circ_{\ge 2} \overset{\textup{\tiny def}}= \ad(\sP) \, e \, \,   \ar@{^{(}->}[r] \ar[d] &
    \fg_{\ge 2} \ar[d]^{\sigma(0)} 
    \\
    \fg^\circ_{2} \overset{\textup{\tiny def}}= \ad(\bG^h) \, e  \, \,  \ar@{^{(}->}[r] & \fg_{
      2}}}\,, 
\end{equation}
in which the horizontal maps are open embeddings and
\begin{equation}
  \sigma(q) \cdot x =
  q^{-1} \ad(q^{h/2}) \, x \,.\label{eq:141_sigma}
\end{equation}
The fibers of $\sigma(0)$ are orbits of the unipotent radical
$\sP_\uni$  isomorphic to 
\begin{equation}
\fg_{\ge
  1}\big/\fc_{e,>0} \cong \fg_{>2} \,.\label{eq:27_}
\end{equation}
\end{Proposition}

\noindent Note that by construction 
\begin{equation}
  \label{eq:58_}
  \fg^\circ_{2}  \cong \bG^h / \bC_\phi \,. 
\end{equation}

\begin{proof}

  Decomposing $\fg$ with respect to the adjoint action of
  $\phi(\mathfrak{sl}_2)$, we see that the maps 
  \[
    \fg_{
    i} \xrightarrow{\quad \ad(e) \quad} \fg_{i+2} \,,
  \quad i  \ge 0 \,,
  \]
  are surjective. This shows that the horizontal maps in \eqref{eq:22_}
  are open embeddings and also proves the isomorphism \eqref{eq:27_}.
  Since the action of $\sigma(q)$ preserves the
  $\sP_\uni$-orbits and contracts $\fg_{\ge 2}$ to $\fg_2$ as
  $q\to 0$, the remaining claims follow. 



\end{proof}

\subsubsection{}

The orbits of $\sP$ on $\cB$ are given by the Bruhat decomposition
\begin{equation}
\cB  = \bigsqcup_{w \in W^h\backslash W} \sP\, w \, 
\fn \,,  \label{Bruhat}
\end{equation}
where it is convenient to take representatives of minimal length as in
Appendix \ref{A_Bh}, and where we assume that $h\in \Lie \bA$ as
in Appendix \ref{s_Lie_def}.
These Bruhat cells are the attracting manifolds for the $q^h$-action on
$\cB$. By definition, the dCLP strata
\begin{equation}
  \label{eq:25}
  \cB^e_w = \cB^e \cap \sP \, w \, \fn 
\end{equation}
are the intersections of the Spinger fibers 
with the $\sP$-orbits \eqref{Bruhat}.

\begin{Proposition}[\cite{dCLP}]
The strata \eqref{eq:25} are smooth. 
\end{Proposition}

\begin{proof}
  The Bruhat cells in \eqref{Bruhat} have the form $\sP/\sP_w$, where
  $\sP_w = \sP \cap w \bB w^{-1}$. The stabilizer $\sP_w$ acts in
  \[
    \fg_{\ge 2} \big/ \left(\fg_{\ge 2} \cap w \fn \right) 
  \]
  and this defines a homogeneous vector bundle over $\sP/\sP_w$ with a
  section
  \[
    \sP/\sP_w \owns g \mapsto \ad(g^{-1}) e \bmod
  \left( w \fn \right)_ {\ge 2} \,.
  \]
  This section is transverse because $\sP$ acts on $\fg_{\ge 2}$ with an open orbit and
  its zero locus is precisely the stratum \eqref{eq:25}. 
\end{proof}

\subsubsection{}

Recall from  Appendix \ref{A_Bh} that 
\begin{equation}
  \cB^h = \bigsqcup_{w\in W^h\backslash W}  \cB^h_w\,, \quad
  \cB^h_w = \bG^h \, w \, \fn  \cong \bG^h/\bB^h \label{eq:85_} \,, 
\end{equation}
and recall the subset $W(e)$ introduced in Section \ref{s_W(e)}.

\begin{Corollary}\label{c_dCLP} 
  We have: 
  \begin{enumerate}
  \item[(1)] the fixed loci $\cB^{h,e}=\bigsqcup \cB^{h,e}_w$ are smooth;
  \item[(2)] $ \cB^{e}_w \ne \varnothing \Leftrightarrow \cB^{h,e}_w \ne \varnothing \Leftrightarrow
    w \in W(e)$;
    \item[(3)] the manifold $\cB^{h,e}_w$ is cut out inside
      $\cB^{h}_w\cong \bG^h/\bB^h$ by a transverse section 
      \[
        \bG^h/\bB^h  \owns g \mapsto \ad(g^{-1}) e \bmod \left( w \fn \right)_2
  \]
  of a homogeneous vector bundle;
\item[(4)] if  $\cB^{h,e}_w\ne \varnothing$ then
  \begin{equation}
  \dim \cB^{h,e}_w = \dim \bG^h/\bB^h -
  \dim \fg_2 \big/\left( w \fn \right)_2\label{eq:35} \,, 
 \end{equation}
 and $\cB^{h,e}_w= \varnothing$ when this number is negative. 
  \end{enumerate}
\end{Corollary}


\subsubsection{}
The dimension formula \eqref{eq:35} also follows from the
identification
\begin{equation}
  \label{eq:31}
  \cB^{h,e}_w  / \bC_\phi = \left( w \fn \right)_2^\circ / \bB^h \,, 
\end{equation}
because
\[
  \dim \bC_\phi = \dim \bG^h - \dim \fg_2 \,.
\]
Similarly, we have
\begin{align}
  \dim \cB^{e}_w &= \dim \sP/\sP_w - \dim \fg_{\ge 2} \big/\left( w \fn
  \right)_{\ge 2} \notag \\ 
&= \dim \fg_{[0,1]}  - \dim \left( w \fb
  \right)_{[0,1]} \label{eq:36}
  \,, 
\end{align}
where the subscripts refer to the ranges of the grading.

\subsection{Nilpotent elements in type $G_2$}\label{s_g2nt}

\subsubsection{}

In the rest of this section, $\fg$ denotes the simple complex Lie algebra of type
$G_2$ and $\bG$ the corresponding simple complex Lie group (which is
both $1$-connected and adjoint). 
Figure \ref{f_g2} explains how we label the roots of $\alpha_i$ of
$\fg$ so that
\begin{align*}
  \{\alpha_1, \alpha_6\} & = \textup{simple positive roots} \,,\\
   \{\alpha_1, \alpha_3, \alpha_5 \} & = \textup{short positive roots}\,,
  \\
  \{\alpha_2, \alpha_4, \alpha_6 \} & = \textup{long positive roots}
                                      \,. 
\end{align*}
We identify $\mathfrak{h}^\vee \owns \alpha_i$ with its dual so that
$\|\alpha_i\|^2 \in \{2,6\}$. Then
\[
  \{\alpha_3,\alpha_4^*\}\,, \quad \alpha_4^*=\tfrac13
\alpha_4\,, 
\]
are the fundamental coroots.
\begin{figure}[!h]
  \centering
 \includegraphics[scale=0.42]{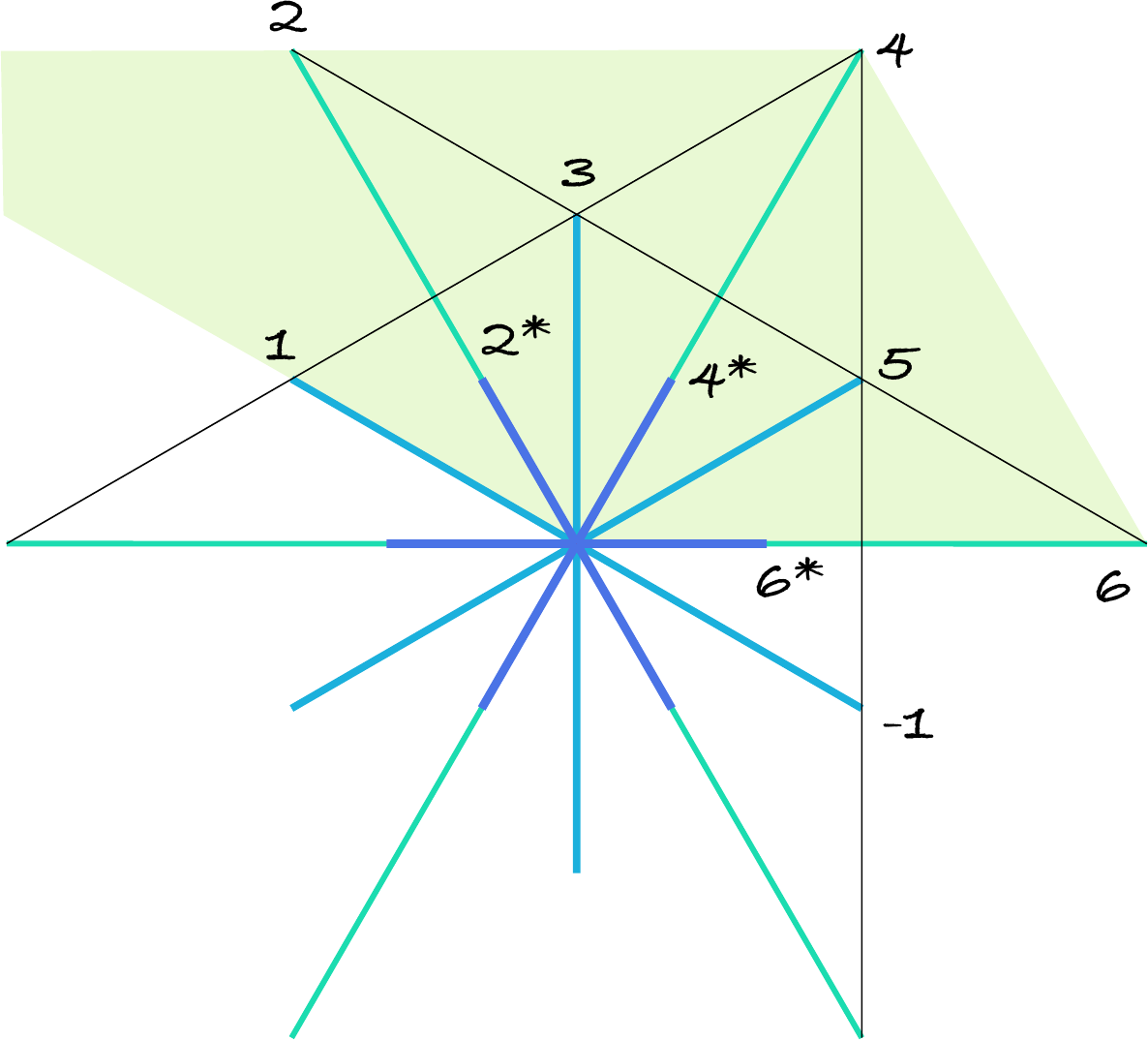}
  \caption{The positive roots $\alpha_i$ of $\mathfrak{g}$ are labelled by
     $\{1,\dots,6\}$. We normalize them so that $\|\alpha_i\|^2 \in
     \{2,6\}$. With this normalization, $3$ and $4^*$ are the
     fundamental coroots. The positive cone is shaded. The black lines
   is where the characteristics $\{2\alpha_2^*, 2\alpha_4^*,
   2\alpha_6^*\}$ take value 2.}
  \label{f_g2}
  \end{figure}

\subsubsection{}

Recall that the nilpotent elements in $\fg$ are classified as
follows, see \cite{Coll}. First, there are those which are not
distinguished, that is, belong to a proper Levi subalgebra. This means
they are either 0 or conjugate to a root vector. The corresponding
characteristic $h$ is thus either 0 or an element in the dual roots
system.

In addition the the principal nilpotent, the Lie algebra $\fg$ has
another distinguished class with the characteristic conjugate to
\begin{equation}
  \label{h_dist}
  h= 2 \alpha_4^* = \tfrac{2}{3} \alpha_4 \,. 
  \end{equation}
  This element in plotted in red in Figure \ref{f_example_g2_I}.
  For the rest of this section, we fix this $h$. It gives 
  \[
    \fg_2 = \fn_2 = \textup{span}(e_2,e_3,e_5,e_6)\,.
  \]
These are the roots that line on one of the black lines in Figure
\ref{f_g2}.

\subsubsection{}

We have
\[
  \bG^h \cong GL(2)
\]
where the $SL(2)$ subgroup corresponds to the root $\alpha_1$ and the
center corresponds to the coroot $h/2$.
%
%
As a $\bG^h$-module, 
\[
  \fg_2  \cong {\det}^{-1} \otimes S^3 \C^2 \,, 
  \]
  where $\C^2$ is the defining representation of $GL(2)$. The
  weights of this module can be seen in Figure \ref{f_example_g2_I}. 

  Since the center of
  $GL(2)$ acts with weight $1$ in this representation, 
  the action of $\bG^h$ on $ \fg_2$ reduces
  to action of $PSL(2)$ on
  \[
    \bP\left(  S^3 \C^2 \right) = S^3 \bP^1 \,,
  \]
  that is, the action on unordered triples of points on $\bP^1$. This
  action has an open orbit corresponding to triples of disjoint
  points, with stabilizer isomorphic to $S(3)$.
  
  The preimage of this open orbit in $ \fg_2$ is $\fg_2^\circ$.
  It corresponds to cubic polynomials in $S^3 \C^2$ with distinct
  roots. For $e\in \fg_2^\circ$, the stabilizer is $S(3)$ and
  the inclusion of the stabilizer
  \begin{equation}
 S(3) = \bC_\phi  \xrightarrow{\quad \iota \quad} \bG^h  =
 GL(2) \label{eq:39}
 \end{equation}
 is the irreducible 2-dimensional representation of $S(3)$.

\subsubsection{}

Since $\fg_2 = \fn_2$, we have $\fg_2/\fn_2 =0$ and hence
\[
  \cB^{e,h}_1 = \bG^h / \bB^h \cong \bP^1 \,.
\]
For other elements in $W$, we have 
\begin{alignat*}{1}
  \left(s_6 \fn\right)_2 &= \textup{span}(e_2,e_3,e_5)\,, 
  \\
  \left(s_2 \fn\right)_2 &= \textup{span}(e_5,e_6)\,,
\end{alignat*}
et cetera. In general, the dimension of $(w \fn)_2$ for various $w$
correspond to the number of positive roots $\alpha$ on lines $\alpha(w^{-1} h)
= 2$. Three of these lines are plotted in Figure \ref{f_g2}, they 
contain 4, 3, and 2 roots, respectively. 
Thus counting dimensions as in Corollary \ref{c_dCLP}, we see that
\[
  \cB^{e,h}_w = \varnothing \,, \quad w \notin \{1,s_6\} \,, 
\]
while 
\[
  \cB^{e,h}_{s_6} = \textup{3 points} \,. 
\]
Indeed, 
\begin{align*}
  \bP( (s_6 \fn)_2^\circ )&\cong \textup{distinct triples in $S^3 \bP^1$
    containing $\infty \in \bP^1$} \\
   & \cong S^2 \A^1 \,, 
\end{align*}
and thus
\[
  \cB^{e,h}_{s_6} / S(3)  \cong (s_6 \fn)_2^\circ \big/ \bB^h  \cong
\pt / S(2)
\,, 
\]
where we used the identification \eqref{eq:31}.

\subsubsection{}

Computing the dimensions, we see that $\cB^e_1 = \bP^1$ while
\[
  \cB^e_{s_6} = \textup{3 copies of $\A^1$} \,.
\]
Altogether, the Spinger fiber $\cB^e$ is the configuration of
four $\bP^1$'s, in agreement with the general
theory \cite{Slodowy}.  See 
Figure \ref{fig_Be}, where the
arrows indicate the directions of
attraction of $q^h$ as $q\to 0$.
\begin{figure}[!h]
  \centering
 \includegraphics[scale=0.5]{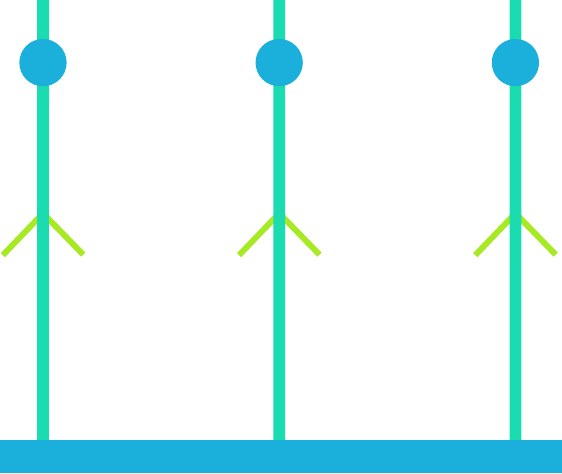}
  \caption{The subregular Springer fiber of type $G_2$ is the
    configuration of $\bP^1$'s whose dual graph is the Dynkin diagram of
    type $D_4$. The $h$-fixed locus, plotted in thick blue,
    is the union of $\cB_e^1\cong \bP^1$ and one point in each of
    the other components. These 3 components
 are permuted by $\bC_\phi\cong S(3)$.}
\label{fig_Be}
\end{figure}

\subsection{Spectral projectors and spectral measure}

\subsubsection{}

Our next goal is to make explicit the spectral projectors and spectral
measure for $e$. We begin with choosing the coordinates $(x_1,x_2)$ on $\bA$ so that
\begin{equation}
\bB^h = \left\{
  \begin{pmatrix}
    x_1 & * \\
    0 & x_2 
  \end{pmatrix} \right\} \subset GL(2) = \bG^h \,.
\label{eq:73}
\end{equation}
The adjoint weights of $\bA$ are displayed in Figure \ref{f_example_g2_I}. 
The Weyl group $W^h\cong S(2)$ acts permuting $x_1$ and $x_2$.

\begin{figure}[!h]
  \centering
  \def\svgwidth{7.7cm}\input{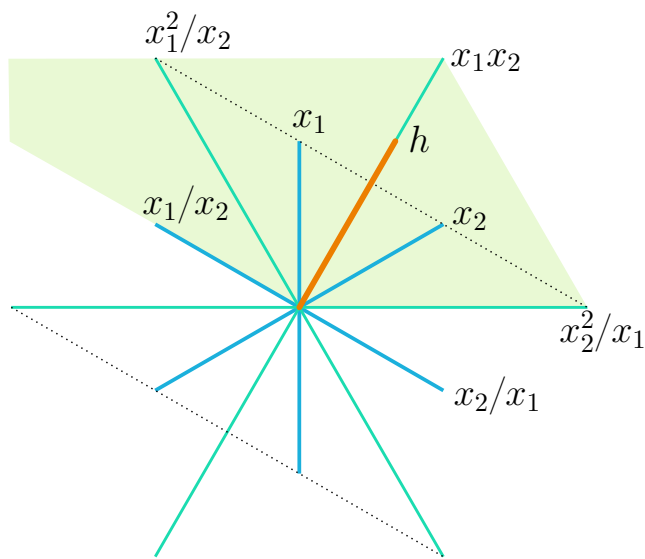}
 \caption{The characteristic $h$, the adjoint weights of $\bA$, and
   the positive cone. The dotted lines are where $h$ takes values $\pm
 2$.}
\label{f_example_g2_I}
\end{figure}

\subsubsection{}

The map \eqref{eq:39} satisfies 
 \begin{equation}
 \iota((12)) \sim \begin{pmatrix}
    1 \\
    & -1 
  \end{pmatrix}\,, \quad \iota((123)) \sim \begin{pmatrix}
    \nu \\
    & \nu^2 
  \end{pmatrix} \,, \quad \nu=e^{2\pi i/3} \,.
\label{eq:77}
\end{equation}
Comparing the characters we conclude that the decomposition
\eqref{eq:21} takes the form
\begin{equation}
  \label{eq:40}
  \fg = V^2 \boxtimes (1 + \C^2) + V^4 \boxtimes \sgn
\end{equation}
where $1$ and $\sgn$ are the two characters of $S(3)$, while
$\C^2$ is the irreducible 2-dimensional representation. 
In particular,
\begin{align}
  \Slice_e q^{-1} \fg &= q^{-2} (1 + \C^2) + q^{-3} \sgn\,,  \notag \\
  \fc_e &= q (1 + \C^2) + q^2 \sgn \,, 
  \label{eq:58}
\end{align}
as $\bCh_e$-modules. We conclude
\begin{equation}
  \label{eq:59}
  \Psi_e \Big|_{q^{h/2} S(3)} = \Psi\Big((q^{-2}-q) (\C^2 +1) + (q^{-3}-q^2) \sgn -4
  q^{-1}\Big)\,. 
\end{equation}

\subsubsection{}
The spectral projectors $\cP_\pm$ are given by integration along the
Spinger fibers as in 
\eqref{e_cPbB}. They take a function $f(x_1,x_2)$ on 
the torus $\bA$, viewed as K-theory class on $\cB_e$
to a conjugation-invariant
function on $S(3)$. The difference between $\cP_\pm$ is in the
identification of $K_\bG(\cB)$ with functions on the torus $\cA$.
Since opposite Borel subgroups are used for this identification, we
have
\[
  \cP_+ = \cP_- \circ w_0 \,, \quad w_0 (x_1,x_2) = (x_1^{-1},x_2^{-1}) \,.
\]
Here $w_0$ is the longest element of the Weyl group.

By stratification and equivariant localization, the integral over
$\cB^e$ is reduced to integrals over $\cB^{h,e}_w$ for $w=1,s_6$,
giving the decomposition
\[
  \cP_+ = \cP_{+,1} + \cP_{+,s_6} \,.
\]
We deal with them separately, starting with $w=1$.

\subsubsection{}

Restricted to $\cB^h_1$, the
virtual tangent bundle to $\cB_e$ is induced from the $\bB^h$-module
\begin{align}
(1-q^{-1}) \fn^\vee &= (1-q^{-1}) (\fg_{-4}+\fg_{-2}) + (1-q^{-1})
                      x_2/x_1 \label{virTnchar}\\
                    &=\underbrace{(1-q^{-1}) (\fg_{-4}+\fg_{-2}) + q^{-1} (2-\fg_0)
                      }_ {\textup{trivial bundle $V_0$}}+
                      \underbrace{x_2/x_1 + q^{-1} x_1/x_2}_{T^*
    \bP^1} \,.  \notag
\end{align}
We observe that last two terms give the cotangent bundle to
$\bB^{h,e}_1 \cong \bP^1$, while the remaining terms, being
$\bG^h$-modules, induce a trivial virtual bundle $V_0$ over $\bP^1$ with
the corresponding action. Computing their $S(3)$-characters, we
conclude
\[
  V_0 = -q^{-3} \sgn -q^{-2} (1+\C^2) + 2 q^{-1} \,. 
\]
Therefore 
\[
  \cP_{+,1} \, f  \Big|_{q^{h/2} S(3)}
= \Psi(V_0)  \, \DD  f  \Big|_{q^{h/2} S(3)}
\]
where
\begin{align}
  \DD  f(x_1,x_2) &= \chi( f \otimes \psi(T^*\bP^1)) \notag \\
  &= (1 + (12)) Z(x_1/x_2) f(x_1,x_2) \,.  \label{eq:64}
\end{align}
The result in \eqref{eq:64} is interpreted
geometrically as an element of $K_{GL(2)}(\pt)$, which we subsequently
restrict to $q^{h/2} S(3)$.

Collecting all terms above, we obtain 
\begin{equation}
  \Psi_e \left(\cP_{+,1} f_1\right) \left( \cP_{-,1} f_2 \right) \Big|_{q^{h/2} S(3)}=
\left.  \frac{\DD f_1(x_1,x_2) \, \DD  f_2(x_1^{-1},x_2^{-1})}
{Z(q^2(1+\C^2)+q^3\sgn)} \right|_{q^{h/2} S(3)} \,. 
  \label{eq:61}
\end{equation}

\subsubsection{}

Note that because of the pole of $Z(x_1/x_2)$ at $x_1=x_2$,
the value of \eqref{eq:64} at $q^{h/2}$, that is, at
$(x_1,x_2)=(q,q)$, involves the derivatives of $f$ at this point,
as well as the subleading term in the expansion of $Z$ at the pole.

Concretely, if
\[
  Z(x) = \frac{z_{-1}}{1-x^{-1}} + z_0 + \dots
\]
then
\[
  \DD f(x_1,x_2) \Big|_{x_1=x_2} = (z_{-1} + 2 z_0) f(x_1,x_1) +
z_{-1} \left.\left(x_1 \frac{\partial}{\partial x_1} - x_2 \frac{\partial}{\partial
      x_2} \right) f \right|_{x_1=x_2} \,.
\]

\subsubsection{}

We now consider $\cB^{h,e}_{s_6}$, which consists of $3$ points
permuted by $S(3)$. Since this is a permutation representation, the
character of $(123)$ vanishes, the character of $(12)$ 
equals the contribution of unique point fixed by this transposition,
while the character of identity is 3 times the contribution of a
single point. 

As a concrete nilpotent element $e$ we may take
\[
  e = e_2 + e_5 \,,
\]
where the labeling of the roots is an in Figure \ref{f_g2}, the
normalization of the root vectors being immaterial. This
belongs to $\fg_2^\circ$, is fixed by
\begin{equation}
  \bA^e=  \left\{
  \begin{pmatrix}
    \varepsilon & 0 \\
    0 & 1
  \end{pmatrix} \right\} \,, \quad \varepsilon = \pm 1\,, 
\label{eq:73_}
\end{equation}
and is contained in $s_6 \fn$, see Figure \ref{f_example_g2_II}. 

\begin{figure}[!h]
  \centering
  \def\svgwidth{7.7cm}\input{g2_example2.pdf_tex}
  \caption{Same as in Figure \ref{f_example_g2_I} but now $s_6 \fn
    \in \cB^e$ is
    shaded.}
\label{f_example_g2_II}
\end{figure}

The reflection $s_6$ acts by
\[
  s_6 \cdot (x_1,x_2) = (x_1,x_1/x_2) 
\]
and we have to apply $s_6$ to both the K-theory class $f$ and to 
the character of the virtual normal bundle \eqref{virTnchar} to compute
the contribution of the point $s_6 \fn \in \cB^{e,h}_{s_6}$.
We obtain
\begin{align}
(1-q^{-1}) T \cB \Big|_{s_6 \fn} &= 2 (q^{-1} - q^{-2}) + (q-q^{-3}) \varepsilon
                                   \,. \notag \\ 
 & = V_0 \Big|_{S(2)} + \varepsilon(q+q^{-2}) \,. \label{eq:72}
\end{align}
Therefore
\begin{alignat}{2}
\cP_{+,s_6} &f \Big|_{q^{h/2}} && = 3 \Psi(V_0) Z(q^2)  f(q,1) \,, \label{eq:65_1}
  \\
\cP_{+,s_6} &f \Big|_{q^{h/2}(12)} && =  \Psi(V_0) Z(-q^2) f(-q,-1) \,, \\
\cP_{+,s_6} &f \Big|_{q^{h/2}(123)} && = 0 \,. \label{eq:65}
\end{alignat}

\subsubsection{}
We now have all the data needed to express the contribution of the
subregular nilpotent $e$ to the spectrum. The contributions consist
of 3 isolated points corresponding to the conjugacy classes of
$S(3)$.

The identity $1\in S(3)$ corresponds to the point
\[
  W q^{h/2} = W(q,q) \in \bA/W
= \spec \sH
\]
and contributes
\begin{equation}
\| \delta_{W(q,q)} f \|_\Psi^2 = \frac{1}{6} \frac{1}{Z(q^2)^3 Z(q^3)}
\left| \DD f \Big|_{(q,q)} + 3 Z(q^2) f(q,1) \right|^2
\,, \label{eq:66}
\end{equation}
to the spectrum, where $\delta_{W(q,q)}$ is the spectral projector. 
The prefactor $\frac{1}{6}$ comes from the Haar measure on $S(3)$ and 
the difference $3 Z(q^2) f(q,1)$ between \eqref{eq:66} and
\eqref{eq:61} is the contribution of \eqref{eq:65_1} \,.

The transposition $(12) \in S(3)$ corresponds to the point $(-q,q) \in
\spec \sH$ and it contributes
\begin{equation}
\| \delta_{W(-q,q)} f \|_\Psi^2 = \frac{1}{2} \frac{
\left| Z(-1)(f(-q,q)+f(q,-q)) + Z(-q^2) f(-q,-1) \right|^2}
{Z(q^2)^2 Z(-q^2) Z(-q^3)} \,.
\label{eq:67_}
\end{equation}
Finally, the cycle $(123)$ corresponds to the point $(q \nu, q \nu^2)$
where $\nu=e^{2\pi i/3}$, and it contributes
\begin{equation}
\| \delta_{W(q\nu,q\nu^2)} f \|_\Psi^2 = \frac{1}{3} \frac{
\left| Z(\nu^{-1}) f(q\nu,q\nu^2)+Z(\nu) f(q\nu^2,q\nu) \right|^2}
{Z(q^2) Z(q^2 \nu) Z(q^2 \nu^2) Z(q^3)} \,. 
 \label{eq:68_}
\end{equation}

\subsubsection{}

In the number field situation, that is, in cohomology, only the point
\eqref{eq:66} survives and becomes the point 
\[
  \tfrac12 h = \alpha^*_4=(1,1) \in \Lie \bA/ W\,.
\]
It is instructive to compare the results
with Landlands' computation.

The additive analog of the operator \eqref{eq:64} is the
operator
\begin{equation}
  \label{eq:80}
  \DDa f = (1+(12)) \, \xi(x_1- x_2) f(x_1,x_2) \,. 
\end{equation}
where
\[
  \xi(x) = - \frac{1}{x} + a + \dots \,,
\]
in the notations introduces at the bottom of page 186 in \cite{L}. This
means that
\[
  \DDa f \Big|_{x_1=x_2} = - (\partial_1 - \partial_2) f
    (x_1,x_1) + 2 a f(x_1,x_1) \,. 
    \]
    Therefore
\begin{equation}
\| \delta_{W(1,1)} f \|_\Psi^2 = \frac{1}{6} \frac{
\left| (\partial_1 - \partial_2) f(1,1) - 2a f(1,1)- 3 \xi(2) f(1,0)
\right|^2}
{\xi(2)^3 \xi(3)}
\,, \label{eq:69_}
\end{equation}
where we have changed the sign under the absolute value sign
to ease the comparison with the matrix in Figure \ref{f_matrix}. 

\begin{figure}[!h]
 \centering
  \includegraphics[scale=0.38]{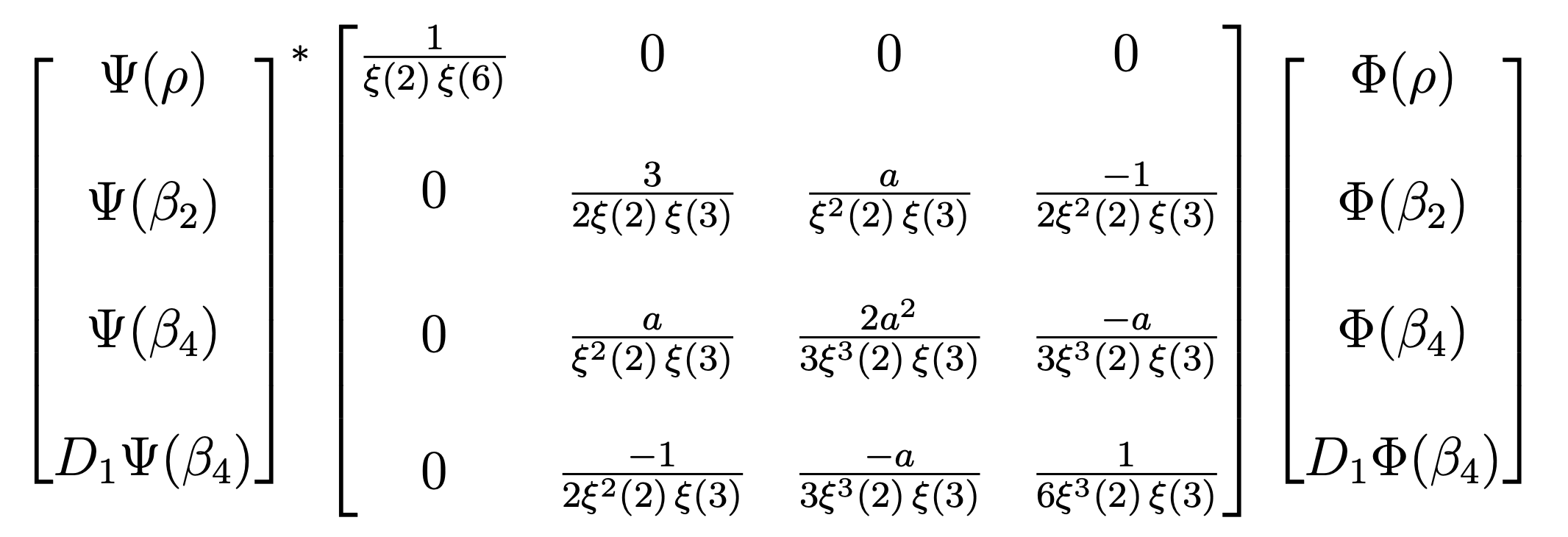}
  \caption{The spectral projector for the discrete spectrum of
    $G_2$ reproduced from page 195 of \cite{L}. The
    point $\rho$, which corresponds to $\rho^\vee$ in our notations,
    is the contribution of the regular nilpotent considered in Section
\ref{s_ex_reg}.}
\label{f_matrix}
\end{figure}

To match the coordinates with Langland's coordinates, we note 
that Langlands' roots $\beta_i$ are our dual roots $\alpha_i^*$,
see the illustration on on page
181 of \cite{L}. In particular, 
\[
  \beta_4 = (1,1) \,, \quad \beta_2 = s_6 \beta_4 = (1,0) \,,
\]
and 
\[
  \beta_1 = 2 \beta_2 - \beta_4 = (1,-1) \,. 
\]
Since the derivative $D_1$ used by Langlands is the derivative in the
direction of $\beta_1$, we conclude
\[
  D_1 =  \partial_1 - \partial_2 \,,
\]
completing the match.

\appendix 

\section{Appendix}

\subsection{Group cohomology}\label{a_coh}

\subsubsection{}

The cohomology of a complex algebraic group $\bG$ with
coefficients in a representation $V$ is defined by
\[
  H^i(\bG,V) = \Ext^i_\bG (\C, V) \,,
\]
where $\C$ is trivial representation and Ext's are computed in the
category of algebraic $\bG$-modules. These are the derived functors of $\bG$-invariants
\[
  H^0(\bG, V) = V^\bG \,.
\]
They may be computed using the usual
cochain complex, see e.g.\ \cite{FF}, in which the terms are given by
morphisms $\bG^{\times n} \to V$ of algebraic varieties.

\subsubsection{}
We define
\begin{equation}
\chi(\bG,V) = \sum_i (-1)^i H^i(\bG,V)  \,.\label{eq:6}
\end{equation}
This is linear in the $K_\bG(\pt)$-class of $V$ by the long exact
sequence of cohomology.

\subsubsection{}
For a group extension
\[
  1 \to \bG_1 \to \bG \to \bG_2 \to 1 \,, 
\]
we have 
\begin{equation}
  \label{eq:7}
  \chi(\bG,V) = \chi(\bG_2, \chi(\bG_1,V) ) 
\end{equation}
by the spectral sequence for derived functors of 
\[
  V^\bG = \left( V^{\bG_1} \right)^{\bG_2} \,.
\]
Formula \eqref{eq:7} reduces the computation of $\chi(\bG,V)$ to that for
the unipotent radical and the reductive quotient in \eqref{Most}.

\subsubsection{}
We have (essentially by definition) 
\[
  \textup{$\bG$ is reductive} \, \Leftrightarrow H^{>0}(\bG,
\, \cdot \,) = 0 \,, 
\]
and if $\bU$ is a unipotent group then
\begin{equation}
\Hd(\bU,V) = \Hd(\fu, V)\,, \quad \fu= \Lie \bU\,, \label{eq:8}
\end{equation}
as a consequence of the Hochschild-Mostow spectral sequence, see
\cite{FF}, Section 1.\S 3.

\subsubsection{}
We recall that Lie algebra cohomology in \eqref{eq:8} is succinctly
computed by the Chevalley-Eilenberg complex
\begin{equation}
\Hd(\fu, V)  = \Hd\left(V \to V \otimes \fu^\vee \to V \otimes
  \Lambda^2 \fu^\vee \to \dots \right) \,.\label{ChE}
\end{equation}

\subsection{Cohomologically proper actions}\label{a_proper}

\subsubsection{}
Consider an action-preserving map
\begin{equation}
  (\bG_1, \bX_1)  \xrightarrow{\quad (f_\bG,f_\bX) \quad} 
    (\bG_2, \bX_2) \label{fXY}
\end{equation}
with $f_\bG$ surjective. It is cohomologically proper if 
$
\Rd(f_\bG,f_\bX)_* \cF
$
is a bounded complex of coherent sheaves for any $\cF\in
\Coh_{\bG_1}(\bX_1)$.
Here
\[
  R^0(f_\bG,f_\bX)_* \cF  =\left(f_{\bX,*}\cF\right)^{\ker f_\bG} \,. 
\]
Obviously, cohomological proper maps are closed under composition and
a proper map is cohomologically proper.

\subsubsection{} 
By definition, an action of an algebraic group $\bG$ on a scheme $\bX$
is cohomologically proper if the canonical map 
\[
  (\bG, \bX) \to (\{1\},\pt) 
\]
is cohomologically proper. In this case, we also say that the
quotient stack $\bX/\bG$ is cohomologically proper.

\subsubsection{}

A 
nontrivial example of a cohomologically proper map is given by the
following

\begin{Proposition}\label{p_proper}
  If the map $f_\bX$ in \eqref{fXY} is affine, $\ker f_\bG$ is reductive, and
  \[
    (f_{\bX,*} \cO_{\bX_1})^{\ker f_\bG}= \cO_{\bX_2} 
  \]
  then \eqref{fXY} is cohomologically proper
\end{Proposition}

\begin{proof}
It suffices to prove that for a finitely generated algebra $\cA$ with
an action of a reductive group $\bG$
\[
  \dim \cA^{\bG} < \infty
\Leftrightarrow \forall V\,, \dim (\cA \otimes V) ^{\bG} < \infty
\]
where $V$ ranges over finite-dimensional $\bG$-modules.

Highest weights vectors of irreducible $\bG$-submodules in $\cA$
span the graded algebra
\begin{equation}
  \label{grA}
  \cA^\bU = \bigoplus_{\lambda \in \,  \textup{dominant weights}}
  \cA^\bU_\lambda \,,
\end{equation}
where $\bU \subset \bG$ is a
maximal unipotent subgroup. This algebra is finitely generated by
Theorem 3.13 in \cite{VinPop}. Therefore
\[
  \dim \cA^{\bU}_0 < \infty
\Leftrightarrow \forall \lambda \,, \dim \cA_\lambda^{\bU} < \infty\,,
\]
as was to be shown. 
\end{proof}

\subsubsection{}

\begin{Proposition}
  The conclusion of Proposition \ref{p_proper} also holds
  for $\ker f_\bG=\Ga$. 
\end{Proposition}

\begin{proof}
 Suffices to prove that 
 \[
   \dim H^1(\Ga,\cA) \le \dim H^0(\Ga,\cA) \,,
  \]
  for any algebraic module $\cA$, that is, any module of the form
  \[
    \cA = \bigcup \cA_n\,, \quad \dim \cA_n < \infty \,. 
  \]
  Here $\cA_n$ is a finite-dimensional vector space with a nilpotent
  operator, for which $H^0$ and $H^1$ are the kernel and the cokernel,
respectively. Therefore, 
\[
  \dim H^1(\Ga,\cA)  \le \lim \dim H^1(\Ga,\cA_n) =
  \lim \dim H^0(\Ga,\cA_n)  = \dim H^0(\Ga,\cA)  \,.
  \]
\end{proof}

\begin{Remark}
For a general unipotent group, the conclusion of
Proposition \ref{p_proper} fails, which has to do with the fact that
a quotient of an affine variety by a unipotent group is, in general,
only quasiaffine. As a simplest example, one can take
the maps
\begin{equation}
  \label{eq:94}
  \xymatrix{
    U \backslash SL(2,\C) /U \ar[rr] \ar[dr] && U \backslash (\C^2 \setminus
    \{0\}) \ar[dl]\\
    & \C }\,, 
\end{equation}
where $U$ is a maximal unipotent subgroup. 
\end{Remark}

\subsection{Flag varieties and Springer fibers}\label{a_Lie} 

\subsubsection{}\label{s_Lie_def}

Let $\bG$ be a complex semisimple Lie group with $\Lie(\bG)=\fg$. Choose a
Borel subgroup $\bB
\subset \bG$. We have 
\[
  \bB = \bA \ltimes \bU
\]
where $\bU$ is the unipotent radical and $\bA$ a maximal torus. We denote
\[
  \fb = \Lie (\bB)\,, \quad \fn = \Lie \bU \,.
\]
We denote by $r$ the rank of $\fg$\,. 

\subsubsection{} 

The flag variety
\[
  \cB \overset{\textup{\tiny def}}= \bG/\bB  = \{ \fb' \subset \fg \} = \{ \fn' \subset \fg \}
\]
may be seen as parametrizing Borel subalgebras $\fb' \subset
\fg$, or the corresponding nilpotent subalgebras $\fn' = [ \fb',
\fb']$, 
with the adjoint action of $\bG$. Further, 
\begin{equation}
T^* \cB = \{(e,\fn'), e \in \fn'\} \,. \label{TsB}
\end{equation}
Here $e\in \fg$ is a nilpotent element. The map
\begin{equation}
\mu: T^*\cB \owns (e,\fn') \mapsto e \in \fg\label{eq:144}
\end{equation}
becomes the moment map upon the identification
$\fg \cong \fg^*$ given by an invariant bilinear form. We
will always treat this map $\Ct_q$-equivariantly,
where $\Ct_q$ acts by $e \mapsto q^{-1} e$. 

\subsubsection{}\label{A_Bh}

For any semisimple $h\in \fg$ we have 
\begin{equation}
\cB^h = \bigsqcup_{w\in W^h\backslash W} \bG^h \, w \, \fn
\,, \label{eq:85}
\end{equation}
where $w$ ranges over coset representatives of minimal length, that
is, such that
\[
  w^{-1} \bB^h w \subset \bB \,.
\]
Each component in \eqref{eq:85} parametrizes $h$-stable maximal
nilpotent subalgebras $\fn'$ with the same dimensions of the graded
components as $w \cdot \fn$.

\subsubsection{}\label{s_App_Spring}

If $e\in \fg$ is a nilpotent element then
\begin{equation}
\cB^e = \mu^{-1}(e)\label{Spri}
\end{equation}
is the Springer fiber of $e$. Being the zero locus of a vector field,
it has a natural derived scheme structure with the
virtual tangent bundle $(1-q^{-1}) T\cB$. 

\subsection{Normalization of measures}\label{s_pT3}

\subsubsection{}

While the spectral decomposition is obviously insensitive to the
precise normalization of the measures, it is worth mentioning
briefly the different choices that contribute to the normalization in 
Langlands' formula \eqref{eq:99}.

\subsubsection{}

The Haar measure on $\bG(\A)$ is the pushforward
of the measure
\begin{equation}
d_{\textup{Haar}}(k) \times |a|^{-2\rho} d_{\textup{Haar}}(a)
\times d_{\textup{Haar}}(u)\label{eq:164}
\end{equation}
under the Iwasawa factorization map 
\[
  \bK \times \bA \times \bU  \to \bG \,. 
\]
There is a canonical way to normalize
the first and the last factor in \eqref{eq:164} by 
\begin{equation}
\vol(\bK) = \vol(\bU(\A)/\bU(\F)) = 1 \,. \label{eq:165}
\end{equation}
Instead of the Iwasawa factorization, one can fix a measure
on the Borel subgroup $\bB = \bA \bU$ and normalize the volume
of the flag manifold 
$\vol(\bG(\A)/\bB(\A)) =1$. This is a more natural and a more
uniform normalization when working with Eisenstein series and
constant terms.

\subsubsection{}
Because the group $\bU$ is unipotent, the measure $d_{\textup{Haar}}(u)$
normalized as in \eqref{eq:165} equals the 
the Tamagawa measure
\[
  d_{\textup{Haar}}(u) = d_{\tama}(u) \,. 
\]
Recall that e.g.\ for the additive
  group of a number
  field we have
  \[
    \A / \F = \prod_{\textup{finite $p$}} \cO_p \times \left(\left. 
    \prod_{p| \infty} F_p \right/ \cO_\F \right)
  \]
  and the Tamagawa measure of $\A / \F$ equals the square root of the ratio of
  the absolute value of discriminant (from the infinite places) and norms of the
  differents (from the finite places), hence $=1$. It is through the
  relation
  \[
    \prod_{\textup{finite $p$}} \vol(\cO_p) = \left|\textup{discr} \cO_F
  \right|^{-1/2} \,, 
  \]
  that the discriminant factor in completed zeta functions appears as
  product of local contribution in \eqref{eq:99}.

\subsubsection{} 

For $\bA$ factor in \eqref{eq:164}, we take the Tamagawa measure. Since $\bA \cong GL(1)^r$,
it suffices to discuss the multiplicative group $GL(1)$. Since
\[
  H^0(GL(1,\overline{F}),\cO^\times)/\overline{F}^\times = \Z \,, \quad
H^1(GL(1,\overline{F}),\cO^\times) = 0\,,
\]
where $\Z$ has the trivial Galois group action, we have \cite{Ono} 
\[
  d_{GL(1),\tama}(x)= \prod_{\textup{finite $p$}}
\frac{1}{1- p^{-1}}  \frac{dx}{\,\,\, |x|_p} \times 
    \prod_{p| \infty} \frac{dx}{\,\, \,|x|_p} \,. 
    \]

\subsubsection{} 

Since we consider  only unramified characters of $\bA$, we
effectively  work with the pushforward under the norm
map
\begin{align}
d_{\cQ,\tama} &= \left(GL(1,\A) \xrightarrow{\quad \|\, \cdot \, \|
    \quad} \cQ \right)_* d_{GL(1),\tama} \notag \\ 
  &= \big( \Res_{s=1} \xi(s) ds \big)  \, \frac{dq}{q} \,. \label{eq:162}
\end{align}
Here the second line follows the presentation of $\xi(s)$ as an
integral with respect to $GL(1,\A)$ as in Tate's thesis \cite{Tate}. In the
case when $\cQ$ is discrete, the measure $\frac{dq}{q}$ is the counting
measure.

\subsubsection{}

By the formula \eqref{eq:108}, we have
 \begin{equation}
  \Res_{s=1} \xi(s) ds  = \begin{cases}
  \frac{1}{\ln q} \, \cz \,,
  &\chr \F > 0 \\ 
 \cz\,, & \chr \F = 0 \,, 
\end{cases}\label{eq:108_res}
\end{equation}
where the constant $\cz$ was defined in \eqref{eq:23}. 
Thus the natural measure on the dual group 
\[
  \cQ^\vee = \C \left/ \frac{2\pi i}{\ln q} \, \Z \right. \,, 
\]
where $q=1$ in the number field case in \eqref{eq:116}, is
\begin{equation}
d_{\cQ^\vee,\tama} (s) = \frac{1}{2\pi i \cz} 
ds \,, \label{eq:163}
\end{equation}
in all cases. Note, in particular, this is correctly normalized in the
$q\to 1$ limit.

This is the how the Haar measure \eqref{eq:3} should be normalized
to reproduce actual $L^2$-norms in \eqref{eq:99}.

\subsubsection{}

In \cite{Lvol}, Langlands compares the above normalization of the
Haar measure on $\bG(\A)$ with the Tamagawa measure for $\bG(\A)$.
This amounts to computing the Tamagawa volumes of the flag
manifolds.

In the number field case, for the nonarchimedian places,
point counts in
flag manifolds give 
\[
  \prod_{p \, \nmid \, \infty} \vol_{\tama} (\bG(\F_p)/\bB(\F_p)) = \left|\textup{discr} \cO_F
\right|^{-\frac12 \dim \bG/\bB} \prod_{i=1}^r \zeta_\F(m_i+1)^{-1}\,,
\]
where
$m_i$ are the exponents of $\bG$, with a similar answer in the
function field case. One should bear in mind that $H^1(\bG/\bB,
\cO^\times)=\Z^r$ with a trivial Galois action, which multiplies 
the point counts by $(1-p^{-1})^r$.

The Archimedean contribution, concretely the
volume of the flag manifold computed on the big Bruhat cell,
Langlands interprets as a particular value of the Gindikin-Karpelevich
formula. Altogether, the proportionality constant is the residue
at the regular nilpotent discussed in Section \ref{s_ex_reg}. Since
the residue of the Eisenstein series at that point is a constant
function on $\bG$, Langlands computes the Tamagawa volume
of $\bG(\F) \backslash \bG(\A)$.

\end{document}